\documentclass[10pt,a4paper, reqno]{amsart}
\usepackage{amsmath,amsthm,amssymb,mathrsfs,stmaryrd,color}
\usepackage[all]{xy}
\usepackage{url}
\usepackage{slashed}
\usepackage{geometry}

\usepackage[utf8]{inputenc}
\usepackage[T1]{fontenc}

\usepackage{relsize}
\usepackage[bbgreekl]{mathbbol}
\usepackage{amsfonts}

\usepackage{tikz-cd}
\usepackage{enumerate}

\DeclareMathOperator{\K}{K}
\DeclareMathOperator{\THH}{THH}
\DeclareMathOperator{\TF}{TF}
\DeclareMathOperator{\TR}{TR}
\DeclareMathOperator{\TP}{TP}
\DeclareMathOperator{\TC}{TC}
\DeclareMathOperator{\HH}{HH}
\DeclareMathOperator{\HP}{HP}
\DeclareMathOperator{\HC}{HC}

\DeclareMathOperator{\Sp}{Sp}

\DeclareMathOperator{\Cycsp}{CycSp}
\DeclareMathOperator{\TCart}{TCart}

\DeclareMathOperator{\Calgc}{CAlg^{\geq 0}}

\DeclareMathOperator{\spf}{spf}

\DeclareMathOperator{\map}{map}

\DeclareMathOperator{\fib}{fib}

\DeclareMathOperator{\eq}{eq}
\DeclareMathOperator{\B}{B}
\DeclareMathOperator{\Fil}{Fil}
\DeclareMathOperator{\gr}{gr}

\DeclareMathOperator{\F}{F}
\DeclareMathOperator{\V}{V}

\DeclareMathOperator{\W}{W}

\DeclareMathOperator{\wcart}{WCart}

\DeclareMathOperator{\dr}{dR}

\newcommand{\FilM}{\Fil_{\mathcal{M}}}

\newcommand{\grM}{\gr_{\mathcal{M}}}

\newcommand{\N}{\mathbb{N}}
\newcommand{\Z}{\mathbb{Z}}

\newcommand{\Sph}{\mathbb{S}}

\newcommand{\Del}{\mathbb{\Delta}}
\newcommand{\Delc}{\widehat{\Del}}

\newcommand{\Fp}{\mathbb{F}_p}
\newcommand{\Zp}{\mathbb{Z}_p}

\theoremstyle{plain}
\newtheorem{thm}{Theorem}[section]
\newtheorem*{definition*}{Definition}
\newtheorem{lem}[thm]{Lemma}
\newtheorem{prop}[thm]{Proposition}
\newtheorem{cor}[thm]{Corollary}
\newtheorem{prin}[thm]{General Principle}

\theoremstyle{definition}
\newtheorem{defn}[thm]{Definition}
\newtheorem{rmk}[thm]{Remark}
\newtheorem{con}[thm]{Construction}
\newtheorem{ex}[thm]{Example}

\DeclareMathOperator{\ainf}{\mathbb{A}_{\inf}}
\DeclareMathOperator{\einfty}{\mathbb{E}_{\infty}}

\usepackage{hyperref}

\usepackage[
backend=biber,
style=alphabetic,
sorting=nyt
]{biblatex}
\DeclareMathOperator{\Spa}{Spa}
\DeclareMathOperator{\Res}{R}

\usepackage{hyperref}
\DeclareMathOperator{\can}{can}
\DeclareMathOperator{\Nm}{Nm}
\DeclareMathOperator{\CycspFr}{\Cycsp^{Fr}}
\DeclareMathOperator{\id}{id}

\usepackage{xcolor}

\DeclareMathOperator{\acrys}{\widehat{\mathbb{A}}_{\mathrm{crys}}}
\DeclareMathOperator{\rlim}{Rlim}

\DeclareMathOperator{\limr}{\underset{\Res}{\lim}}
\DeclareMathOperator{\rlimr}{R\underset{\Res}{\lim}}

\newcommand{\nr}{\mathcal{N}_r}
\newcommand{\ninf}{\mathcal{N}_{\infty}}
\newcommand{\n}{\mathcal{N}}
\newcommand{\HT}{\mathrm{HT}}
\newcommand{\conj}{\mathrm{conj}}
\newcommand{\hod}{\mathrm{Hod}}
\newcommand{\drw}{\mathrm{dRW}}

\usepackage{marginnote}
\usepackage[textsize=small]{todonotes}
\title[TR and the $r$-Nygaard filtered prismatic cohomology]{TR and the $r$-Nygaard filtered prismatic cohomology}
\author{Faidon Andriopoulos}
\thanks{
\textit{Email address:} \texttt{fandri@uchicago.edu}}

\begin{document}

\maketitle

\begin{abstract}
    Given an animated ring $S$, we define a filtration on its absolute prismatic cohomology $\nr^{\geq i} \Del_S$, which we call the \emph{$r$-Nygaard filtration} and study some of its main properties using a mixture of algebraic and homotopy theoretic techniques. This filtration is obtained by suitably gluing $r$-copies of the usual Nygaard filtration and corresponds to the $\xi_r$-adic filtration on $\ainf$, in the case that $S$ is a perfectoid ring.
    Using this, we study the motivic filtration of topological restriction homology $\TR^r (S;\Zp)$ and of its $S^1$-homotopy fixed points. We also pursue connections with the theory of topological cyclic homology. Finally we discuss connections with the de Rham--Witt complex, towards a prismatic - de Rham--Witt comparison theorem.
\end{abstract}

\tableofcontents

\section{Introduction}

Prismatic cohomology is a recently introduced $p$-adic cohomology theory due to B. Bhatt and P. Scholze. The underlying theory is developed in their groundbreaking work \cite{BS} and relies on the notion of a \emph{prism} $(A,I)$. Given a $p$-complete ring $S$ over $\overline{A}:=A/I$, one is able to construct its \emph{relative prismatic cohomology} $\Del_{S/A}$, via the use of the prismatic site. This theory has already proved invaluable for tackling several questions in a plethora of mathematical areas. This is a consequence of the fact that it is believed to be the \emph{right} $p$-adic cohomology theory, as it specializes to and refines all other previously known ones: de Rham, Hodge--Tate, crystalline, \' etale, etc.

A first attempt to construct a good universal $p$-adic cohomology theory was made in the initial work of Bhatt--Morrow--Scholze \cite{BMS1}. The authors mainly used techniques from perfectoid geometry, in order to construct cohomological invariants over Fontaine's period ring $\ainf$. These constructions were closely related to the theory of \emph{Breuil--Kisin--Fargues modules} and, via Fargues' theorem \cite{ScholzeWeinstein}, to \emph{mixed characteristic shtukas} with one leg over $\mathrm{spa}\; C^{\flat}$. Thus, this attempt falls under the general scope of Scholze's insight, who suggested that $p$-adic cohomology theories are expected to have shtuka-like properties \cite{ScholzeICM}.

Further algebro-geometric approaches are those of Drinfeld \cite{Prismatization} and Bhatt--Lurie \cite{APC, APC2, APC3}, who use the language of stacks in an essential way, in order to study prismatic invariants. Their viewpoint is based on stacky reformulations of the prismatic site of Bhatt--Scholze, called \emph{prismatization stacks}, which shed new light on geometric phenomena of \emph{absolute prismatic cohomology} and especially to structure associated with it, such as the \emph{Hodge--Tate cohomology}, the \emph{Nygaard filtration}, and the categories of \emph{prismatic crystals/gauges}.

An approach to absolute prismatic cohomology, of different flavour, lies in the second work of Bhatt--Morrow--Scholze \cite{BMS2}, which uses tools from homotopy theory. Interestingly enough, this story is opposite to the stacky one, as shown in the study of the \emph{filtered prismatization stacks} of \cite{Prismatization} and \cite{APC3}. In \cite{BMS2}, the authors constructed and studied \emph{motivic filtrations} of \emph{topological Hochschild homology} $\THH$ and of some related invariants, in the case of $p$-complete rings. In particular, the associated graded pieces of these motivic filtrations can be expressed in terms of structure related to prismatic cohomology. The inspiration for this work came from the theory of motives, in which the algebraic $\K$-theory of a well-behaved scheme carries a motivic filtration, whose graded pieces are identified with \textit{motivic cohomology}. For a subsequent approach to this, the reader is directed to \cite{elmanto2023motivic, bouis2024motivic}. In fact, one of the many applications of \cite{BMS2}, was the identification of syntomic cohomology with $p$-adic \'etale motivic cohomology, as a consequence of the approximation of algebraic $\K$-theory by topological cyclic homology, via the trace map \cite{CMM}.

The classical approach to topological cyclic homology $\TC$ was paved via the use of another invariant of $\THH$, called \emph{topological restriction homology} $\TR$. This is a rich invariant which has a deep connection to algebraic $\K$-theory, but is also closely related to the Witt vector functor. In particular, it is equipped with a Frobenius endomorphism $\F : \TR \to \TR$, which gives rise to $\TC$, as the following homotopy fibre:
\begin{equation*}
    \TC \simeq \fib \Big( 1 - \F: \TR \longrightarrow \TR \Big)
\end{equation*}
In fact, there is more to say about the relation between $\TR$ and algebraic $\K$-theory. On one hand, $\TR$ rises through the so-called \emph{curves in $\K$-theory} \cite{hesselholt1996p, mccandless2021curves, betley2005cyclotomic}, while on the other hand, it is related via Goodwillie calculus to a closely related invariant, called cyclic $\K$-theory $\K^{\mathrm{cyc}}$ \cite{lindenstrauss2012taylor, blumberg2016, hesselholt2019arbeitsgemeinschaft, nikolaustopological, Nikolaustalk}.

Attempting to see how these different approaches to prismatic cohomology and motivic phenomena in the $p$-adic world relate to each other, one is led to several natural questions:
\begin{itemize}
    \item What is the role of the Witt vector functor in prismatic cohomology, given their importance in the theory of $\delta$-rings and the approach via prismatization?
    \vspace{0.05in}
    \item What is the motivic filtration of $\TR^r$ and in terms of what structure, related to prismatic cohomology, can we interpret its graded pieces with?
    \vspace{0.05in}
    \item How does the theory of the de Rham--Witt complex fit in the prismatic context?
    \vspace{0.05in}
    \item The prismatic theory stands as a variant of shtukas with one leg over $\Spa\, C^{\flat}$, under their correspondence with Breuil--Kisin--Fargues modules. How could one capture the information of shtukas with more than one leg? In general, we would like to think of prismatic cohomology as illuminating an automorphic viewpoint to cohomology theories.
\end{itemize}

In this work, using both algebraic and homotopy theoretic techniques, we study two problems, which are interlinked to each other. On the one hand, we study the motivic filtration of topological restriction homology $\TR^r (S;\Zp)$ and of its $S^1$-homotopy fixed points $\TR^r (S;\Zp)^{hS^1}$. On the other hand, we introduce a filtration $\nr^{\geq \bullet} \Del_S$, for every $1 \leq r \leq \infty$, by suitably gluing $r$ copies of the Nygaard filtration. It follows that we can express the graded pieces of the motivic filtrations of $\TR^r (S;\Zp)^{hS^1}$ and $\TR^r (S;\Zp)$, through $\nr^{\geq \bullet} \Del_S$. We also discuss connections with the theory of the de Rham--Witt complex, which we hope to illuminate further, in future work.

Our main theorems read as follows:

\begin{thm}[Absolute Prismatic Cohomology] \label{thm1}
    Let $S$ be an animated ring. For $1 \leq r \leq \infty$, we introduce, via the following iterated pullback construction, a filtration on its absolute prismatic cohomology $\Del_S$, which we call the \emph{$r$-Nygaard filtration}:
    \begin{equation} \label{iterated pullback construction}
        \nr^{\geq i} \Del_S \{ i\} := \n^{\geq i} \Del_S \{ i\} \times_{\Del_S \{ i\}} \dots \times_{\Del_S \{ i\}} \n^{\geq i} \Del_S \{ i\}
    \end{equation}
    In this iterated pullback construction, the maps on the left are the canonical inclusion maps $\iota : \n^{\geq i} \Del_S \{ i\} \hookrightarrow \Del_S \{ i\}$, while the maps on the right correspond to the divided prismatic Frobenius $\varphi_i : \n^{\geq i} \Del_S \to \Del_S$.

    The $r$-th iteration of prismatic Frobenius takes the $i$-th filtered piece of the $r$-Nygaard filtration $\nr^{\geq i} \Del_S$ to $I_r^i \Del_S$. Mapping from the iterated pullback construction to the left, first to $\n^{\geq i}\Del_S \{ i\}$ via the natural projection and then to $\Del_S \{ i\}$ via $\varphi_i$, gives rise to a divided version of the $r$-th iteration of prismatic Frobenius, which we call the \emph{$r$-divided prismatic Frobenius} map. Mapping to the right, first to $\n^{\geq i} \Del_S \{ i\}$ via projection and then to $\Del_S \{ i\}$ via the canonical inclusion, gives an inclusion map. We denote these two maps as follows:
    \begin{equation*}
        \begin{tikzcd}
            \varphi_{r, i} : \nr^{\geq i} \Del_S \{ i\} \arrow[r] & \Del_S \{ i\} \qquad  \iota : \nr^{\geq i} \Del_S \{ i\} \arrow[r, hook] & \Del_S \{ i\}
        \end{tikzcd}
    \end{equation*}
    
    For $1 \leq r < \infty$, we denote by $\Del_S^{\HT,r}$ the \emph{$r$-Hodge--Tate cohomology}, which is obtained by the vanishing of $I_r$. It is the target of the graded couterpart of $\gr \varphi_{r,i}$, from the $i$-th associated graded piece of the $r$-Nygaard filtration $\nr^i \Del_S \{ i\}$, and therefore fits in the following commutative square.
    \begin{equation} \label{r-divided Frobenius commutative square}
        \begin{tikzcd}[column sep=huge]
            \nr^{\geq i} \Del_S \{ i\} \arrow[r, "\varphi_{r, i}"] \arrow[d] & \Del_S \{ i\} \arrow[d] \\
            \nr^i \Del_S \{ i\} \arrow[r, "\gr \varphi_{r, i}"] & \Del_S^{\HT, r} \{ i\}
        \end{tikzcd}
    \end{equation}
    
    The $r$-Nygaard filtered prismatic cohomology is equipped with natural Restriction, Frobenius, and Verschiebung maps, which also pass to the the graded pieces of the filtration, as indicated in the following diagrams:
    \begin{equation*}
        \begin{tikzcd}
            \n_{r+1}^{\geq i} \Del_S \{ i\} \arrow[r, "\Res"] \arrow[d] & \nr^{\geq i} \Del_S \{ i\} \arrow[d] & \n_{r+1}^{\geq i} \Del_S \{ i\} \arrow[r, "\F"] \arrow[d] & \nr^{\geq i} \Del_S \{ i\} \arrow[d] & \nr^{\geq i} \Del_S \{ i\} \arrow[r, "\V"] \arrow[d] & \n_{r+1}^{\geq i} \Del_S \{ i\} \arrow[d] \\
            \n_{r+1}^i \Del_S \{ i\} \arrow[r, "\Res"] & \nr^i \Del_S \{ i\} & \n_{r+1}^i \Del_S \{ i\} \arrow[r, "\F"] & \nr^i \Del_S \{ i\} & \nr^i \Del_S \{ i\} \arrow[r, "\V"] & \n_{r+1}^i \Del_S \{ i\}
        \end{tikzcd}
    \end{equation*}
    The Restriction and Frobenius maps interact with the $r$-divided prismatic Frobenius as shown in the following diagrams, the left of which is a pullback reformulation of the initial iterated pullback construction \ref{iterated pullback construction}:
    \begin{equation*}
        \begin{tikzcd}[column sep=huge]
            \n_{r+1}^{\geq i} \Del_S \{ i\} \arrow[r, "\Res"] \arrow[d] & \nr^{\geq i} \Del_S \{ i\} \arrow[d, "\varphi_{r,i}"] & \n_{r+1}^{\geq i} \Del_S \{ i\} \arrow[r, "\F"] \arrow[d] \arrow[dd, bend right=60, "\varphi_{r+1, i}"'] & \nr^{\geq i} \Del_S \{ i\} \arrow[d] \arrow[dd, bend left=60, "\varphi_{r,i}"] \\
            \n^{\geq i} \Del_S \{ i\} \arrow[r, "\iota"] & \Del_S \{ i\} & \n^{\geq i} \Del_S \{ i\} \arrow[r] \arrow[d] & \n^{\geq i} \Del_S \{ i\} \arrow[d] \\
            & & \Del_S \{ i\} \arrow[r] & \Del_S \{ i\}
        \end{tikzcd}
    \end{equation*}
    
    Taking the limit with respect to the Restriction maps, we obtain Frobenius endofunctors for the case $r=\infty$:
    \begin{equation*}
        \begin{tikzcd}[column sep=huge]
            \ninf^{\geq i} \Del_S \{ i\} \arrow[r, "\F"] \arrow[d] & \ninf^{\geq i} \Del_S \{ i\} \arrow[d] \\
            \ninf^i \Del_S \{ i\} \arrow[r, "\F"] & \ninf^{\geq i} \Del_S \{ i\}
        \end{tikzcd}
    \end{equation*}
    we obtain a Frobenius endofunctor
    \begin{equation*}
        \F : \ninf^{\geq i} \Del_S \{ i\} \longrightarrow \ninf^{\geq i} \Del_S \{ i\}
    \end{equation*}

    By passing to the $0$-th associated graded piece of the $r$-Nygaard filtered prismatic cohomology, one is able to recover the $r$-truncated Witt vectors:
    \begin{equation} \label{Witt vectors via gr0 of the r-Nygaard filtration}
        \nr^0 \Del_S \{ i\} \simeq \W_r (S)
    \end{equation}    
    Notice that this equivalence respects the natural symmetries of Restriction, Frobenius, and Verschiebung, which we explained above.

    In certain cases, when the input is sufficiently nice, it happens that we can have a simpler desription of the $r$-Nygaard filtration. More specifically, if $S$ is a finitely generated polynomial algebra over $\Z$, then the $r$-Nygaard filtration on its completed prismatic cohomology $\nr^{\geq \bullet} \Delc_S \{ \bullet\}$ can be described as the connective cover, with respect to the Beilinson t-structure, of the $I_r$-adic filtration $I_r^{\bullet} \Delc_S \{ \bullet\}$, under the $r$-th iteration of filtered prismatic Frobenius.

    In particular, notice that in the case $S$ is a perfectoid ring, the $r$-Nygaard filtration on its prismatic cohomology is just the $\xi_r$-adic filtration on $\ainf (S)$. The natural symmetries $\Res$, $\F$, $\V$ interact well with the Fontaine-style maps $\vartheta_r$, $\widetilde{\vartheta}_r$, as explained in \cite[Sec. 3]{BMS1}.
\end{thm}

\begin{rmk}
    Note that in this article we call \emph{prismatic Frobenius}, denoted by $\varphi$, the map that is related to the $\delta$-structure/lift of Frobenius structure on prismatic cohomology. On the other hand, we simply call \emph{Frobenius}, denoted by $\F$, the associated Witt vector Frobenius and/or the map that arises in the study of the $r$-Nygaard filtration on prismatic cohomology. Of course, the two are related, as explained in the theorem above. A direct relation is obtained in the case that the input ring $R_0$ is perfectoid, where we have that $\varphi = \lim_r \F$, coming from the equivalence $\ainf (R_0) \simeq \lim_{\F} \W_r (R_0)$.

    In a similar fashion, we use the term \emph{higher Frobenius} for the map $\THH \to \THH^{tC_p}$ introduced in \cite{NS} (or the one induced by taking $S^1$-homotopy fixed points: $\TC^- \to \TP$), while we reserve the term \emph{Frobenius}, for the related structure of $\TR^r$.
\end{rmk}

Given the relation we obtain between prismatic cohomology $\Del_S$ and the Witt vectors $\W_r (S)$, together with the prismatic-de Rham comparison theorem, it would be natural to ask whether one could lift the latter, to obtain possible connections with the $r$-truncated absolute de Rham--Witt complex $\drw_{r, S}$. Assuming that this is possible, it would imply the existence of a commutative diagram, analogous to the one found in \cite[Sec. 1.8]{APC}, involving the absolute de Rham--Witt complex:
\begin{equation*} \label{Bhatt-Lurie thread with the de Rham--Witt complex}
    \begin{tikzcd}[column sep=huge]
        \Fil_{\hod}^i \drw_{r, S} \arrow[d] & \nr^{\geq i} \Del_S \{ i\} \arrow[l] \arrow[r, "\varphi_{r, i}"] \arrow[d] & \Del_S \{ i\} \arrow[d] \\
        \gr_{\hod}^i \drw_{r, S} \arrow[d, "\simeq"] & \nr^i \Del_S \{ i\} \arrow[l] \arrow[r, "\gr \varphi_{r, i}"] \arrow[d] & \Del_S^{\HT, r} \{ i\} \arrow[d] \\
        LW_r\widehat{\Omega}_S^i [-i] & ? \arrow[l] \arrow[r] & ?
    \end{tikzcd}
\end{equation*}

The questionmarks in the two spots above should be filled by some object that behaves in a similar way with the diffracted Hodge complex. In order for such a construction to be possible, one should consider geometric objects such as an $r$-Hodge--Tate stack $\Sigma^{\HT, r}$ (which we briefly do in this present article) and versions of a Sen operator/diffracted Hodge--Witt complex (which we hope to address in future work).

In order to be able to understand this setup regarding the absolute prismatic cohomology, some understanding of the relative version should be at hand. We discuss this in the following theorem:

\begin{thm}[Relative Prismatic Cohomology] \label{thm2}
    Let $(A,I)$ be a bounded prism and consider a $p$-complete, animated ring $S$, over $\overline{A}$. The structure that we described for the case of absolute prismatic cohomology passes to the case of relative prismatic cohomology, as well. Additionally, for $1 \leq r < \infty$, the $r$-Hodge--Tate cohomology is equipped with a multiplicative, increasing, exhaustive \emph{conjugate filtration} $\Fil_i^{\conj} \Del_{S/ A}^{\HT, r}$, whose $i$-th filtered term fits in the following commutative square:
    \begin{equation*} \label{relative r-divided Frobenius square and the conjugate filtration}
        \begin{tikzcd}[column sep=huge]
            \nr^{\geq i} (\varphi_A^r)^* \Del_{S/A} \{ i\} \arrow[rr] \arrow[d] & & \Del_{S/A} \{ i\} \arrow[d] \\
            \nr^i (\varphi_A^r)^* \Del_{S/A} \{ i\} \arrow[rr] \arrow[rd, bend right=15, "\simeq" description] & & \Del_{S/A}^{\HT, r} \{ i\} \\
            & \Fil_i^{\conj} \Del_{S/A}^{\HT, r} \{ i\} \arrow[ru, hook, bend right=15]
        \end{tikzcd}
    \end{equation*}

    If $S$ is $p$-completely smooth over $\overline{A}$, then the $r$-th iteration of the relative prismatic Frobenius on $\Del_{S/A}$ factors isomorphically through the d\'ecalage with respect to $I_r$:
    \begin{equation*}
        \begin{tikzcd}[column sep=huge]
            (\varphi_A^r)^* \Del_{S/A} \arrow[r, "\simeq"] & L \eta_{I_r} \Del_{S/A} \arrow[r] & \Del_{S/A}
        \end{tikzcd}
    \end{equation*}
    In particular, the relative $r$-Nygaard filtration $\nr^{\geq i} (\varphi_A^r)^* \Del_{S/A} \{ i\}$ can be described as the connective cover, with respect to the Beilinson t-structure, of the relative $I_r$-adic filtration $I_r^{\bullet} \Del_{S/A} \{ \bullet\}$, under the $r$-th iteration of the relative prismatic Frobenius. It follows that the conjugate filtration on the relative $r$-Hodge--Tate cohomology identifies with the Postnikov one:
    \begin{equation}
        \begin{tikzcd}
            \nr^i (\varphi^r)^* \Del_{S/A} \arrow[r, "\simeq"] & \Fil_i^{\conj} \Del_{S/A}^{\HT, r} \simeq \tau^{\leq i} \Del_{S/A}^{\HT, r}
        \end{tikzcd}
    \end{equation}

    If we further assume that $(A,I)$ is a perfect prism, then the main result of \cite{molokov2020prismatic} ensures that the $i$-th associated graded pieces of the conjugate filtration can be described in terms of the $p$-complete, relative, $i$-th de Rham--Witt forms of Langer--Zink:
    \begin{equation}
        \gr_i^{\conj} \Del_{S/A}^{\HT, r} \{ i\} \simeq LW_r\Omega_{S/ \overline{A}}^{i, \mathrm{cont}} [-i]
    \end{equation}
    
    In this case, it follows that we can rephrase part of the above in terms of a relative de Rham--Witt comparison isomorphism, between the $r$-Nygaard filtered relative prismatic cohomology and the Hodge filtered relative $r$-truncated de Rham--Witt complex of Langer--Zink:
    \begin{equation}
        A/I_r \widehat{\otimes}_A^{\mathbb{L}} (\varphi_A^r)^* \Del_{S/A} \{ i\} \simeq \drw_{r, S/ \overline{A}}
    \end{equation}
\end{thm}

Generalizing from the case of the Nygaard filtration $(r=1)$, to the case of the $r$-Nygaard filtration $(r \geq 1)$, we are inclined to formulate a general principle, along the following lines:
\begin{prin} \label{General principle}
    For every statement in the prismatic theory, which is true regarding the tuple:
    \begin{equation*}
        \Big(\n^{\geq \bullet} \Del, I^{\bullet}, \varphi, \Fil_{\hod}^{\bullet} \dr, \FilM^{\bullet} \TC^- \Big)
    \end{equation*}
    there should be a roughly analogous one for the tuple:
    \begin{equation*}
        \Big(\nr^{\geq \bullet} \Del, I_r^{\bullet}, \varphi^r, \Fil_{\hod}^{\bullet} \drw_r, \FilM^{\bullet} (\TR^r)^{hS^1} \Big)
    \end{equation*}
\end{prin}

Moving to the realm of homotopy theory, we know that $\big(\n^{\geq \bullet} \Del, I^{\bullet}, \varphi, \Fil_{\hod}^{\bullet} \dr \big)$ arises through the study of topological Hochschild homology $\THH$ and its $S^1$-homotopy fixed points, as shown in the work of Bhatt--Morrow--Scholze \cite{BMS2}. In analogy, one can read the data of the graded/filtered pieces of the $r$-Nygaard filtration by studying the motivic filtration of topological restriction homology $\TR^r$ and of its $S^1$-fixed points, respectively. The theorem reads as follows:

\begin{thm} [The motivic filtration of $\TR^r$] \label{thm3}
    Let $S$ be a quasisyntomic ring. The following hold:
    \begin{enumerate}[1)]
    \item For $1 \leq r \leq \infty$, the invariants $\TR^r (S;\Zp)^{hS^1} \to \TR^r (S;\Zp)$ are equipped with complete, exhaustive, descending, multiplicative, $\Z$-indexed \emph{motivic filtrations}:
    \begin{equation*}
        \FilM^{\geq \bullet} \TR^r (S;\Zp)^{hS^1} \longrightarrow \FilM^{\geq \bullet} \TR^r (S;\Zp)
    \end{equation*}
    which are the quasisyntomic sheafifications of their respective double speed Postnikov filtrations. Passing to their associated graded pieces, these can be identified with:
    \begin{equation*}
        \begin{cases}
            \grM^i \TR^r (S;\Zp)^{hS^1} \simeq R\Gamma_{syn} \Big(S, \tau_{[2i-1,2i]} \TR^r (-;\Zp)^{hS^1}\Big) \\[5pt]
        \grM^i \TR^r (S;\Zp) \simeq R\Gamma_{syn} \Big(S, \tau_{[2i-1,2i]} \TR^r (-;\Zp)\Big)
        \end{cases}
    \end{equation*}
    by applying quasisyntomic descent from the quasiregular-semiperfectoid case, in which both are identified with two-term complexes.
    \vspace{0.05in}
    
    \item Let us denote by $\grM^{i, \mathrm{odd}}$ the corresponding quasisyntomic sheafification of odd homotopy groups $\pi_{2i-1}$ and by $\grM^{i,\mathrm{even}}$ the corresponding quasisyntomic sheafification of even homotopy groups $\pi_{2i}$, of either $\TR^r (S;\Zp)^{hS^1}$ or $\TR^r (S;\Zp)$. Then $\grM^{i,\mathrm{even}}$ is the $0^{\text{th}}$ cohomology group of the two-term complex, while $\grM^{i, \mathrm{odd}}$ is the $1^{\text{st}}$ cohomology group. For these, the following identifications hold:

    For $1 \leq r < \infty$, the even parts can be expressed in terms of the $r$-Nygaard filtration on Nygaard completed prismatic cohomology $\Delc_S$:
    \begin{equation*}
        \begin{cases}
            \grM^{i,\mathrm{even}} \TR^r (S;\Zp)^{hS^1} \simeq \nr^{\geq i} \Delc_S \{ i\} [2i] \\[5pt]
            \grM^{i,\mathrm{even}} \TR^r (S;\Zp) \simeq \nr^i \Delc_S \{ i\} [2i]
        \end{cases}
    \end{equation*}
    On the other hand, the odd parts $\grM^{i,\mathrm{odd}} \TR^r (S;\Zp)^{hS^1}$ and $\grM^{i,\mathrm{odd}} \TR^r (S;\Zp)$, which correspond to the odd homotopy groups, locally vanish for the quasisyntomic topology.

    Taking the limit over Restriction maps, we pass to the case of $r=\infty$, in which we have the following identifications:
    \begin{equation*}
        \begin{cases}
            \grM^{i,\mathrm{even}} \TR (S;\Zp)^{hS^1} \simeq \limr \nr^{\geq i} \Delc_S \{ i\} [2i]
            \\[5pt]
            \grM^{i,\mathrm{even}} \TR (S;\Zp) \simeq \limr \nr^i \Delc_S \{ i\} [2i]
        \end{cases}
    \end{equation*}
    Locally in the quasisyntomic topology, we have that
    \begin{equation*}
        \begin{cases}
            \grM^i \TR (-;\Zp)^{hS^1} \simeq \ninf^{\geq i} \Delc_{(-)} \{ i\} [2i] := \rlimr \nr^{\geq i} \Delc_{(-)} \{ i\} [2i]
            \\[5pt]
            \grM^i \TR (-;\Zp) \simeq \ninf^i \Delc_{(-)} \{ i\} [2i] := \rlimr \nr^i \Delc_{(-)} \{ i\} [2i]
        \end{cases}
    \end{equation*}
    Thus, locally for the quasisyntomic topology, the odd parts correspond to the $\limr^1$ terms. For $1 \leq r \leq \infty$, the canonical map
    \begin{equation*}
        \TR^r (S;\Zp)^{hS^1} \longrightarrow \TR^r (S;\Zp)
    \end{equation*}
    gives rise to the map from the $i$-th filtered to the $i$-th graded piece for the $r$-Nygaard filtration:
    \begin{equation*}
        \nr^{\geq i} \Delc_S \{ i\} [2i] \longrightarrow \nr^i \Delc_S \{ i\} [2i]
    \end{equation*}
    If we restrict to the case $1 \leq r < \infty$, it follows that the invariants $\TR^r (-;\Zp)$ and $\TR^r (-;\Zp)^{hS^1}$ are locally even with respect to the quasisyntomic topology. Additionally, we can consider the commutative diagram involving higher Frobenii:
    \begin{equation}
        \begin{tikzcd}[column sep=huge]
            \TR^r (S;\Zp)^{hS^1} \arrow[r] \arrow[d] & \TC^- (S;\Zp) \arrow[r, "\varphi^{hS^1}"] \arrow[d] & \TP (S;\Zp) \arrow[d] \\
            \TR^r (S;\Zp) \arrow[r] & \THH (S;\Zp)^{hC_{p^{r-1}}} \arrow[r, "\varphi^{hC_{p^{r-1}}}"] & \THH (S;\Zp)^{tC_{p^r}}
        \end{tikzcd}
    \end{equation}
    gives rise to the commutative diagram on $r$-Nygaard filtered prismatic cohomology groups, involving the $r$-divided prismatic Frobenius, together with its graded version:
    \begin{equation}
        \begin{tikzcd}[column sep=huge]
            \nr^{\geq i} \Delc_S \{ i\} [2i] \arrow[r, "\varphi_{r, i}"] \arrow[d] & \Delc_S \{ i\} [2i] \arrow[d] \\
            \nr^i \Delc_S \{ i\} [2i] \arrow[r, "\gr \varphi_{r, i}"] & \Del_S^{\HT, r} \{ 2i\} [2i]
        \end{tikzcd}
    \end{equation}

    Also, the symmetries of Restriction, Frobenius, and Verschiebung on $\TR^r$ and of its $S^1$-homotopy fixed points, give rise to corresonding symmetries on the graded and filtered pieces for the $r$-Nygaard filtered $\Delc_S$.
    \vspace{0.05in}
    
    \item From the vanishing of odd homotopy groups, locally for the quasisyntomic topology, by applying quasisyntomic descent, we obtain multiplicative spectral sequences, for $1 \leq r \leq \infty$:
    \begin{equation*}
        \begin{cases}
            E_2^{ij} = H^{i-j} \Big( \nr^{\geq -j} \Delc_{(-)} \Big) \Rightarrow \pi_{-i-j} \TR^r (-;\Zp)^{hS^1} \\[5pt]
            E_2^{ij} = H^{i-j} \Big( \nr^{-j} \Delc_{(-)} \Big) \Rightarrow \pi_{-i-j} \TR^r (-;\Zp)
        \end{cases}
    \end{equation*}
    Thus, for $1 \leq r < \infty$ and locally in the quasisyntomic topology, we can identify the $r$-Nygaard filtration on $\nr^{\geq i} \Delc_S \{ i\} [2i]$ as the one coming from the degeneration of the $S^1$-homotopy fixed points spectral sequence.
    \vspace{0.05in}
    
    \item By left Kan extending, all statements can be extended to the case $S$ is a $p$-complete animated ring. One of the failures, as pointed out in \cite{APC}, is that the associated motivic filtrations are not exhaustive.
    \end{enumerate}
\end{thm}

Topological restriction homology was the invariant used in the first attempt to understand topological cyclic homology, via the formula:
\begin{gather*}
    \TC (-;\Zp) \simeq \fib \Big( 1-\F : \TR (-;\Zp) \longrightarrow \TR (-;\Zp) \Big) \\
    \simeq \limr \fib \Big( \Res - \F : \TR^{r+1} (-;\Zp) \longrightarrow \TR^r (-;\Zp) \Big) =: \limr \TC_r (-;\Zp)
\end{gather*}

However, $\TR$ is a cyclotomic spectrum (with Frobenius lifts) on its own, therefore it is natural to attempt to study $\TC \big( \TR \big)$, but also $\widetilde{\TC} \big( \TR \big)$, where $\widetilde{\TC} (-) := \map_{\CycspFr} (\Sph, - )$ is a version of topological cyclic homology for cyclotomic spectra with Frobenius lifts. The following theorem is concerned with these questions:

\begin{thm}[Topological cyclic homology of $\TR$] \label{thm4}
Applying the previous results to the setup of topological cyclic homology, the following are true:
\begin{enumerate}[1)]
    \item Topological cyclic homology for cyclotomic spectra with Frobenius lifts of $\TR (-;\Zp)$ is equivalent to:
    \begin{gather*}
        \widetilde{\TC} \Big( \TR (-;\Zp) \Big) \simeq \fib \Big( 1 - \F^{hS^1} : \TR (-;\Zp)^{hS^1} \longrightarrow \TR (-;\Zp)^{hS^1} \Big) \simeq \limr \widetilde{\TC}^r \Big( \TR (-;\Zp)\Big)
    \end{gather*}
    where we define the family of spectra $\widetilde{\TC}^r \Big( \TR (-;\Zp)\Big)$ to be:
    \begin{equation*}
        \widetilde{\TC}^r \Big( \TR (-;\Zp)\Big) := \fib \Big( \Res^{hS^1} - \F^{hS^1} : \TR^r (-;\Zp)^{hS^1} \longrightarrow \TR^{r-1} (-;\Zp)^{hS^1} \Big)
    \end{equation*}
    Mapping further to $\TP$, one obtains the following family of spectra, defined as:
    \begin{equation*}
        \TC^r \Big( \TR (-;\Zp) \Big) := \fib \Big( \can - \varphi^{hS^1} : \TR^r (-;\Zp)^{hS^1} \longrightarrow \TC^- (-;\Zp) \longrightarrow \TP (-;\Zp) \Big)
    \end{equation*}
    which "interpolates" between $\TC \big( \TR (-;\Zp) \big)$, for $r=\infty$, and $\TC (-;\Zp)$, for $r=1$:
    \begin{gather*}        
        \TC^r \Big( \TR (-;\Zp) \Big) \simeq \fib \Big( \Res^{hS^1} - \F^{hS^1} : \TR^r (-;\Zp)^{hS^1} \longrightarrow \TR^{r-1} (-;\Zp)^{hS^1} \longrightarrow \TP (-;\Zp) \Big) \\
        \simeq \fib \Big( \can - \varphi^{hS^1} : \TR^r (-;\Zp)^{hS^1} \longrightarrow \TC^- (-;\Zp) \longrightarrow \TP (-;\Zp) \Big)
    \end{gather*}
    \vspace{0.05in}
    
    \item For $1 \leq r \leq \infty$ and the input of  a quasisyntomic ring $S$, topological cyclic homology $\TC_r (-;\Zp)$, as well as the spectra of part (1) are equipped with exhaustive, decreasing, multiplicative, $\Z$-indexed motivic filtrations. Locally for the quasisyntomic topology, these can be expressed as follows:
    \begin{equation*}
        \begin{cases}
            \grM^i \TC_r (-;\Zp) \simeq \fib \Big( \Res - \F : \mathcal{N}_{r+1}^i \Delc_{(-)} \{ i\} [2i] \longrightarrow \nr^i \Delc_{(-)} \{ i\} [2i] \Big) \\[5pt]
            \grM^i \widetilde{\TC}^r \Big( \TR (-;\Zp) \Big) \simeq \fib \Big( \Res^{hS^1} - \F^{hS^1} : \nr^{\geq i} \Delc_{(-)} \{ i\} [2i] \longrightarrow \mathcal{N}_{r-1}^{\geq i} \Delc_{(-)} \{ i\} [2i] \Big) \\[5pt]
            \grM^i \TC^r \Big( \TR (-;\Zp) \Big) \simeq \fib \Big( \can - \varphi^{hS^1} : \nr^{\geq i} \Delc_{(-)} \{ i\} [2i] \longrightarrow \Delc_{(-)} \{ i\} [2i] \Big)
        \end{cases}
    \end{equation*}
    Locally for the quasisyntomic topology, these give rise to spectral sequences which degenerate.
    Of course, as in the main theorem, the statements above can also be suitably extended to the case we work with a $p$-complete animated ring.
\end{enumerate}
\end{thm}

Let us, finally, specialize to the cases of mixed and positive characteristic. In the former, we recover the relation to the objects $\widetilde{W_r\Omega}_S$ and $A\Omega \simeq \lim_r \widetilde{W_r\Omega}_S$, which were introduced and studied in \cite{BMS1}. In particular, the authors explain that they were able to build these complexes by studying $\TR^r$ of perfectoid rings. This was indeed the precursor of the prismatic theory.
\begin{thm}[Mixed characteristic] \label{thm5}
    Let $(A,I)$ be a perfect prism and consider S to be a $p$-completely smooth ring over the perfectoid ring $R_0 := \overline{A}$. Then, the $r$-Nygaard filtration can be expressed in terms of the d\'ecalage functor, associated with the element $\xi_r \in A$:
    \begin{equation*}
        \begin{tikzcd}[column sep=huge]
            \nr^{\geq i} \Del_S \arrow[r, "\simeq"] & \nr^{\geq i} (\varphi^r)^* \Del_{S/A} \arrow[r, "\simeq"] & L \eta_{\xi_r} ^{\geq i} \Del_{S/A}
        \end{tikzcd}
    \end{equation*}
    Let $R_0=\mathcal{O}_C$, where $C$ is a perfectoid field containing all $p$-roots of unity. Then we find ourselves in the setting of \cite{BMS1} and thus recover a relation to the $A\Omega$-cohomology. More specifically, we have that $\nr^{\geq i} \Del_S \simeq L\eta_{\xi_r}^{\geq i} A \Omega_S$. Passing to the associated graded pieces (see also the related discussion of \cite[Rem. 1.20]{BMS1} on the lift of the Cartier isomorphism in mixed characteristic), we have the following interpretation of the lower row of the square in Theorem \ref{thm2}:
    \begin{equation*}
        \begin{tikzcd}[column sep=huge]
            \nr^i A\Omega_S \arrow[rr] \arrow[dr, bend right=15, "\simeq" description] & & \widetilde{W_r\Omega}_S \\[-10pt]
            & \tau^{\leq i} \widetilde{W_r\Omega}_S \arrow[ur, bend right=15, hook] &
        \end{tikzcd}
    \end{equation*}
    This allows us to obtain the following expression, regarding the motivic filtration of $\TR^r$:
    \begin{equation*}
        \begin{tikzcd}
            \grM^{i, \mathrm{even}} \TR^r (S;\Zp) \simeq \nr^i A\Omega_S [2i] \arrow[r, "\simeq"] & \tau^{\leq i} \widetilde{W_r\Omega}_S \{ i\} [2i]
        \end{tikzcd}
    \end{equation*}
    Taking the limit with respect to the Frobenius maps, we recover $A\Omega$-cohomology via the motivic filtration of topological Frobenius homology:
    \begin{equation*}
        \grM^{i, \mathrm{even}} \TF (S;\Zp) \simeq \tau^{\leq i} A\Omega_S \{ i\} [2i]
    \end{equation*}
\end{thm}

Passing to the latter case of positive characteristic, we show that in the quasiregular-semiperfect setting, $\TR^r$ satisfies odd-vanishing. Therefore, things work as in the case of \cite{BMS2} and, ultimately, we obtain a relation with the classical de Rham--Witt complex in positive characteristic.

\begin{thm}[Positive Characteristic case] \label{thm6}
    If $S$ is a quasiregular-semiperfect $\Fp$-algebra, the odd homotopy groups of the following invariants vanish, for $1 \leq r \leq \infty$:
    \begin{equation*}
        \begin{cases}
            \pi_{\mathrm{odd}} \TR^r (S;\Zp)^{hS^1} \simeq 0 \\[5pt]
            \pi_{\mathrm{odd}} \TR^r (S;\Zp) \simeq 0
        \end{cases}
    \end{equation*}
    In particular, and in analogy with \cite{BMS2}, we can identify the $r$-Nygaard filtration on $\Delc_S \simeq \acrys (S)$ as the one coming from the $S^1$-homotopy fixed points spectral sequence, for $\TR^r (S;\Zp)^{hS^1}$, which degenerates. This was also recently discussed in \cite{darrell2023mathrm} and \cite{devalapurkar2023p}.

    If we assume that $S$ is a smooth $k$-algebra, where $k$ is a perfect field of characteristic $p$, Theorem \ref{thm5}, which is related to the mixed characteristic lift of the Cartier isomorphism, reduces to a statement about the classical de Rham--Witt complex in positive characteristic. Therefore, one has the following identification regarding the $r$-Nygaard filtration:
    \begin{equation*}
        \nr^{\geq i} (\varphi^r)^* \Del_{S/ W(k)} \simeq L \eta_{p^r} R\Gamma_{crys} (S)
    \end{equation*}
    It follows that:
    \begin{equation*}
        \grM^i \TR^r (S;\Zp) \simeq \tau^{\leq i} W_r\Omega_{S/k}^{\bullet} [2i]
    \end{equation*}
\end{thm}

\subsection*{Proof Outline}
In this paper, we first deal with the $r$-Nygaard filtration, as it arises on completed prismatic cohomology $\Delc_S$, via the homotopy theoretic machinery. This way we obtain the non-trivial relation with the Witt vectors: $\nr^0 \Delc_S \simeq \W_r (S)$. We then move on to describe the $r$-Nygaard filtration on the non-completed version of prismatic cohomology, via algebraic means.

Regarding the homotopy-theoretic viewpoint, the main theme of this work is, instead of directly dealing with $\TR^r (-;\Zp)$, to first study its $S^1$-homotopy fixed points $\TR^r (-;\Zp)^{hS^1}$, for $1 \leq r < \infty$. This is where the $r$-Nygaard filtration on completed prismatic cohomology arises from. Then, we use a trick of Nikolaus--Scholze \cite{AN}, where by taking quotient with the class generating $\pi_{-2} \TR^r (-;\Zp)^{hS^1}$, we are able to pass to $\TR^r (-;\Zp)$. Thus, this enables us to identify the relevant structure of the latter with the graded pieces for the $r$-Nygaard filtration. This is a pattern systematically used in \cite{BMS2}, where $\TC^- (-;\Zp)$ gives rise to the Nygaard filtration on completed prismatic cohomology, while passing to $\THH (-;\Zp)$ gives rise to its associated graded pieces.

As in \cite{BMS2}, the first step is to try and make precise calculations in the perfectoid case. This is, indeed, possible as for a perfectoid ring $R_0$, the following identification holds:
\begin{gather*}
    \pi_* \TR^r (R_0;\Zp)^{hS^1} \simeq \ainf (R_0) [u_r, v_r]/(u_r v_r - \xi_r) \\
    \mathrm{deg}(u_r) = 2, \mathrm{deg}(v_r) = -2, \mathrm{deg}(\xi_r) = 0
\end{gather*}
Remember that $\widetilde{\vartheta}_r : \ainf (R_0) \to \W_r (R_0)$ is the usual projection to the $r$-truncated Witt vectors, while $\vartheta_r := \widetilde{\vartheta}_r \circ \varphi^r : \ainf (R_0) \to \W_r (R_0)$ is twisted by the $r$-th iterated prismatic Frobenius. The kernel of the latter is generated by the element $\xi_r$, while of the former by $\widetilde{\xi}_r = \varphi^r (\xi_r)$.

Taking quotient with respect to $v_r \in \pi_{-2} \TR^r (R_0;\Zp)^{hS^1}$ takes us back to $\TR^r (R_0;\Zp)$, thus giving rise to the following equivalence:
\begin{gather*}
    \pi_* \TR^r (R_0;\Zp) \simeq \W_r (R_0) [u_r]
\end{gather*}

Equivalently, we can reformulate these as:
\begin{equation*}
    \begin{cases}
        \pi_{2i} \TR^r (R_0;\Zp)^{hS^1} \simeq \xi_r^i \ainf (R_0) \{ i\} \\[5pt]
        \pi_{2i} \TR^r (R_0;\Zp) \simeq \xi_r^i/ \xi_r^{i+1} \ainf (R_0) \{ i\}
    \end{cases}
\end{equation*}

By taking the limit over Restriction maps, we have the following identification for the two-term complexes:
\begin{equation*}
    \begin{cases}
        \tau_{[2i-1,2i]} \TR (R_0;\Zp)^{hS^1} \simeq \rlimr \xi_r^i \ainf (R_0) \{ i\} \\[5pt]
        \tau_{[2i-1,2i]} \TR (R_0;\Zp) \simeq \rlimr \xi_r^i/\xi_r^{i+1} \ainf (R_0) \{ i\}
    \end{cases}
\end{equation*}

For this, we use the following iterated pullback formula of the $S^1$-homotopy fixed points, for $1 \leq r \leq \infty$:
\begin{align*}
    \TR^r (-;\Zp)^{hS^1} & \simeq \TC^- (-;\Zp) \times_{\TP (-;\Zp)} \dots \times_{\TP (-;\Zp)} \TC^- (-;\Zp) \\
    & \simeq \fib \Bigg( \prod_{1 \leq k \leq r} \TC^- (-;\Zp) \longrightarrow \prod_{1 \leq k \leq r-1} \TP (-;\Zp) \Bigg) 
\end{align*}

Following \cite{BMS2}, we then pass to the study of quasiregular-semiperfectoid rings, which provide a basis for applying quasisyntomic descent. Suppose we have a quasirgular-semiperfectoid ring over a fixed perfectoid base $R_0 \to S$. Using the last presentation for $\TR^r (S;\Zp)^{hS^1}$, we have the following identification for the even homotopy groups, with the $r$-Nygaard filtered complete prismatic cohomology:
\begin{equation*}
    \pi_{2i} \TR^r (S;\Zp)^{hS^1} \simeq \nr^{\geq i} \Delc_S \{ i\}
\end{equation*}
The filtration $\nr^{\geq \bullet} \Delc_S$ is a decreasing, multiplicative, complete filtration on $\Delc_S$, with the following iterated pullback description:
\begin{align*}
    \nr^{\geq i} \Delc_S \{ i\} & := \n^{\geq i} \Delc_S \{ i\} \times_{\Delc_S \{ i\}} \dots \times_{\Delc_S \{ i\}} \n^{\geq i} \Delc_S \{ i\}
\end{align*}
Equivalently, we have the following more descriptive definition, in analogy with the classical Nygaard filtration:
\begin{equation*}    
    \nr^{\geq i} \Delc_S = \Bigg\{ x \in \Delc_S \; \Big| \; \varphi^{ri} (x) \in \widetilde{\xi}_r^i \Delc_S \Bigg\}
\end{equation*}

Passing back to $\TR^r (S;\Zp)$ via the Nikolaus-Scholze trick, we are able to identify its even homotopy groups with the associated graded pieces of the $r$-Nygaard filtration:
\begin{equation*}
    \pi_{2i} \TR (S;\Zp) \simeq \nr^i \Delc_S \{ i\}
\end{equation*}

It follows that the map $\varphi^{hS^1} : \TR^r (S;\Zp)^{hS^1} \to \TC^- (S;\Zp) \to \TP (S;\Zp)$ gives rise to an $r$-divided prismatic Frobenius:
\begin{equation*}
    \varphi_{r, i} : \nr^{\geq i} \Delc_S \{ i\} \longrightarrow \Delc_S \{ i\}
\end{equation*}
This is related to the $r$-th iterated prismatic Frobenius, via the formula: $\varphi_{r, i} = \varphi^{ri}/\widetilde{\xi}_r^i$

In order to move forward, as we already mentioned, we need to use the quasisyntomic topology. Under this scope, using the vanishing result of \cite{BS}, the odd homotopy groups
\begin{equation*}
    \pi_{2i-1} \TR^r (-;\Zp)^{hS^1}, \quad \pi_{2i-1} \TR^r (-;\Zp)
\end{equation*}
vanish locally in the quasisyntomic topology. Notice that for quasiregular semiperfect rings, this vanishing holds without requiring to pass to a suitable quasisyntomic cover.

It follows that, locally for the quasisyntomic topology, we can identify the graded pieces for the motivic filtrations (which come from the double-speed Postnikov filtrations) as:
\begin{equation*}
    \begin{cases}
        \grM^i \TR^r (-;\Zp)^{hS^1} \simeq \nr^{\geq i} \Delc_{(-)} \{ i\} [2i] \\[5pt]
        \grM^i \TR^r (-;\Zp) \simeq \nr^i \Delc_{(-)} \{ i\} [2i] \\[10pt]
        \grM^i \TR (-;\Zp)^{hS^1} \simeq \rlimr \nr^{\geq i} \Delc_{(-)} \{ i\} [2i] = \ninf^{\geq i} \Delc_{(-)} \{ i\} [2i] \\[5pt]
        \grM^i \TR (-;\Zp) \simeq \rlimr \nr^i \Delc_{(-)} \{ i\} [2i] = \ninf^i \Delc_{(-)} \{ i\} [2i]
    \end{cases}
\end{equation*}
The remaining results regarding $\TR$ and its related invariants follow from these identifications.

In analogy with Nygaard-complete prismatic cohomology, which is the one obtained from the homotopy theoretic story, we are able to have similar results for the non-completed version. Given a $p$-complete animated ring, we can define the $r$-Nygaard filtration via the iterated construction:
\begin{equation*}
    \nr^{\geq i} \Del_S \{ i\} := \n^{\geq i} \Del_S \{ i\} \times_{\Del_S \{ i\}} \dots \times_{\Del_S \{ i\}} \n^{\geq i} \Del_S \{ i\}
\end{equation*}
Mapping to $\Del_S \{ i\}$, from the left, gives a divided version of the $r$-th iteration of prismatic Frobenius. This fits in the commutative diagram where the upper row corresponds to the filtered invariant, while the lower to the graded versions:
\begin{equation*}
    \begin{tikzcd}[column sep=huge]
        \nr^{\geq i} \Del_S \{ i\} \arrow[r, "\varphi_{r, i}"] \arrow[d] & \Del_S \{ i\} \arrow[d] \\
        \nr^i \Del_S \{ i\} \arrow[r, "\gr \varphi_{r, i}"] & \Del_S^{\HT, r} \{ i\}
    \end{tikzcd}
\end{equation*}

There is an analogous square in the setting of relative prismatic cohomology. Using generalities regarding the d\'ecalage functor $L\eta_{I_r}$ and the Beilinson t-structure, one is able to construct a conjugate filtration on the relative $r$-Hodge--Tate cohomology $\Fil_{\bullet}^{\conj} \Del_{S/A} \{\bullet\}$. In the case of working over a perfect prism, one obtains an $\F$-$\V$-procomplex, which provides a comparison with the continuous, $p$-complete, relative $r$-truncated de Rham--Witt complex of Langer--Zink. We hope to address this in a more general setting in future work.

\subsection*{Overview of the Paper}
The format of the paper follows the natural succession that we just described.

In particular, in paragraph 2, we recall the main definitions and properties regarding topological Hochschild homology, topological restriction homology, and related invariants. In addition, we provide a quick review of the main results of the second work of Bhatt--Morrow--Scholze \cite{BMS2} regarding the motivic filtrations of $\THH$ and related invariants. Finally, we briefly review the theory of the (derived) de Rham complex, together with its proximity to the plain Hochschild homology via the HKR-filtration, as well as the theory of the de Rham--Witt complex.

In paragraph 3, we gather the calculations of $\TR^r (R_0;\Zp)^{hS^1} \to \TR^r (R_0;\Zp)$, in the case $R_0$ is a perfectoid ring. We provide explicit presentations for their motivic filtrations, which for finite $1 \leq r < \infty$ coincide with the double-speed Postnikov filtration, as well as polynomial presentations of their cohomology rings.

In paragraph 4, we first start with calculations in the case of quasiregular-semiperfectoid rings. As we already described, the even homotopy groups carry the structure relevant to the $r$-Nygaard filtration. Next, we use quasisyntomic descent to pass to the wider class of quasisyntomic rings (and left-Kan extending to go to the case of $p$-complete animated rings). We apply these results to the study of topological cyclic homology of $\TR$, providing an interpolating family of spectra between $\widetilde{\TC} ( \TR )$ and $\TC ( \TR)$.

In paragraph 5, we briefly define certain sheaves on the prismatization stack $\Sigma$ and also define $r$-Hodge--Tate divisors $\Sigma^{\HT, r}$. We introduce the $r$-Nygaard filtered absolute prismatic cohomology and study some of its main properties. Finally, we specialize to the relative situation and also discuss connections with the theory of the de Rham--Witt complex.

In paragraph 6, we specialize to the cases of mixed and positive characteristic, relating to $A\Omega$ and crystalline cohomology, ultimately drawing connections to the beginnings of the prismatic story, as developed in \cite{BMS1}.

Finally, in paragraph 7, we explain ongoing work and possible future directions. This includes the objects $\Sigma^{\HT, r}$ and $\Sigma_r'$, which encode in stacky terms the $r$-Hodge--Tate cohomology and the $r$-Nygaard filtered prismatic cohomology, respectively, the role of the absolute de Rham complex, as well as connections with factorization phenomena.

\subsection*{Conventions} Regarding the theory of $\infty$-categories and higher algebra, we follow standard conventions, as they are presented in Lurie's treatises \cite{HTT, HA}. Regarding the theory of cyclotomic spectra, topological Hochschild homology and topological restriction homology, we treat the $p$-typical case, following the works of Nikolaus--Scholze \cite{NS} and Antieau--Nikolaus \cite{AN}. Finally, regarding the theory of perfectoid rings, prismatic cohomology and how these relate to the theory of $\THH$, we follow the recent works of Bhatt--Morrow--Scholze \cite{BMS1, BMS2}, Bhatt--Scholze \cite{BS}, and Bhatt--Lurie \cite{APC}. In the next section, we briefly recall some of that background.

\subsection*{Acknowledgements}
This article stems from the author's PhD thesis \cite{andriopoulos2024motivic}, which was defended in the Mathematics Department of the University of Chicago, in Spring 2024. The author would like to express his deep gratitude towards his academic advisor, Prof. Sasha Beilinson, for his warmth and kindness, for the countless walks and conversations, and in general, the overall very inspiring interaction throughout the years. Without the crucial support from Sasha, especially during the difficult days of the pandemic, and his commitment to (mathematical) beauty, this work would not have been possible. The author would also like to extend his gratitude towards Prof. Akhil Mathew, for various discussions over the years, regarding the relevant material.

\newpage

\section{Preliminaries}
In this section, we briefly gather some background regarding topological Hochschild homology, topological restriction homology, and other related invariants. We mostly present the $p$-typical aspects of the story, providing references of the integral aspects, for the interested reader. We end with a brief discussion of the work of \cite{BMS2} and other on the motivic filtrations of $\THH$ and $\HH$ related invariants.

\subsection{Cyclotomic spectra and topological Hochschild homology}

The main advantage of the theory of topological Hochschild homology, as opposed to the original story of Hochschild homology, is the existence of the \emph{higher Frobenius} maps. In order to study them, Nikolaus--Scholze \cite{NS} constructed the $\infty$-category of \emph{cyclotomic spectra}; here we remind the reader of the $p$-typical story, the version outlined by Antieau--Nikolaus \cite{AN}.

\begin{defn}[$p$-typical cyclotomic spectra]
    We fix a prime number $p$. A \emph{$p$-typical cyclotomic spectrum} $X$ is a spectrum equipped with an $S^1$-action and an $S^1$-equivariant $\emph{higher Frobenius}$ map to its $C_p$-Tate fixed points $\varphi_p : X \to X^{tC_p}$. These assemble into the presentable, stable $\infty$-category of $\emph{$p$-typical cylotomic spectra}$ $\Cycsp_p$, which is defined to be the lax equalizer of the higher Frobenius and identity maps in $\Sp^{\mathrm{B}S^1}$.
\end{defn}

\begin{rmk}[Comparison to the genuine theory]
    As shown in \cite{NS}, this approach to the theory of cyclotomic spectra is particularly well-behaved in the bounded-below case, where it actually identifies with the more classical theory of genuine cyclotomic spectra.
\end{rmk}

\begin{ex}[Topological Hochschild homology]
    The main example of a $p$-typical cyclotomic spectrum, comes from the theory of \emph{topological Hochschild homology}. Since we are interested in $\THH$ of classical/animated rings, we can restrict ourselves to the following definition. Given a connective, $\einfty$-ring spectrum $A \in \Calgc$, topological Hochschild homology of $A$ is defined to be:
    \begin{equation*}
        \THH (A) := A \otimes_{A \otimes A^{\mathrm{op}}} A
    \end{equation*}
    The $S^1$-action on $\THH (A)$ can be seen via the presentation of the simplicial circle as a push-out:
    \begin{equation*}
        S^1 \simeq * \amalg_{* \amalg *} *
    \end{equation*}

    A geometric approach to topological Hochschild homology can be achieved through the theory of factorization homology and the equivalence:
    \begin{equation*}
        \THH (A) \simeq \int_{S^1} A
    \end{equation*}

    We can associate more invariants to $\THH (A)$, by playing with the $S^1$-action. Its $S^1$-homotopy fixed points constitute the \emph{negative topological cyclic homology} $\TC^- (A)$ and its $S^1$-Tate fixed points constitute the \emph{periodic topological cyclic homology} $\TP (A)$. Between those two, there exists the canonical map $\can : \TC^- (A) \to \TP (A)$ and the associated higher Frobenius $\varphi^{hS^1} : \TC^- (A) \to \TP (A)$.
\end{ex}

\begin{con}[Topological cyclic homology]
    Given a $p$-typical cyclotomic spectrum $X \in \Cycsp_p$, mapping out of the sphere spectrum $\Sph$, which is also an element of $\Cycsp_p$ with trivial $S^1$-action and higher Frobenius, produces \emph{topological cyclic homology}:
    \begin{equation*}
        \TC (X) := \map_{\Cycsp_p} (\Sph, X)
    \end{equation*}

    Using the lax equalizer presentation of $\Cycsp_p$, Nikolaus--Scholze produce a formula for computing topological cyclic homology:
    \begin{equation*}
        \TC (X) \simeq \eq \Big( \can, \varphi^{hS^1}: X^{hS^1} \rightrightarrows (X^{tC_p})^{hS^1}\Big)
    \end{equation*}
    In particular, for $\Zp$-coefficients due to the Tate orbit lemma, we are able to compute the topological cyclic homology of $A$, using $\TC^- (A)$ and $\TP (A)$:
    \begin{equation*}
        \TC (A;\Zp) := \TC (\THH (A;\Zp) ) \simeq \Big( \can,\varphi^{hS^1} : \TC^- (A;\Zp) \rightrightarrows \TP (A;\Zp) \Big)
    \end{equation*}
\end{con}

\subsection{Topological restriction homology}
The classical theory of cyclotomic spectra paves an approach to topological cyclic homology through genuine fixed points, with respect to finite subgroups of $S^1$. Since we restrict to the bounded-below case, genuine fixed points identify with topological restriction homology, as shown in \cite{NS}.

\begin{defn}[Topological restriction homology]
    Given a bounded-below $X \in \Cycsp_p$, we can construct its \emph{$r$-truncated topological restriction homology}, which identifies with the $C_{p^{r-1}}$-genuine fixed points of $X$ and is given by the following iterated pullback formula, for $r \geq 1$:
    \begin{equation*}
        \TR^r (X) := X^{hC_{p^r}} \times_{(X^{tC_p})^{hC^{p^{r-1}}}} X^{hC_{p^{r-1}}} \times_{(X^{tC_p})^{hC_{p^{r-2}}}} \dots \times_{X^{tC_p}} X \simeq X^{C_{p^{r-1}}}
    \end{equation*}
    where the maps on the right are $\varphi^{hC_{p^k}}$, for $0 \leq k \leq r-1$, and the maps on the left are the canonical ones.
    
    In the case $X = \THH (A)$, for some $A \in \Calgc$, we simply write $\TR^r (A) := \TR^r (\THH (A))$.

    There is a deep analogy between topological restriction homology and the theory of Witt vectors. In particular, there exist \emph{Restriction} maps $\Res : \TR^{r+1} (X) \to \TR^r (X)$, for $r \geq 1$ by forgetting the leftmost factor in the iterated pullback presentation, \emph{Frobenius} maps $\F : \TR^{r+1} (X) \to \TR^r (X)$, for $r \geq 1$ by forgetting the rightmost factor together with part of the homotopy fixed points information, and \emph{Verschiebung} maps $\V : \TR^r (X) \to \TR^{r+1} (X)$, for $r \geq 1$ by moving everything one place to the left, applying $(\cdot)^{hC_p}$, and introducing a copy of $X$ in the rightmost place.

    Taking the limit with respect to the Restriction maps yields the \emph{topological restriction homology} of $X$:
    \begin{equation*}
        \TR (X) := \limr \TR^r (X) \simeq \dots \times_{(X^{tC_p})^{hC_{p^r}}} X^{hC_{p^r}} \times_{(X^{tC_p})^{hC^{p^{r-1}}}} \dots \times_{X^{tC_p}} X
    \end{equation*}
    In this case, Frobenius and Verschiebung become endofunctors of $\TR$.
\end{defn}

The relation between topological restriction homology and the Witt vectors becomes explicit via the following theorem of Hesselholt-Madsen:
\begin{thm}[Hesselholt-Madsen \cite{hesselholt1997k}]
    Let $A \in \Calgc$ be a connective $\einfty$-ring spectrum. Then $\TR^r (A)$ is also a connective $\einfty$-ring spectrum. Applying the $0$-th homotopy group functor produces the following equivalence, which maps the structure associated with topological restriction homology to the usual structure associated with the Witt vectors of $\pi_0 A$:
    \begin{equation*}
        \pi_0 \TR^r (A) \simeq \W_r (\pi_0 A)
    \end{equation*}
\end{thm}

\begin{rmk}
    An equivalent way to present topological restriction homology is either as an equalizer, for $1 \leq r \leq \infty$:
    \begin{equation*}
        \TR^r (X) \simeq \eq \Bigg( \prod_{0 \leq k \leq r-1} X^{hC_{p^{i-1}}} \rightrightarrows \prod_{1 \leq k \leq r-1 } (X^{tC_p})^{hC_{p^{i-1}}} \Bigg)
    \end{equation*}
    
    or, inductively, via the pullback square, for $r \geq 1$:
    \begin{equation*}
        \begin{tikzcd}[column sep=huge]
            \TR^{r+1} (X) \arrow[r, "\Res"] \arrow[dd] & \TR^r (X) \arrow[dd] \arrow[dr, dotted, bend left = 15] &[-20pt] \\
            & & X^{hC_{p^{r-1}}} \arrow[dl, dotted, bend left = 15, "\varphi^{hC_{p^{r-1}}}"] \\
            X^{hC_{p^r}} \arrow[r, "\can"'] & (X^{tC_p})^{hC_{p^r}} &
        \end{tikzcd}
    \end{equation*}

    For $\Zp$-coefficients, observe that due to Tate orbit lemma \cite{NS}, the $C_{p^k}$-Tate fixed points of $\TR^r (A;\Zp)$ reduce to the following, via the natural projection map to $\THH (A;\Zp)$, for $1 \leq k \leq \infty$:
    \begin{equation*}
        \TR^r (A;\Zp)^{tC_{p^k}} \simeq \THH (A;\Zp)^{tC_{p^k}}
    \end{equation*}
    
    In fact, the iterated pullback formula simplifies to the following, for $1 \leq r \leq \infty$:
    \begin{equation*}
        \TR^r (A;\Zp) \simeq \THH (A;\Zp)^{hC_{p^r}} \times_{\THH (A;\Zp)^{tC_{p^{r-1}}}} \dots \times_{\THH (A;\Zp)^{tC_p}} \THH (A;\Zp)
    \end{equation*}
    There are similar simplifications in ways of presenting $\TR^r (A;\Zp)$.

    Finally, applying $S^1$-homotopy fixed points to the last formula, we have the following very important identification for $\TR^r (A;\Zp)^{hS^1}$, for $1 \leq r \leq \infty$:
    \begin{equation*}
        \TR^r (A;\Zp)^{hS^1} \simeq \TC^- (A;\Zp) \times_{\TP (A;\Zp)} \dots \times_{\TP (A;\Zp)} \TC^- (A;\Zp)
    \end{equation*}
    where the maps on the left are the canonical ones $\can: \TC^- (A;\Zp) \to \TP (A;\Zp)$, while the maps on the right are the $S^1$-homotopy fixed points of higher Frobenius $\varphi^{hS^1} : \TC^- (A;\Zp) \to \TP (A;\Zp)$.

    In particular, we highlight two maps who have $\TR^r (A;\Zp)^{hS^1}$ as their source. In particular, the first one corresponds to the $C_{p^r}$-homotopy fixed points of the higher Frobenius. It is obtained by mapping to $\TC^- (A;\Zp)$ on the left and, afterwards, to $\TP (A;\Zp)$ via the $S^1$-homotopy fixed points of higher Frobenius:
    \begin{equation*}
        \varphi^{hS^1} : \TR^r (A;\Zp)^{hS^1} \longrightarrow \TC^- (A;\Zp) \longrightarrow \TP (A;\Zp)
    \end{equation*}
    And, thus, it fits in the following commutative diagram:
    \begin{equation*}
        \begin{tikzcd}[column sep=huge]
            \TR^r (A;\Zp)^{hS^1} \arrow[r] \arrow[d] & \TC^- (A;\Zp) \arrow[r, "\varphi^{hS^1}"] \arrow[d] & \TP (A;\Zp) \arrow[d] \\
            \TR^r (A;\Zp) \arrow[r] & \THH (A;\Zp)^{hC_{p^{r-1}}} \arrow[r, "\varphi^{hC_{p^{r-1}}}"] & \THH (A;\Zp)^{tC_{p^r}}
        \end{tikzcd}
    \end{equation*}
    In addition to this, we also have a counterpart of the inclusion map, by mapping first to $\TC^- (A;\Zp)$ on the right and then to $\TP (A;\Zp)$ via the inclusion map:
    \begin{equation*}
        \can : \TR^r (A;\Zp)^{hS^1} \longrightarrow \TC^- (A;\Zp) \longrightarrow \TP (A;\Zp)
    \end{equation*}

    Note that the aforementioned description for the $S^1$-homotopy fixed points of $\TR$ can be also given in the form of an equalizer:
    \begin{equation*}
        \TR^r (A;\Zp)^{hS^1} \simeq \eq \Bigg( \prod_{1 \leq k \leq r-1} \TC^- (A;\Zp) \rightrightarrows \prod_{2\leq k \leq r-1} \TP (A;\Zp) \Bigg)
    \end{equation*}
    
    Of course, there exist such presentations in general, without the need to restrict to $\Zp$ coefficients. However, in our study, we mostly restrict to the case of working over $\Zp$.
\end{rmk}

\subsection{Cyclotomic spectra with Frobenius lifts}
If cyclotomic spectra provide a framework for studying the properties of higher Frobenius, then there should be a good notion of Frobenius lift. Such a notion does, in fact, exist and the prime example is given by $\TR (X)$, for $X \in \Cycsp_p$; it happens that Frobenius factors through the $C_p$-homotopy fixed points of $\TR (X)$.

\begin{defn}[$p$-typical cyclotomic spectra with Frobenius lifts]
    A \emph{$p$-typical cyclotomic spectrum with a Frobenius lift} is a spectrum $Y$ with an $S^1$-action and a \emph{Frobenius lift} map to its $C_p$-homotopy fixed points $F_p : Y \to Y^{hC_p}$. The assemble into the presentable, stable $\infty$-category of \emph{$p$-typical cyclotomic spectra with Frobenius lifts} $\Cycsp_p^{\mathrm{Fr}}$, which is defined to be the lax equalizer of the Frobenius lift and identity maps in $\Sp^{\B S^1}$.
\end{defn}

\begin{rmk}
    Any $Y \in \Cycsp_p^{\mathrm{Fr}}$ is also in $\Cycsp_p$, by considering the Frobenius map to be the Frobenius lift, postcomposed with the canonical map to the $C_p$-Tate fixed points:
    \begin{equation*}
        \begin{tikzcd}[column sep=huge]
            \varphi_p = \can \circ \F_p : Y \arrow[r, "\F_p"] & Y^{hC_p} \arrow[r, "\can"] & Y^{tC_p}
        \end{tikzcd}
    \end{equation*}
\end{rmk}

\begin{thm}[$\TR$ is a $p$-typical cyclotomic spectrum with Frobenius lifts, \cite{krause2018lectures}]
    The forgetful functor provided by the previous remark admits a right adjoint, given by $\TR$. Therefore, the $\infty$-categories $\Cycsp_p^{\mathrm{Fr}}$ and $\Cycsp_p$ fit into the following adjunction:
    \begin{equation*}
        \begin{tikzcd}[column sep=huge]
            \Cycsp_p^{\mathrm{Fr}} \arrow[rr, bend left=15, "\mathrm{forg.}"] & \perp & \Cycsp_p \arrow[ll, bend left=15, "\TR"]
        \end{tikzcd}
    \end{equation*}
\end{thm}

In fact, there is more to the story, as also the Verschiebung map of $\TR(X)$, for $X \in \Cycsp_p$, seems to factor through the $C_p$-homotopy orbits.

\begin{defn}[$p$-typical topological Cartier modules, \cite{AN}]
    There exists a presentable, stable $\infty$-category $\TCart_p$ of \emph{$p$-typical topological Cartier modules}, an element $M$ of which, is defined to be a spectrum equipped with and $S^1$-action, whose $C_p$-norm has a factorization, as follows:
    \begin{equation*}
        \begin{tikzcd}[column sep=huge]
            M_{hC_p} \arrow[r, "\V"]  \arrow[rr, bend right=25, "\Nm"'] & M \arrow[r, "\F"] & M^{hC^p}
        \end{tikzcd}
    \end{equation*}
\end{defn}

\begin{rmk}
    There exists an obvious chain of forgetful functors:
    \begin{equation*}
        \begin{tikzcd}[column sep=huge]
            \TCart_p \arrow[r, "\text{forg.}"] & \Cycsp_p^{\mathrm{Fr}} \arrow[r, "\text{forg.}"] & \Cycsp_p
        \end{tikzcd}
    \end{equation*}
\end{rmk}

\begin{rmk}
    Let us denote with $( \cdot )/ \V$ the cofiber of the Verschiebung map in $\TCart_p$. Given an $M \in \TCart_p$, consider the spectrum $M/ \V$. Then $M/ \V$ is a cyclotomic spectrum.
\end{rmk}

\begin{thm}[Antieau--Nikolaus adjunction, \cite{AN}]
    The functor $( \cdot )/ \V$ provided by the previous remark admits a right adjoint, given by $\TR$. Therefore, the $\infty$-categories $\TCart_p$ and $\Cycsp_p$ fit into the following adjunction:
    \begin{equation*}
        \begin{tikzcd}[column sep=huge]
            \TCart_p \arrow[rr, bend left=15, "( \cdot )/ \V"] & \perp & \Cycsp_p \arrow[ll, bend left=15, "\TR"]
        \end{tikzcd}
    \end{equation*}
    In fact, there exists a so-called \emph{cyclotomic $t$-structure} on $\Cycsp_p$, for which the functor $\TR$ is $t$-exact. In addition, in the bounded-below case, it is also fully faithful, thus identifying $\Cycsp_p$ with the \emph{$\V$-complete $p$-typical topological Cartier modules} $\widehat{\TCart}_p \subset \TCart_p$.
\end{thm}

\subsection{Topological cyclic homology}
As we already noted, topological restriction homology served as the initial tool for approaching topological cyclic homology.

\begin{con}[$\TC$ via $\TR$, \cite{NS}]
    Let $X$ be a bounded below $p$-typical cyclotomic spectrum. Then, we have the following equivalence, for $r \geq 1$:
    \begin{align*}
        \TC_r (X):= \fib \big(\Res - \F: \TR^{r+1} (X) \longrightarrow \TR^r (X) \big) \simeq \\ \fib \Big( \can - \varphi^{hC_{p^{r+1}}} : X^{hC_{p^{r+1}}} \longrightarrow (X^{tC_p})^{hC_{p^r}} \Big)
    \end{align*}
    It follows that, over $\Zp$-coefficients, the following holds:
    \begin{equation*}
        \TC(X) \simeq \limr \TC_r (X) \simeq \fib \big( 1 -\F: \TR (X) \longrightarrow \TR (X) \big)
    \end{equation*}
\end{con}

Therefore, understanding $\TR$ could possibly lead to a better understanding of $\TC$. In addition, $\TR$ has the structure of a $p$-typical cyclotomic spectrum with a Frobenius lift, and as a result via the forgetful functor, the structure of a $p$-typical cyclotomic spectrum. Thus, trying to understand $\TC ( \TR )$ would also be an interesting goal, in its own sake.

\begin{defn}[Topological cyclic homology $\widetilde{\TC}$ for $\Cycsp_p^{\mathrm{Fr}}$]
    Given a $p$-typical cyclotomic spectrum with a Frobenius lift $Y \in \Cycsp_p^{\mathrm{Fr}}$, mapping out of the sphere spectrum $\Sph$, which is also an element of $\Cycsp_p^{\mathrm{Fr}}$ with trivial $S^1$-action and Frobenius lift, produces \emph{topological cyclic homology for Frobenius lifts}:
    \begin{equation*}
        \widetilde{\TC} (Y) := \map_{\Cycsp_p^{\mathrm{Fr}}} (\Sph, Y)
    \end{equation*}
\end{defn}

\begin{con}[Computing $\widetilde{\TC}$ for $\Cycsp_p^{\mathrm{Fr}}$] \label{topological cyclic homology}
    In complete analogy with the standard $\TC$, we can use the lax equalizer approach to defining $\Cycsp_p^{\mathrm{Fr}}$, in order to have a formula for computing topological cyclic homology for Frobenius lifts:
    \begin{equation*}
        \widetilde{\TC} (Y) \simeq \eq \Big( \id, \F^{hS^1} : Y^{hS^1} \rightrightarrows Y^{hS^1} \Big)
    \end{equation*}
\end{con}
\begin{proof}
    The process is similar to the one in \cite[Prop. II.1.9]{NS}, for $\TC$.
\end{proof}

The natural question to ask is what happens in the case of the standard $\TC$ of $p$-typical cyclotomic spectra with Frobenius lifts.
\begin{con}[Computing $\TC$ of $\TR (A;\Zp)$]
    Let $A$ be a connective $\einfty$-ring spectrum and consider $\TR (A;\Zp)$. Then, the following equivalences hold, regarding topological cyclic homology:
    
    \begin{equation*}
        \widetilde{\TC} \Big( \TR (A;\Zp) \Big) \simeq \limr \widetilde{\TC}^r (A;\Zp)
    \end{equation*}
    for the spectra
    \begin{equation*}
        \widetilde{\TC}^r (A;\Zp) := \fib \Big( \Res^{hS^1} - \F^{hS^1} : \TR^r (A;\Zp)^{hS^1} \longrightarrow \TR^{r-1} (A;\Zp)^{hS^1} \Big)
    \end{equation*}
    
    Mapping further to the Tate constructions, we get that:
    \begin{equation*}
        \TC \Big( \TR (A;\Zp) \Big) \simeq \Big( 1 - \F^{hS^1} : \TR (A;\Zp)^{hS^1} \longrightarrow \TP (A;\Zp) \Big) \simeq \limr \TC^r (A;\Zp)
    \end{equation*}
    for the spectra:
    \begin{gather*}
        \TC^r (A;\Zp) := \fib \big( \Res^{hS^1} - \F^{hS^1} : \TR^r (A;\Zp)^{hS^1} \longrightarrow \TR^{r-1} (A;\Zp)^{hS^1} \longrightarrow \TP (A;\Zp) \Big) \\
        \simeq \fib \Big( \can - \varphi^{hS^1} : \TR^r (A;\Zp)^{hS^1} \longrightarrow \TC^- (A;\Zp) \longrightarrow \TP (A;\Zp) \Big)
    \end{gather*}
    Note that $\TC^r (A;\Zp)$ interpolates between $\TC (A;\Zp)$, for $r=1$, and $\TC \big( \TR (A;\Zp) \big)$, for $r=\infty$.
\end{con}
\begin{proof}
    These essentially follow from the commutativity of the diagrams, below. The diagrams on the left correspond to the action of the Restriction operators above/ Frobenius operators below, for $\TR^r (A;\Zp)$. Applying $S^1$-homotopy fixed points, we obtained the diagrams on the right.
    \begin{equation*}
        \begin{tikzcd}
            \TR^{r+1} (A;\Zp) \arrow[r, "\Res"] \arrow[d] &[10pt] \TR^r (A;\Zp) \arrow[d] &[20pt]
            \TR^{r+1} (A;\Zp)^{hS^1} \arrow[r, "\Res^{hS^1}"] \arrow[d] &[10pt] \TR^r (A;\Zp)^{hS^1} \arrow[d] \\
            \THH (A;\Zp)^{hC_{p^r}} \arrow[r, "\can"] & \THH (A;\Zp)^{tC_{p^r}} &
            \TC^- (A;\Zp) \arrow[r, "\can"] & \TP (A;\Zp) \\[10pt]
            \TR^{r+1} (A;\Zp) \arrow[r, "\F"] \arrow[d] & \TR^r (A;\Zp) \arrow[d] &
            \TR^{r+1} (A;\Zp)^{hS^1} \arrow[r, "\F^{hS^1}"] \arrow[d] & \TR^r (A;\Zp)^{hS^1} \arrow[d] \\
            \THH (A;\Zp)^{hC_{p^r}} \arrow[r, dotted] \arrow[d, "\varphi^{hC_{p^r}}"] & \THH (A;\Zp)^{tC_{p^r}} & \TC^- (A;\Zp) \arrow[r, "\varphi^{hS^1}"] & \TP (A;\Zp) \\
            \THH (A;\Zp)^{tC_{p^{r+1}}} \arrow[ur, "\mathrm{proj.}"]
        \end{tikzcd}
    \end{equation*}

Taking homotopy fibres, it follows that:
\begin{gather*}
    \TC^r (A;\Zp) := \fib \Big( \Res^{hS^1} - \F^{hS^1} : \TR^r (A;\Zp)^{hS^1} \longrightarrow \TR^{r-1} (A;\Zp)^{hS^1} \longrightarrow \TP (A;\Zp) \Big) \\
    \simeq \fib \Big( \can - \varphi^{hS^1} : \TR^r (A;\Zp)^{hS^1} \longrightarrow \TC^- (A;\Zp) \longrightarrow \TP (A;\Zp) \Big)
\end{gather*}
\end{proof}

We have only presented the $p$-typical story, which is going to suffice for the majority of this work, except for the last part of section $4$. For the reader who wishes to view the integral statements, we direct them to works such as \cite{NS}, \cite{krause2018lectures}, \cite{krause2023polygonic}, \cite{mccandless2021curves}.

Let us finally note that one of the reasons for the importance of the theory of topological cyclic homology is its proximity to algebraic $\K$-theory, as a result of the theory of traces and the story of curves in $\K$-theory.

\subsection{Motivic filtrations following \cite{BMS2}}
For the full approach to prismatic cohomology, via the theory of $\THH$, as well as for the basics of quasiregular-semiperfectoid rings and quasisyntomic descent, we direct the reader to \cite{BMS2}. A brief recount can also be found in \cite{AMMN}. Here, we provide a reminder on the basics of quasisyntomic descent and focus on the main commutative square involving invariants of $\THH$, which gathers the substance of \cite{BMS2} ideas.

\begin{defn}[Quasisyntomic rings and quasisyntomic topology, {\cite[Sec. 4]{BMS2}}]
    1) We call a $p$-complete ring, $R$ \emph{quasisyntomic}, if it has bounded $p^{\infty}$ torsion and its cotangent complex $\mathbb{L}_{A/\Zp}$ has $p$-complete Tor amplitude concentrated in $[-1,0]$. The category of quasisyntomic rings is denoted by $\mathrm{QSyn}$.\\
    2) A map of $p$-complete rings is called a \emph{quasisyntomic map/quasisyntomic cover} if $B$ is $p$-completely flat/faithfully flat over $A$ and the cotangent complex $\mathbb{L}_{B/A}$ has $p$-complete Tor amplitude concentrated in $[-1,0]$.\\
    3) The category $\mathrm{QSyn}^{\mathrm{op}}$ obtains the structure of a site, when equipped with the quasisyntomic covers. We call this the \emph{quasisyntomic site}.
\end{defn}

Let us recall some important, special cases of quasisyntomic rings, the first of which is the class of perfectoid rings.

\begin{defn}[Perfectoid rings {\cite[Sec. 3]{BMS1}}]
    A ring $R_0$ is called \emph{perfectoid} if and only if it is $\pi$-adically complete for some element $\pi \in R_0$, for which $\pi^p$ divides $p$, the Frobenius map $\varphi: R_0/p \to R_0/p$ is surjective, and the kernel of Fontaine's map $\vartheta: \ainf (R_0) \to R_0$ is principal, where $\ainf (R_0) := \W (R_0^{\flat})$.
\end{defn}

Certain quotients of perfectoid rings are also included in the category of quasisyntomic rings:

\begin{defn}[Quasiregular-semiperfectoid rings, following \cite{BMS2}]
    1) We call a ring $S$ \emph{quasiregular-semiperfectoid} if it is quasisyntomic, has a perfectoid base $R_0 \to S$ and its reduction $S/p$ is semiperfect.\\
    2) We denote by $\mathrm{QRSPerfd}$ the category of quasiregular-semiperfectoid rings. In particular, $\mathrm{QRSPerfd}^{\mathrm{op}}$ becomes a site, when equipped with quasisyntomic covers.
\end{defn}

The significance of quasiregular-semiperfectoid rings stems from the fact that on the one hand are quite computable, while on the other, they provide a basis for the quasisyntomic topology:
\begin{prop}[Quasisyntomic descent, following \cite{BMS2}]
    The natural map:
    \begin{equation*}
        u : \mathrm{QRSPerfd}^{\mathrm{op}} \longrightarrow \mathrm{QSyn}^{\mathrm{op}}
    \end{equation*}
    identifies the class of quasiregular-semiperfectoid rings as a basis for the quasisyntomic topology.
\end{prop}
For further details, the interested reader is directed to \cite[Sect. 4]{BMS2}.

Let $S$ be a quasisyntomic ring, for which we consider the following commutative square of \cite{BMS2}:
\begin{equation*}
    \begin{tikzcd}[column sep=huge]
        \TC^- (S;\Zp) \arrow[r, "\varphi^{hS^1}"] \arrow[d] & \TP (S;\Zp) \arrow[d] \\
        \THH (S;\Zp) \arrow[r, "\varphi"'] & \THH (S;\Zp)^{tC_p}
    \end{tikzcd}
\end{equation*}

The invariants of this square are equipped with the complete, exhaustive, decreasing, multiplicative, $\Z$-indexed \emph{motivic filtrations}, which induce associated \emph{motivic spectral sequences}. In the case $S$ is a quasiregular-semiperfectoid ring, the filtrations are nothing but the double-speed Postnikov filtrations; in fact, the general case follows from this, via quasisyntomic descent. Passing to $i$-th graded pieces for the motivic filtrations, we obtain the following commutative square:
\begin{equation*}
    \begin{tikzcd}[column sep=huge]
        \n^{\geq i} \Delc_S \{ i\} [2i] \arrow[r, "\varphi_i", "\text{divided Frob.}"'] \arrow[d] &[1in] \Delc_S \{ i\} [2i] \arrow[d] \\
        \n^i \Delc_S [2i] \arrow[r, "\gr \varphi_i", "\text{graded Frob.}"'] & \overline{\Del}_S \{ i\} [2i]
    \end{tikzcd}
\end{equation*}
In this $\THH$-approach to prismatic cohomology, all invariants happen to be complete, with respect to the \emph{Nygaard filtration}. This is a filtration detected by $\TC^- (S;\Zp)$, as a result of the $S^1$-homotopy fixed points spectral sequence: On the upper row, we have the divided prismatic Frobenius $\varphi_i$ which takes the $i$-th \emph{Nygaard filtered} piece $\n^{\geq i} \Delc_S$ to $\Delc_A$. On the lower row, we have the graded counterpart of the aforementioned situation, where the $n$-th graded piece of the Nygaard filtration $\n^i \Delc_S$ maps to the \emph{Hodge-Tate cohomology} $\overline{\Del}_S$.

For $S$ a $p$-complete, animated ring, such a diagram also arises in the non-Nygaard complete case of absolute prismatic cohomology:
\begin{equation*}
    \begin{tikzcd}[column sep=huge]
        \n^{\geq i} \Del_S \{ i\} \arrow[r, "\varphi_i"] \arrow[d] & \Del_S \{ i\} \arrow[d] \\
        \n^i \Del_S \{ i\} \arrow[r] & \overline{\Del}_S \{ i\}
    \end{tikzcd}
\end{equation*}
Let us move to the relative case. Given a bouned prism $(A,I)$ and assuming that $S$ is an $\overline{A}$-algebra, it follows that we can describe the relative divided prismatic Frobenius as $\varphi = I^i \varphi_i$. The $i$-th associated graded piece gets identified with the $i$-th filtered piece of the \emph{relative conjugate filtration} on the relative Hodge--Tate cohomology:
\begin{equation*}
    \begin{tikzcd}
        \n^i \varphi_A^* \Del_{S/A} \{ i\} \arrow[r, "\simeq"] & \Fil_i^{\conj} \overline{\Del}_{S/A} \{ i\}
    \end{tikzcd}
\end{equation*}
and thus the map to $\overline{\Del}_S \{ i\}$ gets identified with the canonical inclusion. The graded pieces of the conjugate filtration give rise to the relative Hodge--Tate comparison:
\begin{equation*}
    \gr_i^{\conj} \overline{\Del}_{S/A} \{ i\} \simeq \gr_{\hod}^i \dr_{S/\overline{A}} 
\end{equation*}
where the invariant on the right is the $p$-complete version of the relative derived de Rham complex, equipped with its Hodge filtration.

Taking homotopy fibres, one can have an explicit formula calculating the associated graded pieces for the motivic filtration of topological cyclic homology:
\begin{equation*}
    \grM^i \TC (A;\Zp) \simeq \fib \Big( \can - \varphi_i : \n^{\geq i} \Delc_S \{ i\} \longrightarrow \Delc_S \{ i\} \Big)
\end{equation*}

The starting point for the calculations of \cite{BMS2} is the case when $R=R_0$ a perfectoid ring. Let us recall this, as this is essentially the basis for our discussions in the next section. In particular, passing to even homotopy groups, the commutative square gives rise to:
\begin{equation*}
    \begin{tikzcd}[row sep=huge]
        \ainf (R_0) [u,v] / (u v - \xi) \arrow[r, "u \mapsto \sigma \text{, } v \mapsto \widetilde{\xi} \sigma^{-1}", "\varphi \text{-linear (divided Frob.)}"'] \arrow[d, "u \mapsto u \text{, } v \mapsto 0", "\vartheta \text{-linear}"'] &[1.5in] \ainf (R_0) [\sigma, \sigma^{-1}] \arrow[d, "\sigma \mapsto \sigma", "\widetilde{\vartheta} \text{-linear}"'] \\
        R_0 [u] \arrow[r, "u \mapsto \sigma", "R_0 \text{-linear}"'] & R_0 [\sigma, \sigma^{-1}]
    \end{tikzcd}
\end{equation*}

\subsection{Absolute prismatic cohomology via prismatization stacks}

The algebro-geometric path to the invariants of absolute prismatic cohomology has been paved via the prismatization stacks, constructed invependently by Bhatt-Lurie \cite{APC, APC2, APC3} and Drinfeld \cite{Prismatization}. The notation regarding the prismatization stack varies - it has been denoted by $\wcart$, $(\spf \Zp)^{\Del}$ or $\Sigma$. In this article, we use variations of the latter, for simplicity.

\begin{defn}[Prismatization]
    The prismatization stack $\Sigma$ is defined as the fpqc sheaf in (animated) $p$-nilpotent rings, whose $R$ points is defined to be the Cartier--Witt divisors on $R$, \cite[Def. 3.1.4]{APC}. If we have a fixed $p$-adic formal scheme $X$, then its prismatization $\Sigma_X$ is defined to be the sheaf, whose $R$-points consists of the maps $X \to \spf W(R)/I$, where $I \to W(R)$ is a Cartier--Witt divisor. There is a canonical projection:
    \begin{equation}
        \pi: \Sigma_X \longrightarrow \Sigma
    \end{equation}

    Given a prism $(A,I)$, one can construct a map of sheaves, associated with it, \cite[Con. 3.2.4]{APC}:
    \begin{equation*}
        \rho_A : \spf A \longrightarrow \Sigma
    \end{equation*}
    Under this map, we have the following description of the derived $\infty$-category on $\Sigma$, \cite[Prop. 3.3.5]{APC}:
    \begin{equation*}
        \begin{tikzcd}
            \mathcal{D} (\Sigma) \arrow[r, "\simeq"] & \lim_{(A,I)} \widehat{\mathcal{D}} (A)
        \end{tikzcd}
    \end{equation*}
    Using this equivalence, one can consider certain sheaves on $\Sigma$:
    \begin{itemize}
        \item The structure sheaf $\mathcal{O}_{\Sigma}$, as the unit object in $\mathcal{D}(\Sigma)$
        
        \item The ideal sheaf $\mathcal{I} \to \mathcal{O}_{\Sigma}$

        \item The Breuil--Kisin twist $\mathcal{O}_{\Sigma} \{ \bullet\}$

        \item Given an animated ring $S$, the prismatic sheaf $\mathcal{H}_{\Del} (S) \{ \bullet\}$
    \end{itemize}
    It happens that $\mathcal{D} (\Sigma)$ is generated, under shifts and colimits,  by the invertible sheaves $\mathcal{I}^n$, $n \in \Z$, \cite[Cor. 3.5.16]{APC}.
    
    Taking global sections over $\Sigma$ for these sheaves yields the associated structures in the absolute prismatic theory. This is denoted as follows: $\Del_S \{ \bullet\}$ for the absolute prismatic cohomology, $\Del_S^{\HT} \{ \bullet\}$ or $\overline{\Del}_S \{ \bullet\}$ for the absolute Hodge--Tate cohomology, and $I \Del_S$ for the structure that comes from the Hodge--Tate ideal sheaf $\mathcal{I}$.

    Note that part of the essence in using the stacky approach to prismatic cohomology, lies in the fact that for sufficiently nice $S$, the following equivalence holds, \cite[Cor. 8.17]{APC2}:
    \begin{equation*}
        \pi_* \mathcal{O}_{\Sigma_{\spf S}} \simeq \mathcal{H}_{\Del} (S)
    \end{equation*}
\end{defn}

An important part of the geometry of $\Sigma$ is controlled by its closed substack $\Sigma^{\HT}$, which encodes the Hodge--Tate cohomology.

\begin{defn}[Hodge--Tate divisor]
    We denote by $\Sigma^{\HT}$ the Hodge--Tate divisor of $\Sigma$, which is obtained as the closed substack associated with the invertible ideal sheaf $\mathcal{I} \in \mathcal{D}(\Sigma)$. Its $R$ points can be identified with those Cartier--Witt divisors $I \to \W (R)$, for which the following composition vanishes:
    \begin{equation*}
        \begin{tikzcd}
            I \arrow[r] \arrow[rr, bend right=25, "\simeq\; 0" description] & \W (R) \arrow[r] & W(R)/ VW(R) \simeq R
        \end{tikzcd}
    \end{equation*}

    Given a prism $(A,I)$, this defines a map $\rho_A^{\HT}$ to $\Sigma^{\HT}$, which fits in the following pullback square:
    \begin{equation*}
        \begin{tikzcd}[column sep=huge]
            \spf A/I \arrow[r, "\rho_A^{\HT}"] \arrow[d, hook] & \Sigma^{\HT} \arrow[d, hook] \\
            \spf A \arrow[r, "\rho_A"] & \Sigma
        \end{tikzcd}
    \end{equation*}
    Restricting the sheaves from the prismatization stack to the Hodge--Tate divisor, under the canonical inclusion $\iota: \Sigma^{\HT} \hookrightarrow \Sigma$, we have that $\mathcal{D} (\Sigma^{\HT})$ is generated, under shifts and colimits, by the Breuil--Kissin twists $\mathcal{O}_{\Sigma^{\HT}} \{ n\}$, $n \geq 0$, \cite[Prop. 3.5.15]{APC}.

    Using the prismatic maps to the Hodge--Tate divisor, one can construct a Hodge--Tate sheaf $\mathcal{H}_{\Del^{\HT}} (S) \{ \bullet\}$ on $\Sigma^{\HT}$. It happens that it can be identified with the restriction of the prismatic sheaf on the Hodge--Tate divisor:
    \begin{equation*}
        \mathcal{H}_{\Del} (S) \big|_{\Sigma^{\HT}} \simeq \mathcal{H}_{\Del^{\HT}} (S)
    \end{equation*}

    Again, the prism map constructions equip the Hodge--Tate cohomology sheaf with a multiplicative, increasing exhaustive conjugate filtration $\Fil_{\bullet}^{\conj} \mathcal{H}_{\Del^{\HT}} (S) \{ \bullet\}$. Note that taking the global sections yields the absolute conjugate filtration on the absolute Hodge--Tate cohomology $\Fil_{\bullet}^{\conj} \Del_S^{\HT} \{ \bullet \}$. Passing to the associated graded pieces, the Hodge--Tate comparison theorem yields the following identification of sheaves on $\Sigma^{\HT}$:
    \begin{equation*}
        \gr_i^{\conj} \mathcal{H}_{\Del^{\HT}} (S) \{ i\} \simeq L\Omega_S^i \otimes \mathcal{O}_{\Sigma^{\HT}} [-i]
    \end{equation*}
\end{defn}

One of the advantages of using the prismatization stack, is that it comes with a version of Frobenius, which we denote by $\Phi: \Sigma \to \Sigma$. On the absolute prismatic cohomology sheaves, equipped with the Nygaard filtration on the Frobenius pullback, it induces the following natural map:
\begin{equation*}
    \begin{tikzcd}
        \n^{\geq i} \Phi^* \mathcal{H}_{\Del} (S) \{ i\} \arrow[r, "\simeq"] & \mathcal{I}^i \otimes \mathcal{H}_{\Del} (S) \{ i\}
    \end{tikzcd}
\end{equation*}
The existence of a prismatic Frobenius map on Nygaard-filtered prismatic cohomology follows:
\begin{equation*}
    \varphi: \n^{\geq i} \Del_S \{ i\} \longrightarrow I^i \Del_S \{ i\}
\end{equation*}
In turn, this induces a divided prismatic Frobenius map:
\begin{equation*}
    \varphi_i : \n^{\geq i} \Del_S \{ i\} \longrightarrow \Del_S \{ i\}
\end{equation*}



\subsection{The de Rham--Witt complex and the HKR-filtration on Hochschild homology}

The invariants that play the most important role in the prismatic theory are the algebraic K\"ahler differentials/the algebraic de Rham complex. The interested reader is directed to \cite{APC}. In brief, given an $A$-algebra $B$, one is able to construct the relative K\"ahler differentials $\Omega_{B/A}^1$ by hand, via generators and relations. Taking wedge powers $\Omega_{B/A}^i :=\bigwedge^i \Omega_{B/A}^1$, one is able to form the de Rham complex:
\begin{equation*}
    \begin{tikzcd}
        0 \arrow[r] & B \arrow[r] & \Omega_{B/A}^1 \arrow[r] & \dots \arrow[r] & \Omega_{B/A}^i \arrow[r] & \dots
    \end{tikzcd}
\end{equation*}
which is equipped with the Hodge filtration.

By taking the animated version of this construction, one is able to consider the Hodge filtered derived de Rham complex $\Fil_{\hod}^{\bullet} \dr_{B/A}$, which we also sometimes denote by $L\Omega_{B/A}^{\geq \bullet}$. Even though such constructions work in the absolute case, with integral coefficients, in this article we only consider the relative case, with $p$-typical/$p$-complete coefficients.

The theory of the de Rham complex has natural generalizations in the context of mixed characteristic. The original construction is due to Bloch--Deligne--Illusie, for inputs of positive characteristic. However, there exist more recent constructions, which work in greater generality, in the $p$-typical or integral settings.

The relative version is due to Langer--Zink \cite{LangerZink} and works for a pair of $\Z_{(p)}$-algebras $A \to B$, producing a so-called \emph{$F$-$\V$-procomplex} $W_r \Omega_{B/A}$, for $r \geq 1$. It is this version which has been shown to give rise to a de Rham--Witt comparison theorem for prismatic cohomology \cite{BMS1, molokov2020prismatic}.

On the other hand, Hesselholt--Madsen have constructed an absolute de Rham--Witt complex, defined both in the $p$-typical and integral settings (big de Rham--Witt complex). This relies on the notion of a \emph{Witt complex} and is denoted by $W_r\Omega$, \cite{hesselholt2015big}. It has been studied extensively and shown to be intricately connected to $\TR^r$, \cite{hesselholt1996p, hesselholt1997k, hesselholt2003k, hesselholt2004rham, hesselholt2005absolute, hesselholt2015big, AN}. In the remainder of this article, we denote the animated versions of the Hodge-filtered de Rham--Witt complex (absolute or relative) by $\Fil_{\hod}^{\bullet} \drw_{r, -}$.

The above connection between Hochschild-style invariants and de Rham--Witt complexes is well understood in the case $r=1$, under the Hochschild--Kostant--Rosenberg theorem, which involves plain Hochschild homology and the derived de Rham complex.

As shown in the works of  Antieau \cite{AntieauDerived}, Bhatt-Morrow-Scholze \cite{BMS2}, Moulinos--Robalo--To\"en \cite{moulinos2022universal}, and Raksit \cite{Raksit}, the invariants that are associated with Hochschild homology are equipped with integral motivic filtrations, whose graded pieces can be computed in terms of the Hodge-completed derived de Rham complex.

\begin{thm}[HKR filtration] \label{filtered circle}
    The following invariants associated with Hochschild homology are equipped with integral, complete, exhaustive, decreasing, multiplicative, $\Z$-indexed motivic filtrations, for an animated ring $A$. Passing to their associated graded pieces, these can be expressed in terms of the Hodge-completed derived de Rham complex:
    \begin{gather*}
        \grM^n \HH (A) \simeq \wedge^n \mathbb{L}_A [2n] \\
        \grM^n \HC^- (A) \simeq \widehat{\dr}_A^{\geq n} [2n] \\
        \grM^n \HP (A) \simeq \widehat{\dr}_A [2n] \\
        \grM^n \HC (A) [1] \simeq \widehat{\dr}_A / \widehat{\dr}_A^{\geq n} [2n-1]
    \end{gather*}
\end{thm}
Using these, in the next section we recall certain fibre sequences involving $\THH$ of perfectoid rings, from \cite{BMS2}.

\newpage

\section{The perfectoid case}
In this section we begin by recalling some basic constructions regarding perfectoid rings, mainly following \cite[Sec. 3]{BMS1}. We apply these in the study of the invariants $\TR^r$ and $(\TR^r)^{hS^1}$, in the case of perfectoid rings, building on ideas of \cite[Sec. 6]{BMS2} and \cite[Sec. 7]{MK1}.

\subsection{Perfectoid Rings} Here we gather some background regarding perfectoid rings. In what follows, $R_0$ always denotes a perfectoid ring.

\begin{con}[Fontaine-style maps $\vartheta_r, \widetilde{\vartheta}_r$ {\cite[Lem. 3.2]{BMS1}}] \label{Fontaine maps}
    Let $R_0$ be a perfectoid ring. Then the following equivalence holds:
    \begin{equation*}
        \ainf (R_0) \simeq \underset{\F}{\lim} \W_r (R_0)
    \end{equation*}
    As a result of this, the Frobenius automorphism $\varphi$ on $\ainf (R_0)$ is identified with the Witt vector Frobenius $\F : \ainf (R_0) \to \ainf (R_0)$, while its inverse $\varphi^{-1}$ is identified with the restriction map $\Res : \ainf (R_0) \to \ainf (R_0)$.

    Under this identification, we can consider the projection maps to the $r$-truncated Witt vectors of $R_0$:
    \begin{equation*}
        \widetilde{\vartheta}_r: \ainf (R_0) \longrightarrow \W_r (R_0)
    \end{equation*}
    as well as their twists by the $r$-th iteration of $\varphi^{-1}$:
    \begin{equation*}
        \vartheta_r := \widetilde{\vartheta}_r \circ \varphi^r : \ainf (R_0) \longrightarrow \W_r (R_0)
    \end{equation*}
    In particular, for $r=1$, we obtain the aforementioned Fontaine map $\vartheta_1 = \vartheta$, whose kernel is generated by the degree $1$ distinguished element $\xi \in \ker \vartheta$.

    From this, one is able to also deduce that the kernel of the map $\vartheta_r$ is generated by the element:
    \begin{equation*}
        \xi_r := \xi \varphi^{-1} (\xi) \dots \varphi^{-(r-1)} (\xi)
    \end{equation*}
    and the kernel of $\widetilde{\vartheta}_r$ is generated by
    \begin{equation*}
        \widetilde{\xi}_r := \varphi^r (\xi_r) = \varphi (\xi) \varphi^2 (\xi) \dots \varphi^r (\xi)
    \end{equation*}

    Finally, taking the derived limit over the Restriction mas, we write $\xi_{\infty}:= \rlimr \xi_r$. It follows that:
    \begin{equation*}
        \ainf (R_0)/ \xi_{\infty} \simeq \rlimr \ainf (R_0) / \xi_r \simeq \W (R_0)
    \end{equation*}
    In particular, if $R_0 = \mathcal{O}_K$, for a spherically complete perfectoid ring $K$, we get that $\xi_{\infty} = \underset{\Res}{\lim} \xi_r = \mu$ and $\ainf (R_0)/\mu \simeq \W (R_0)$.
\end{con}

The interactions between the maps $\vartheta_r$, $\widetilde{\vartheta}_r$ and the Restriction, Frobenius maps on the Witt vectors are documented in the following lemma:
\begin{lem}[The action of $\Res$, $\F$ following {\cite[Lem. 3.4]{BMS1}}] \label{Fontaine diagrams} Consider the Witt vector Restriction and Frobenius maps. Under their action, the following diagrams are commutative for $\vartheta_r$:
\begin{equation*}
    \begin{tikzcd}[column sep=huge, row sep=huge]
        \ainf (R_0) \arrow[d, "\vartheta_{r+1}"'] \arrow[r, "\id"] & \ainf (R_0) \arrow[d, "\vartheta_r"] & \ainf (R_0) \arrow[d, "\vartheta_{r+1}"'] \arrow[r, "\varphi"] & \ainf(R_0) \arrow[d, "\vartheta_r"] \\
        \W_{r+1} (R_0) \arrow[r, "\Res"] & \W_r (R_0) & \W_{r+1} (R_0) \arrow[r, "\F"] & \W_r (R_0)
    \end{tikzcd}
\end{equation*}
Also, the following diagrams are commutative for $\widetilde{\vartheta}_r$:
\begin{equation*}
    \begin{tikzcd}[column sep=huge, row sep=huge]
        \ainf (R_0) \arrow[r, "\varphi^{-1}"] \arrow[d, "\widetilde{\vartheta}_{r+1}"'] & \ainf (R_0) \arrow[d, "\widetilde{\vartheta}_r"] & \ainf (R_0) \arrow[r, "\id"] \arrow[d, "\widetilde{\vartheta}_{r+1}"'] & \ainf (R_0) \arrow[d, "\widetilde{\vartheta}_r"] \\
        \W_{r+1} (R_0) \arrow[r, "\Res"] & \W_r (R_0) & \W_{r+1} (R_0) \arrow[r, "\F"] & W_r (R_0)
    \end{tikzcd}
\end{equation*}
\end{lem}

\subsection{$\TR$ of perfectoid rings} Now we move on to study the invariants of $\TR^r$ of perfectoid rings. Given a perfectoid ring $R_0$, an important property is that most invariants of $\THH$ are concentrated on even degrees. This is also true for $\TR^r (R_0;\Zp)$, whose properties are documented in the following proposition.

\begin{prop}[The properties of $\TR^r (R_0;\Zp)$, following {\cite[Sec. 6]{BMS2}}, {\cite[Sec. 7]{MK1}}] \label{perfectoid TR}
    Consider the following maps, defined for each $1 \leq r < \infty$:
    \begin{equation*}
        \begin{tikzcd}[column sep=huge]
            \TR^r (R_0;\Zp) \arrow[r] & \THH (R_0;\Zp)^{hC_{p^{r-1}}} \arrow[r, "\varphi^{hC_{p^{r-1}}}"] & \THH (R_0;\Zp)^{tC_{p^r}}
        \end{tikzcd}
    \end{equation*}
    Then, these are equivalences on connective covers. In particular, the following hold for $\TR^r (R_0;\Zp)$:\\
    1) The spectrum $\TR^r (R_0;\Zp)$ is concentrated in even degrees, for $1 \leq r < \infty$.\\
    2) The object $\pi_2 \TR^r (R_0;\Zp)$ is an invertible module over $\pi_0 \TR^r (R_0;\Zp) \simeq \W_r (R_0)$, as a result of which, the following multiplication map is an isomorphism for $i \geq 0$:
    \begin{equation*}
        \mathrm{Sym}^i \pi_2 \TR^r (R_0;\Zp) \longrightarrow \pi_{2i} \TR^r (R_0;\Zp)
    \end{equation*}
\end{prop}

As we noted in the introduction, our investigations follow a somewhat indirect course. In particular, the plan is to first look at the $S^1$-homotopy fixed points $\TR^r (R_0;\Zp)^{hS^1}$ and then pass to $\TR^r (R_0;\Zp)$ itself, for $1 \leq r < \infty$. The following two propositions are devoted to these. In particular, we first identify the structure of the homotopy groups of $\TR^r (R_0;\Zp)^{hS^1}$ and then explain how to pass from the $S^1$-homotopy fixed points to the original spectrum $\TR^r (R_0;\Zp)$.

We are now able to identify the structure of the spectra $\TR^r (R_0;\Zp)^{hS^1} \to \TR^r (R_0;\Zp)$, the bulk of which is concentrated in the following theorem:

\begin{prop}[Understanding $\pi_* \TR^r (R_0;\Zp)^{hS^1}$] \label{graded ring}
    The spectrum $\TR^r (R_0;\Zp)^{hS^1}$, $1 \leq r < \infty$ is concentrated in even degrees. In particular, the $S^1$-homotopy fixed points spectral sequence degenerates, yielding the following identification for the even homotopy groups:
    \begin{equation*}
        \pi_{2i} \TR^r (R_0;\Zp)^{hS^1} \simeq
        \begin{cases}
            \xi_r^i \ainf (R_0) \{ i\}, & i \geq 0 \\[5pt]
            \ainf (R_0) \{ i\}, & i < 0
        \end{cases}
    \end{equation*}
\end{prop}
\begin{proof}
    Since the spectrum $\TR^r (R_0;\Zp)$ is concentrated in even degrees, the same happens for the spectra $\TR^r (R_0;\Zp)^{hS^1} \to \TR^r (R_0;\Zp)^{tS^1}$, from say \cite[Rem. 6.1.7]{APC}. However, from the Tate orbit lemma of Nikolaus--Scholze \cite[Lem. I.2.1]{AN} we know that $\TR^r (-;\Zp)^{tC_p} \simeq \THH (-;\Zp)^{tC_p}$, and therefore, $\TR^r (-;\Zp)^{tS^1} \simeq \TP (-;\Zp)$. From \cite[Sec. 6]{BMS2}, it follows that $\TP (R_0;\Zp)$ has the following identification for its even homotopy groups:
    \begin{equation*}
        \pi_{2i} \TP (R_0;\Zp) \simeq \ainf (R_0) \{ i\}
    \end{equation*}
    
    Hence, forgetting the multiplicative structure of $\TR^r (R_0;\Zp)^{hS^1}$, the same is \emph{non-canonically} true for its even homotopy groups, as well:
    \begin{equation*}
        \pi_{2i} \TR^r (R_0;\Zp)^{hS^1} \simeq \ainf (R_0) \{ i\}
    \end{equation*}
    
    Since the spectrum $\TR^r (R_0;\Zp)$ is concentrated even degrees, the associated $S^1$-homotopy fixed points spectral sequence degenerates, thus endowing $\pi_{2i} \TR^r (R_0;\Zp)^{hS^1}$ with a filtration. Applying $\pi_0$ to the map $\TR^r (R_0;\Zp)^{hS^1} \to \TR^r (R_0;\Zp)$, this is identified with the map:
    \begin{equation*}
        \vartheta_r : \ainf (R_0) \longrightarrow \W_r (R_0)
    \end{equation*}
    whose kernel is generated by the element $\xi_r$. This follows by the commutative square:
    \begin{equation*}
        \begin{tikzcd}[column sep=huge, row sep=huge]
            \ainf (R_0) \arrow[r, "\varphi^r"] \arrow[d, "\vartheta_r"'] & \ainf (R_0) \arrow[d, "\widetilde{\vartheta}_r"] \\
            \W_r (R_0) \arrow[r] & \W_r (R_0)
        \end{tikzcd}
    \end{equation*}
    which is obtained by applying $\pi_0$ on the commutative square:
    \begin{equation*}
        \begin{tikzcd}[column sep=huge]
            \TR^r (R_0;\Zp)^{hS^1} \arrow[r] \arrow[d] & \TC^- (R_0;\Zp) \arrow[r, "\varphi^{hS^1}"] \arrow[d] & \TP (R_0;\Zp) \arrow[d] \\
            \TR^r (R_0;\Zp) \arrow[r] & \THH (R_0;\Zp)^{hC_{p^{r-1}}} \arrow[r, "\varphi^{hC_{p^{r-1}}}"] & \THH (R_0;\Zp)^{tC_{p^r}}
        \end{tikzcd}
    \end{equation*}
    This is true for $r=1$, by \cite{BMS2}. For general $r \geq 1$, the claim follows by induction. One needs to use the fact that the rightmost map $\TP (R_0;\Zp) \to \THH (R_0;\Zp)^{tC_{p^r}}$ always gives rise to $\widetilde{\vartheta}_r : \ainf (R_0) \to \W_r (R_0)$, upon applying $\pi_0$ \cite[Sec. 6]{BMS2}.
    
    It follows that the filtration on $\pi_0 \TR^r (R_0;\Zp)^{hS^1}$ coming from the $S^1$-homotopy fixed points spectral sequence is the $\xi_r$-adic filtration on $\ainf (R_0)$. It propagates to all even homotopy groups of $\TR^r (R_0;\Zp)^{hS^1}$, via its multiplicative structure. The canonical identification follows:
    \begin{equation*}
        \pi_* \TR^r (R_0;\Zp)^{hS^1} \simeq
        \begin{cases}
            \xi_r^i \ainf (R_0) \{ i\}, & *=2i \geq 0 \\[5pt]
            \ainf (R_0) \{ i\}, & *=2i < 0 \\[5pt]
            0, & \text{otherwise}
        \end{cases}
    \end{equation*}

    This identification can also be shown via the iterated pullback description:
    \begin{equation*}
        \TR^r (R_0;\Zp)^{hS^1} \simeq \TC^- (R_0;\Zp) \times_{\TP (R_0;\Zp)} \dots \times_{\TP (R_0;\Zp)} \TC^- (R_0;\Zp)
    \end{equation*}
    which gives rise to the short exact sequences:
    \begin{equation*}
        \begin{tikzcd}
            0 \arrow[r] & \pi_{2i} \TR^r (R_0;\Zp)^{hS^1} \arrow[r] & \displaystyle{\prod_{1 \leq k \leq r} \pi_{2i} \TC^- (R_0;\Zp)} \arrow[r] & \displaystyle{\prod_{1 \leq k \leq r-1} \pi_{2i} \TP (R_0;\Zp)} \arrow[r] & 0
        \end{tikzcd}
    \end{equation*}
\end{proof}

A direct reformulation of this proposition, using the graded ring descriptions from \cite[Sec. 6]{BMS2}:
\begin{gather*}
    \begin{cases}
        \pi_* \TC^- (R_0;\Zp) \simeq \ainf (R_0) [u, v]/(uv-\xi), & \mathrm{deg}(u) = 2,\; \mathrm{deg}(v) =-2,\; \mathrm{deg}(\xi)=0 \\[5pt]
        \pi_* \TP (R_0;\Zp) \simeq \ainf (R_0) [\sigma, \sigma^{-1}], & \mathrm{deg}(\sigma) = 2
    \end{cases}
\end{gather*}
is the following corollary, regarding the multiplicative structure of $\TR^r (R_0;\Zp)^{hS^1}$:
\begin{cor}
    The graded ring associated with $\TR^r (R_0;\Zp)^{hS^1}$ has the following description, for generating elements with $\mathrm{deg}(u_r)=2$, $\mathrm{deg}(v_r)=-2$, and $\mathrm{deg}(\xi_r)=0$:
    \begin{equation*}
        \pi_* \TR^r (R_0;\Zp)^{hS^1} \simeq \ainf (R_0) [u_r, v_r]/ (u_r v_r - \xi_r)
    \end{equation*}
\end{cor}

The following proposition is the main trick that allows us to go from the $S^1$-homotopy fixed points $\TR^r (-;\Zp)^{hS^1}$, back to the original spectrum $\TR^r (-;\Zp)$. It will be used repeatedly throughout the paper, as this is paves a way to go from fairly more accessible calculations via $\TR^r (-;\Zp)^{hS^1}$, back to identifying the structure of $\TR^r (-;\Zp)$.

\begin{prop}[Back to $\TR^r (-;\Zp)$] \label{quotient} Let $A$ be a connective $\einfty$-$R_0$-algebra. Then the following natural map is an equivalence of $\einfty$-ring spectra, for any $1 \leq r < \infty$:
\begin{equation*}
    \begin{tikzcd}[column sep=large]
        \TR^r (A;\Zp)^{hS^1}/v_r \simeq \TR^r (A;\Zp)^{hS^1} \otimes_{\TR^r (R_0;\Zp)^{hS^1}} \TR^r (R_0;\Zp) \arrow[r, "\simeq"] & \TR^r (A;\Zp)
    \end{tikzcd}
\end{equation*}
\end{prop}
\begin{proof}
    The proof is a direct analogue of the one provided in \cite[Prop 6.4]{BMS2}, which is an application of \cite[Lem IV.4.12]{AN}.
\end{proof}

Now, we are able to state our main theorem in the case of perfectoid rings.

\begin{thm}[Main result in the perfectoid case] \label{main perfectoid} Given a fixed perfectoid ring $R_0$, the following are true for $\TR^r (R_0;\Zp)$ and $\TR^r (R_0;\Zp)^{hS^1}$:
\begin{enumerate}[1)]
\item Consider the following commutative diagrams for the Restriction and Frobenius maps of $\TR^r (R_0;\Zp)$ and its $S^1$-homotopy fixed points $\TR^r (R_0;\Zp)^{hS^1}$, for $1 \leq r < \infty$.
\begin{equation*}
    \begin{tikzcd}[column sep=large]
        \TR^{r+1} (R_0;\Zp)^{hS^1} \arrow[r, "\Res^{hS^1}"] \arrow[d] & \TR^r (R_0;\Zp)^{hS^1} \arrow[d] & \TR^{r+1} (R_0;\Zp)^{hS^1} \arrow[r, dashed, "\F^{hS^1}"] \arrow[d] & \TR^r (R_0;\Zp)^{hS^1} \arrow[d] \\
        \TR^{r+1} (R_0;\Zp) \arrow[r, "\Res"] & \TR^r (R_0;\Zp) & \TR^{r+1} (R_0;\Zp) \arrow[r, dashed, "\F"] & \TR^r (R_0;\Zp)
    \end{tikzcd}
\end{equation*}
Passing to homotopy groups, they give rise to the following:
\begin{equation*}
    \begin{tikzcd}[column sep=huge, row sep=huge]
        \displaystyle{\frac{\ainf (R_0) [u_{r+1}, v_{r+1}]}{(u_{r+1} v_{r+1} - \xi_{r+1})}} \arrow[r, "u_{r+1} \mapsto \varphi^{-r}(\xi) u_r \text{, } v_{r+1} \mapsto v_r", "\ainf (R_0) \text{-linear}"'] \arrow[d, "u_{r+1} \mapsto u_{r+1} \text{, } v_{r+1} \mapsto 0", "\vartheta_{r+1} \text{-linear}"'] &[1.7in] \displaystyle{\frac{\ainf (R_0) [u_r, v_r]}{(u_r v_r - \xi_r)}} \arrow[d, "u_r \mapsto u_r \text{, } v_r \mapsto 0", "\vartheta_r - \text{linear}"'] \\
        \W_{r+1} (R_0) [u_{r+1}] \arrow[r, "u_{r+1} \mapsto \varphi^{-r} (\xi) u_r"] & \W_r (R_0) [u_r] \\[-10pt]
        \displaystyle{\frac{\ainf (R_0) [u_{r+1}, v_{r+1}]}{(u_{r+1} v_{r+1} - \xi_{r+1})}} \arrow[r, dashed, "u_{r+1} \mapsto u_r \text{, } v_{r+1} \mapsto \varphi (\xi) v_r", "\varphi \text{-linear}"'] \arrow[d, "u_{r+1} \mapsto u_{r+1} \text{, } v_{r+1} \mapsto 0", "\vartheta_{r+1} \text{-linear}"'] &[1.5in] \displaystyle{\frac{\ainf (R_0) [u_r, v_r]}{(u_r v_r - \xi_r)}} \arrow[d, "u_r \mapsto u_r \text{, } v_r \mapsto 0", "\vartheta_r \text{-linear}"'] \\
        \W_{r+1} (R_0) [u_{r+1}] \arrow[r, dashed, "u_{r+1} \mapsto u_r"] & \W_r (R_0) [u_r]        
    \end{tikzcd}
\end{equation*}
\vspace{0.05in}

\item Mapping further to the Tate fixed points, we obtain the following commutative diagrams:
\begin{equation*}
    \begin{tikzcd}[column sep=large]
        \TR^{r+1} (R_0;\Zp)^{hS^1} \arrow[r] \arrow[d] &[10pt] \TC^- (R_0;\Zp) \arrow[r, "\can"] \arrow[d] &[10pt] \TP (R_0;\Zp) \arrow[d] \\
        \TR^{r+1} (R_0;\Zp) \arrow[r] & \THH (R_0;\Zp)^{hC_{p^r}} \arrow[r, "\can"] & \THH (R_0;\Zp)^{tC_{p^r}} \\
        \TR^{r+1} (R_0;\Zp)^{hS^1} \arrow[r] \arrow[d] &[10pt] \TC^- (R_0;\Zp) \arrow[r, dashed, "\varphi^{hS^1}"] \arrow[d] &[10pt] \TP (R_0;\Zp) \arrow[d] \\
        \TR^{r+1} (R_0;\Zp) \arrow[r] & \THH (R_0;\Zp)^{hC_{p^r}} \arrow[r, dashed, "\varphi^{hC_{p^r}}"] & \THH (R_0;\Zp)^{tC_{p^{r+1}}}
    \end{tikzcd}
\end{equation*}
Passing to homotopy groups, they give rise to the following:
\begin{equation*}
    \begin{tikzcd}[row sep = huge, column sep=huge]
        \displaystyle{\frac{\ainf (R_0) [u_{r+1}, v_{r+1}]}{(u_{r+1} v_{r+1} - \xi_{r+1})}} \arrow[r, "u_{r+1} \mapsto \xi_{r+1} \sigma \text{, } v \mapsto \sigma^{-1} "] \arrow[d,"u_{r+1} \mapsto u_{r+1} \text{, } v_{r+1} \mapsto 0", "\vartheta_{r+1} \text{-linear}"'] &[2in] \ainf (R_0) [\sigma, \sigma^{-1}] \arrow[d, "\sigma \mapsto \sigma", "\widetilde{\vartheta}_{r+1} \text{-linear}"'] \\
        \W_{r+1} (R_0) [u_{r+1}] \arrow[r, "u_{r+1} \mapsto \xi_{r+1}\sigma"] & \W_{r+1} (R_0) [\sigma, \sigma^{-1}] \\
        \displaystyle{\frac{\ainf (R_0) [u_{r+1}, v_{r+1}]}{(u_{r+1} v_{r+1} - \xi_{r+1})}} \arrow[r, dashed, "u_{r+1} \mapsto \sigma \text{, } v_{r+1} \mapsto \widetilde{\xi}_r \sigma^{-1}"] \arrow[d, "u_{r+1} \mapsto u_{r+1} \text{, } v_{r+1} \mapsto 0", "\vartheta_{r+1} \text{-linear}"'] & \ainf (R_0) [\sigma, \sigma^{-1}] \arrow[d, "\sigma \mapsto \sigma", "\widetilde{\vartheta}_{r+1} \text{-linear}"'] \\
        \W_{r+1} (R_0) [u_{r+1}] \arrow[r, dashed, "u_{r+1} \mapsto \sigma"] & \W_{r+1} (R_0) [\sigma, \sigma^{-1}] \\
    \end{tikzcd}
\end{equation*}
\end{enumerate}
\end{thm}

\begin{proof}
    This essentially follows from iteration of the arguments presented in the proof of Proposition \ref{graded ring}, together with the results of Proposition \ref{quotient}. First of all, via applying the latter, we pass from $\TR^r (R_0;\Zp)^{hS^1}$ to $\TR^r (R_0;\Zp)$, by mapping $v_r \mapsto 0$. It follows that the graded ring of $\TR^r (R_0;\Zp)$ is equivalent to:
    \begin{equation*}
        \pi_* \TR^r (R_0;\Zp) \simeq \W_r (R_0) [u_r]
    \end{equation*}

    The results follow after passage to homotopy groups and applying \cite[Sec. 6]{BMS2}, Proposition \ref{perfectoid TR}, as well as from the effect of the diagrams in \ref{Fontaine diagrams}, for the Fontaine style maps $\vartheta_r$, $\widetilde{\vartheta}_r$.

    In particular, on $\pi_0$, the diagrams in (1) correspond to the following diagrams for $\vartheta_r$ and its relation to the Restriction and Frobenius maps:
    \begin{equation*}
        \begin{tikzcd}[column sep=huge]
            \ainf (R_0) \arrow[r, "\id \text{, } \varphi"] \arrow[d, "\vartheta_r"'] & \ainf (R_0) \arrow[d, "\vartheta_r"] \\
            \W_r (R_0) \arrow[r, "\Res \text{, } \F"'] & \W_r (R_0)
        \end{tikzcd}
    \end{equation*}
    On the other hand, passing to the Tate constructions, the diagrams we obtain on $\pi_0$ correspond to passing from $\vartheta_r$ to $\widetilde{\vartheta}_r$:
    \begin{equation*}
        \begin{tikzcd}[column sep=huge]
            \ainf (R_0) \arrow[r, "\varphi^r"] \arrow[d, "\vartheta_r"] & \ainf (R_0) \arrow[d, "\widetilde{\vartheta}_r"] \\
            \W_r (R_0) \arrow[r] & \W_r (R_0)
        \end{tikzcd}
    \end{equation*}
\end{proof}

The following is a reformulation of our results, in the style of motivic filtrations:

\begin{prop}[Motivic filtrations in the perfectoid case] \label{motivic perfectoid} For a perfectoid ring $R_0$, the spectra $\TR^r (R_0;\Zp)^{hS^1} \to \TR^r (R_0;\Zp)$ are equipped with \emph{motivic filtrations}, which in this case are the double-speed Postnikov filtrations.

\begin{enumerate}[1)]
\item For $1 \leq r < \infty$, we define a multiplicative, decreasing, complete filtration on $\ainf (R_0)$, which we call the \emph{$r$-Nygaard filtration}. This is defined as:
\begin{equation*}
    \nr^{\geq i} \ainf (R_0) :=
    \begin{cases}
        \xi_r^i \ainf (R_0), & i \geq 0 \\[5pt]
        \ainf (R_0), & i < 0
    \end{cases}
\end{equation*}
The even homotopy groups (graded pieces for the motivic filtrations) of the spectra $\TR^r (R_0;\Zp)^{hS^1}$ and $\TR^r (R_0;\Zp)$ can be interpreted as follows, using the $r$-Nygaard filtration on $\ainf (R_0)$:
\begin{equation*}
    \begin{cases}
        \pi_* \TR^r (R_0;\Zp)^{hS^1} \simeq \bigoplus_{i \in \Z} \nr^{\geq i} \ainf (R_0) \{ i\} = \bigoplus_{i \in \Z} \xi_r^i \ainf (R_0) \{ i\} \\[5pt]
        \pi_* \TR^r (R_0;\Zp) = \bigoplus_{i \geq 0} \nr^i \ainf (R_0) \{ i\} = \bigoplus_{i \geq 0} \xi_r^i/\xi_r^{i+1} \ainf (R_0) \{ i\}
    \end{cases}
    \end{equation*}
The Restriction and Frobenius maps
\begin{equation*}
    \Res, \F : \TR^{r+1} (R_0;\Zp) \longrightarrow \TR^r (R_0;\Zp)^{hS^1}
\end{equation*}
induce natural maps on $\pi_{2i}$:
\begin{equation*}
    \Res, \F : \n_{r+1}^{\geq i} \ainf (R_0) \{ i\} \longrightarrow \nr^{\geq i} \ainf (R_0) \{ i\}
\end{equation*}
In particular, Restriction induces the natural embedding:
\begin{equation*}
    \Res : \n_{r+1}^{\geq i} \ainf (R_0) \{ i\} \hookrightarrow \nr^{\geq i} \ainf (R_0) \{ i\}
    \end{equation*}
while Frobenius corresponds to the following map:
\begin{gather*}
    \F : \n_{r+1}^{\geq n} \ainf (R_0) \longmapsto \big( \varphi(\xi_{r+1} )\big)^i \ainf (R_0) \{ i\} = \big( \varphi (\xi) \xi_r \big)^i \ainf (R_0) \{ i\} \hookrightarrow \nr^{\geq i} \ainf (R_0) \{ i\}
\end{gather*}
Passing to graded pieces, we also obtain graded versions of these:
\begin{equation*}
    \Res, \F : \n_{r+1}^i \ainf (R_0) \{ i\} \longrightarrow \nr^i \ainf (R_0) \{ i\}
\end{equation*}

\item Further mapping to $\TP (R_0;\Zp)$, the canonical and higher Frobenius maps
\begin{equation*}
    \can, \varphi^{hS^1} : \TR^r (R_0;\Zp) \longrightarrow \TC^- (R_0;\Zp) \longrightarrow \TP (R_0;\Zp)
\end{equation*}
induce natural maps on $\pi_{2i}$:
\begin{equation*}
    \can, \varphi_{r, i} : \nr^{\geq i} \ainf (R_0) \{ i\} \longrightarrow \ainf (R_0) \{ i\}
\end{equation*}
The canonical map induces the natural embedding:
\begin{equation*}
    \can : \nr^{\geq i} \ainf (R_0) \{ i\} \hookrightarrow \ainf (R_0) \{ i\}
\end{equation*}
On the other hand, the higher Frobenius $\varphi^{hS^1}$ induces an \emph{$r$-divided Frobenius}:
\begin{equation*}
    \varphi_{r, i} : \nr^{\geq i} \ainf (R_0) \{ i\} \longrightarrow \ainf (R_0) \{ i\}
\end{equation*}
This is defined away from $\xi_r$
\begin{equation*}
    \varphi_{r, i} : \ainf (R_0) \{ i\} \Bigg[ \frac{1}{\xi_r} \Bigg] \longrightarrow \ainf (R_0) \{ i\} \Bigg[ \frac{1}{\widetilde{\xi}_r} \Bigg]
\end{equation*}
and is related to the $r$-th iterated Frobenius via the formula:
\begin{equation*}
    \varphi^r = \xi_r^i \varphi_{r, i}
\end{equation*}
The $r$-divided Frobenius naturally comes from the following commutative diagram, where the top row gives rise to filtered invariants, while the lower one gives rise to graded ones:
\begin{equation*}
    \begin{tikzcd}[column sep=huge]
        \TR^r (R_0;Zp)^{hS^1} \arrow[r] \arrow[d] & \TC^- (R_0;\Zp) \arrow[r, "\varphi^{hS^1}"] \arrow[d] & \TP (R_0;\Zp) \arrow[d] \\
        \TR^r (R_0;\Zp) \arrow[r] & \THH (R_0;\Zp)^{hC_{p^{r-1}}} \arrow[r, "\varphi^{hC_{p^{r-1}}}"'] & \THH (R_0;\Zp)^{tC_{p^r}}
    \end{tikzcd}
\end{equation*}
\item Taking the limit over restriction maps yields equivalences:
\begin{gather*}
    \begin{cases}
        \tau_{[2i-1,2i]} \TR (R_0;\Zp)^{hS^1} \simeq \rlimr \; \xi_r^i \ainf (R_0) \{ i\} =: \ninf^{\geq i} \ainf (R_0) \{ i\}  \\[5pt]
        \tau_{[2n-1,2n]} \TR (R_0;\Zp) \simeq \rlimr \; \xi_r^i/\xi_r^{i+1} \ainf (R_0) \{ i\} =: \ninf^i \ainf (R_0) \{ i\}
    \end{cases}
\end{gather*}
In particular, we can identify:
\begin{align*}
    \begin{cases} \pi_{2i} \TR (R_0;\Zp)^{hS^1} \simeq \limr \nr^{\geq i} \ainf (R_0) \{ i\} \simeq \limr \xi_r^i \ainf (R_0) \{ i\}\\[5pt]
    \pi_{2i-1} \TR (R_0;\Zp)^{hS^1} \simeq \limr^1 \nr^i \ainf (R_0) \{ i\} \simeq \limr^1 \xi_r^n \ainf (R_0) \{ i\} \\[10pt]
    \pi_{2i} \TR (R_0;\Zp) \simeq \limr \nr^i \ainf (R_0) \{ i\} \\[5pt]
    \pi_{2i-1} \TR (R_0;\Zp) \simeq \limr^1 \nr^i \ainf (R_0) \{ i\}
    \end{cases}
\end{align*}
Notice that the $\rlim^1$ terms, and therefore the odd homotopy groups, vanish in the case $R_0 = \mathcal{O}_K$ is the ring of integers of a spherically complete perfectoid field $K$, thus identifying $\mu = \limr \xi_r$.\\
The Frobenius maps
\begin{equation*}
    \begin{tikzcd}[column sep=huge]
        \TR (R_0;\Zp)^{hS^1} \arrow[r, "\F^{hS^1}"] \arrow[d] & \TR (R_0;\Zp)^{hS^1} \arrow[d] \\
        \TR (R_0;\Zp) \arrow[r, "\F"] & \TR (R_0;\Zp)
    \end{tikzcd}
\end{equation*}
induce natural Frobenius endofunctors:
\begin{equation*}
    \begin{tikzcd}[column sep=huge]
        \ninf^{\geq i} \ainf (R_0) \{ i\} \arrow[r, "\F"] \arrow[d] & \ninf^{\geq i} \ainf (R_0) \arrow[d] \\
        \ninf^i \ainf (R_0) \{ i\} \arrow[r, "\F"] & \ainf (R_0) \{ i\}
    \end{tikzcd}
\end{equation*}
\end{enumerate}
\end{prop}

\begin{proof}
    The statements for finite $1 \leq r < \infty$ are a direct reformulation of Theorem \ref{main perfectoid}. For the case $r = \infty$, taking the limit over restriction maps, we get that the graded pieces for the double-speed Postnikov filtrations give rise to the following two-term complexes for $\TR$ and its $S^1$-homotopy fixed points:
    \begin{equation*}
        \tau_{[2i-1,2i]} \TR (R_0;\Zp)^{hS^1} \longrightarrow \tau_{[2i-1,2i]} \TR (R_0;\Zp)
    \end{equation*}
    Remember that the iterated pullback description of $\TR (R_0;\Zp)^{hS^1}$ is:
    \begin{equation*}
        \TR (R_0;\Zp)^{hS^1} \simeq \dots \times_{\TP (R_0;\Zp)} \TC^- (R_0;\Zp) \times_{\TP (R_0;\Zp)} \dots \times_{\TP (R_0;\Zp)} \TC^- (R_0;\Zp)
    \end{equation*}
    Passing to homotopy groups, we obtain the following exact sequences:
    \begin{equation*}
        \begin{tikzcd}[row sep=small]
            0 \arrow[r] & \pi_{2i} \TR (R_0;\Zp)^{hS^1} \arrow[r] & \displaystyle{\prod \pi_{2i} \TC^- (R_0;\Zp)} \arrow[r] & \, \\
            \, \arrow[r] & \displaystyle{\prod \pi_{2i} \TP (R_0;\Zp)} \arrow[r] & \pi_{2i-1} \TR (R_0;\Zp)^{hS^1} \arrow[r] & 0
        \end{tikzcd}
    \end{equation*}
    This is equivalent to:
    \begin{equation*}
        \begin{tikzcd}
            0 \arrow[r] & \pi_{2i} \TR (R_0;\Zp)^{hS^1} \arrow[r] & \displaystyle{\prod \n^{\geq i} \ainf (R_0) \{ i\} } \arrow[r] & \, \\
            \, \arrow[r] & \displaystyle{\prod \pi_{2i} \ainf (R_0) \{ i\}} \arrow[r] & \pi_{2i-1} \TR (R_0;\Zp)^{hS^1} \arrow[r] & 0
        \end{tikzcd}
    \end{equation*}
    which we identify with a $\lim^1$ sequence. In particular, we have the following equivalences:
    \begin{equation*}
        \begin{cases}
            \tau_{[2i-1,2i]} \TR (R_0;\Zp)^{hS^1} \simeq \rlimr \nr^{\geq i} \ainf (R_0) \{ i\} =: \ninf^{\geq i} \ainf (R_0) \{ i\} \\[5pt]
            \pi_{2i} \TR (R_0;\Zp)^{hS^1} \simeq \limr \nr^{\geq i} \ainf (R_0) \{ i\}, \quad \pi_{2i-1} \TR (R_0;\Zp)^{hS^1} \simeq \limr^1 \nr^{\geq i} \ainf (R_0) \{ i\}
        \end{cases}
    \end{equation*}
    Passing to $\TR^r (R_0;\Zp)$ by taking quotient with $v_r$ and taking the limit with respect to the Restriction maps, we also have that:
    \begin{equation*}
        \begin{cases}
            \tau_{[2i-1,2i]} \TR (R_0;\Zp) \simeq \rlimr \nr^i \ainf (R_0) \{ i\} =: \ninf^i \ainf (R_0) \{ i\} \\[5pt]
            \pi_{2i} \TR (R_0;\Zp) \simeq \limr \nr^i \ainf (R_0) \{ i\}, \quad \pi_{2i-1} \TR (R_0;\Zp) \simeq \limr^1 \nr^i \ainf (R_0) \{ i\}
        \end{cases}
    \end{equation*}
    The remaining statements are a direct consequence of these identifications.
\end{proof}

\newpage

\section{The general case}

In this section we focus on explaining the main theorem of this work. Based on the calculations for perfectoid rings and some general properties of $\TR$, we understand the motivic filtrations in the case of quasiregular-semperfectoid rings, which we later extend via quasisyntomic descent.

\subsection{Calculations for QRSPerfd rings}

In what follows, we study the motivic filtrations of invariants associated with $\TR$ of quasiregular-semiperfectoid rings. Let us fix a quasiregular-semiperfectoid ring $S$. Suppose we also make a choice of a perfectoid base $R_0 \to S$, mostly for simplicity.

For what follows, given any $R_0$-algebra A, we view:
\begin{equation*}
    \begin{cases}
        \pi_* \TR^r (A;\Zp)^{hS^1} & \text{as a graded algebra over } \pi_* \TR^r (R_0;\Zp)^{hS^1} \simeq \ainf (R_0) [u_r, v_r]/ (u_r v_r - \xi_r) \\[5pt]
    \pi_* \TR^r (A;\Zp) & \text{as a graded algebra over }  \pi_* \TR^r (R_0;\Zp) \simeq \W_r (R_0) [u_r]
    \end{cases}
\end{equation*}

\begin{thm}[Motivic filtration in the QRSPerfd case]
    Let $S$ be a quasiregular-semiperfectoid ring over a fixed perfectoid base $R_0 \to S$. The following hold:
    
    \begin{enumerate}[1)]
    \item For $1 \leq r < \infty$, the spectra $\TR^r (S;\Zp)^{hS^1} \to \TR^r (S;\Zp)$ admit functorial, complete and exhaustive, descending, multiplicative $\Z$-indexed (resp. $\N$-indexed) \emph{motivic filtrations}, with:
    \begin{equation*}
        \begin{cases}
            \grM^i \TR^r (S;\Zp)^{hS^1} \simeq \tau_{[2i-1,2i]} \TR^r (S;\Zp)^{hS^1} \\[5pt]
            \grM^i \TR^r (S;\Zp) \simeq \tau_{[2i-1,2i]} \TR^r (S;\Zp)
        \end{cases}
    \end{equation*}
    In particular, the associated spectral sequences calculating $\TR^r (S;\Zp)^{hS^1}$ and $\TR^r (S;\Zp)^{tS^1} \simeq \TP (S;\Zp)$ equip:
    \begin{equation*}
        \Delc_S \simeq \pi_0 \TR^r (S;\Zp)^{hS^1} \simeq \pi_0 \TR^r (S;\Zp)^{tS^1} \simeq \pi_0 \TP (S;\Zp)
    \end{equation*}
    with the same complete, descending $\N$-indexed filtration $\nr^{\geq \bullet} \Delc_S$, which we call the $r$-Nygaard filtration.
    \vspace{0.05in}
    
    \item Via the multiplicative structure of $\pi_* \TR^r (R_0;\Zp)^{hS^1}$, one can identify $\nr^{\geq i} \Delc_S \subset \Delc_S = \pi_0 \TR^r (S;\Zp)^{hS^1}$, with $\pi_{2i} \TR^r (S;\Zp)^{hS^1}$, via multiplication with $v_r^i \in \pi_{-2i} \TR^r (S;\Zp)^{hS^1}$. In particular, we have the following descriptions of the even homotopy groups, setting $\nr^{\geq i} \Delc_S = \Delc_S$, for $i \leq 0$:
    \begin{equation*}
        \pi_{2i} \TR^r (S;\Zp)^{hS^1} \simeq \nr^{\geq i} \Delc_S \{ i\}
    \end{equation*}
    Taking quotient with $v_r \in \pi_{-2} \TR^r (S;\Zp)^{hS^1}$, we pass to $\TR^r (S;\Zp)$, thus obtaining the following identification for its even homotopy groups:
    \begin{equation*}
        \pi_{2i} \TR^r (S;\Zp) \simeq \nr^i \Delc_S \{ i\}
    \end{equation*}
    where, in particular, we have that:
    \begin{equation*}
        \nr^0 \Delc_S \simeq \W_r (S)
    \end{equation*}
    \vspace{0.05in}
    
    \item The Restriction and Frobenius maps
    \begin{equation*}
        \begin{tikzcd}[column sep=huge]
            \TR^{r+1} (S;\Zp)^{hS^1} \arrow[r, "\Res^{hS^1} \text{, } \F^{hS^1}"] \arrow[d] &[30pt] \TR^r (S;\Zp)^{hS^1} \arrow[d] \\
            \TR^{r+1} (S;\Zp) \arrow[r, "\Res \text{, } \F"] & \TR^r (S;\Zp)
        \end{tikzcd}
    \end{equation*}
    induce natural maps on the filtered and graded pieces of the $r$-Nygaard filtered $\Delc_S$:
    \begin{equation*}
        \begin{tikzcd}[column sep=huge]
            \n_{r+1}^{\geq i} \Delc_S \{ i\} \arrow[r, "\Res \text{, } \F"] \arrow[d] & \nr^{\geq i} \Delc_S \{ i\} \arrow[d] \\
            \n_{r+1}^i \arrow[r, "\Res \text{, } \F"] & \nr^i \Delc_S \{ i\}
        \end{tikzcd}
    \end{equation*}
    Mapping further to $\TP (S;\Zp)$:
    \begin{equation*}
        \begin{tikzcd}[column sep=huge]
            \TR^r (S;\Zp)^{hS^1} \arrow[r] & \TC^- (S;\Zp) \arrow[r, "\can \text{, } \varphi^{hS^1}"] &[30pt] \TP (S;\Zp)
        \end{tikzcd}
    \end{equation*}
    we obtain maps on even homotopy groups $\can, \varphi_{r, i} : \pi_{2i} \TR^r (S;\Zp)^{hS^1} \to \TP (S;\Zp)$, which are equivalent to the canonical injection
    \begin{equation*}
        \can : \nr^{\geq i} \Delc_S \{ i\} \hookrightarrow \Delc_S \{ i\}
    \end{equation*}
    and to the \emph{$r$-divided prismatic Frobenius}:
    \begin{equation*}
        \varphi_{r, i} : \nr^{\geq i} \Delc_S \{ i\} \longrightarrow \Delc_S \{ i\}
    \end{equation*}
    which relates to the $r$-th iteration of prismatic Frobenius, which maps the $r$-Nygaard filtration to the $\widetilde{\xi}_r$-adic filtration, via the formula:
    \begin{equation*}
        \varphi^r = \widetilde{\xi}_r^i \varphi_{r, i}
    \end{equation*}
    Using the commutative diagram
    \begin{equation*}
        \begin{tikzcd}[column sep=huge]
            \TR^r (S;\Zp)^{hS^1} \arrow[r] \arrow[d] & \TC^- (S;\Zp) \arrow[r, "\varphi^{hS^1}"] \arrow[d] &[30pt] \TP (S;\Zp) \arrow[d] \\
            \TR^r (S;\Zp) \arrow[r] & \THH (S;\Zp)^{hC_{p^{r-1}}} \arrow[r, "\varphi^{hC_{p^{r-1}}}"] & \THH (S;\Zp)^{tC_{p^r}}
        \end{tikzcd}
    \end{equation*}
    one can identify the effect of $\varphi^{hC_{p^{r-1}}}$, as the graded version of $\varphi^{hS^1}$. In particular, passing to even homotopy groups, we obtain a graded prismatic Frobenius map from the associated graded pieces of the $r$-Nygaard filtration to the \emph{$r$-Hodge--Tate cohomology} $\Del_S^{\HT,r} \simeq \Delc_S / \widetilde{\xi}_r$, obtained from $\THH (S;\Zp)^{tC_{p^r}}$:
    \begin{equation*}
        \gr \varphi_{r, i} : \nr^i \Delc_S \{ i\} \longrightarrow \Del_S^{\HT,r} \{ i\} = \Delc_S/\widetilde{\xi}_r \{ i\}
    \end{equation*}
    \vspace{0.05in}
    
    \item Taking the limit over the Restriction maps yields the following identifications:
    \begin{equation*}
        \begin{cases}
            \pi_{2i} \TR (S;\Zp)^{hS^1} \simeq \limr \nr^{\geq i} \Delc_S \{ i\} \\[5pt]
            \pi_{2i} \TR (S;\Zp) \simeq \limr \nr^i \Delc_S \{ i\}
        \end{cases}
    \end{equation*}
    We have natural Frobenius maps:
    \begin{equation*}
        \begin{tikzcd}[column sep=huge]
            \TR (S;\Zp)^{hS^1} \arrow[r, "\F^{hS^1}"] \arrow[d] &[10pt] \TR (S;\Zp)^{hS^1} \arrow[d] \\
            \TR (S;\Zp) \arrow[r, "\F"] & \TR (S;\Zp)
        \end{tikzcd}
    \end{equation*}
    Passing to even homotopy groups, induces Frobenius endofunctors on filtered and graded pieces:
    \begin{equation*}
        \begin{tikzcd}[column sep=huge]
            \limr \nr^{\geq i} \Delc_S \{ i\} \arrow[r, "\F"] \arrow[d] &[20pt] \limr \nr^{\geq i} \Delc_S \{ i\} \arrow[d] \\
            \limr \nr^i \Delc_S \{ i\} \arrow[r, "\F"] & \limr \nr^i \Delc_S \{ i\}
        \end{tikzcd}
    \end{equation*}
    
    \end{enumerate}
\end{thm}

\begin{proof}
    The basics of all calculations follow from the iterated pullback descriptions:
    \begin{equation*}
        \TR^r (S;\Zp)^{hS^1} \simeq \TC^- (S;\Zp) \times_{\TP (S;\Zp)} \dots \times_{\TP (S;\Zp)} \TC^- (S;\Zp)
    \end{equation*}
    and
    \begin{equation*}
        \TR^r (S;\Zp) \simeq \THH (S;\Zp)^{hC_{p^{r-1}}} \times_{\THH (S;\Zp)^{tC_{p^{r-1}}}} \dots \times_{\THH (S;\Zp)^{tC_p}} \THH (S;\Zp)
    \end{equation*}
    
    These can be re-written as the fibers:
    \begin{equation*}
        \begin{tikzcd}
            \TR^r (S;\Zp)^{hS^1} \arrow[r] & \displaystyle{ \prod_{1 \leq k \leq r} \TC^- (S;\Zp)} \arrow[r] & \displaystyle{ \prod_{1 \leq k \leq r-1} \TP (S;\Zp)}
        \end{tikzcd}
    \end{equation*}
    and
    \begin{equation*}
        \begin{tikzcd}
            \TR^r (S;\Zp) \arrow[r] & \displaystyle{\prod_{1 \leq k \leq r} \THH (S;\Zp)^{hC_{p^k}}} \arrow[r] & \displaystyle{\prod_{1 \leq k \leq r-1} \THH (S;\Zp)^{tC_{p^k}} }
        \end{tikzcd}
    \end{equation*}
    
    In particular, the motivic filtrations come from the double-speed Postnikov filtrations, whose graded pieces are the two term complexes:
    \begin{equation*}
        \tau_{[2i-1, 2i]} \TR^r (S;\Zp)^{hS^1}, \qquad \tau_{[2i-1, 2i]} \TR (S;\Zp)
    \end{equation*}

    More specifically, the homotopy groups fit in exact sequences:
    \begin{equation*}
        \begin{tikzcd}[row sep=tiny]
            0 \arrow[r] & \pi_{2i} \TR^r (S;\Zp)^{hS^1} \arrow[r] & \displaystyle{ \prod_{1 \leq k \leq r} \pi_{2i} \TC^- (S;\Zp)} \arrow[r] & \, \\
            \, \arrow[r] & \displaystyle{\prod_{1 \leq k \leq r-1} \pi_{2i} \TP (S;\Zp)} \arrow[r] & \pi_{2i-1} \TR^r (S;\Zp)^{hS^1} \arrow[r] & 0
        \end{tikzcd}
    \end{equation*}
    or, equivalently
    \begin{equation*}
        \begin{tikzcd}[row sep=tiny]
            0 \arrow[r] & \pi_{2i} \TR^r (S;\Zp)^{hS^1} \arrow[r] & \displaystyle{ \prod_{1 \leq k \leq r} \n^{\geq i} \Delc_S \{ i\} } \arrow[r] & \, \\
            \, \arrow[r] & \displaystyle{\prod_{1 \leq k \leq r-1} \Delc_S \{ i\} } \arrow[r] & \pi_{2i-1} \TR^r (S;\Zp)^{hS^1} \arrow[r] & 0
        \end{tikzcd}
    \end{equation*}
    and
    \begin{equation*}
        \begin{tikzcd}[row sep=tiny]
            0 \arrow[r] & \pi_{2i} \TR^r (S;\Zp) \arrow[r] & \displaystyle{\prod_{1 \leq k \leq r} \pi_{2i} \THH (S;\Zp)^{hC_{p^k}}} \arrow[r] & \, \\
            \, \arrow[r] & \displaystyle{\prod_{1 \leq k \leq r-1} \pi_{2i} \THH (S;\Zp)^{tC_{p^k}}} \arrow[r] & \pi_{2i-1} \TR^r (S;\Zp) \arrow[r] & 0
        \end{tikzcd}
    \end{equation*}

    From the first of these exact sequences it follows that $\pi_{2i} \TR^r (S;\Zp)^{hS^1}$ is identified as the $i$-th layer of the $r$-Nygaard filtration on $\Delc_S \{ i\}$, which is defined via the following iterated pullback diagram:
    \begin{equation*}
        \nr^{\geq i} \Delc_S \{ i\} \simeq \n^{\geq i} \Delc_S \{ i\} \times_{\Delc_S \{ i\}} \dots \times_{\Delc_S \{ i\}} \n^{\geq i} \Delc_S
    \end{equation*}

    If, for simplicity, we assume that to be working over a fixed perfectoid $R_0 \to S$, it follows that we have the equivalent description:
    \begin{equation*}
        \nr^{\geq i} \Delc_S \simeq \Bigg\{ x \in \Delc_S \; \Big| \; \varphi^{ri} (x) \in \widetilde{\xi}_r^i \Delc_S \Bigg\}
    \end{equation*}

    To go back to $\TR^r (S;\Zp)$, we let $v_r \mapsto 0$. Through this, we obtain the identification:
    \begin{equation*}
        \pi_{2i} \TR^r (S;\Zp) \simeq \nr^i \Delc_S \{ i\}
    \end{equation*}

    Finally, the identification regarding $\Delc_S / \xi_r$ is a direct corollary of Proposition \ref{fibre sequences for perfectoids}, in analogy with \cite{BMS2}.

    The rest of the claims are a direct application of induction on $1 \leq r < \infty$, using the iterated pullback descriptions for $\TR^r (S;\Zp)^{hS^1}$ and $\TR^r (S;\Zp)$, together with the base cases on perfectoids:
    \begin{equation*}
        \begin{cases}
            \grM^i \TR^r (R_0;\Zp)^{hS^1} \simeq \nr^{\geq i} \ainf (R_0) \{ i\} \\[5pt]
            \grM^i \TR^r (R_0;\Zp) \simeq \nr^i \ainf (R_0) \{ i\}
        \end{cases}
    \end{equation*}
\end{proof}

\subsection{Applying quasisyntomic descent} In the previous section we managed to provide an overview of the nature of the invariants $\TR^r (S;\Zp)^{hS^1} \to \TR^r (S;\Zp)$, $1 \leq r \leq \infty$, for a given quasiregular-semiperfectoid ring $S$. Following the road paved by \cite{BMS2}, we can now extend to the quasisyntomic (and even animated) case. The main steps are provided in what follows.
\begin{proof}[Proof of Theorem \ref{thm3}]
    For simplicity, we work over a fixed perfectoid ring $R_0 \to S$.
    As the category of quasiregular-semiperfectoid rings provide a basis for the quasisyntomic topology, most of the claims follow in a completely analogous way to the main theorem of \cite{BMS2}. In particular, we apply quasisyntomic descent to the two term complexes:
    \begin{equation*}
        \tau_{[2i-1, 2i]} \TR^r (S;\Zp)^{hS^1}, \qquad \tau_{[2i-1,2i]} \TR^r (S;\Zp)
    \end{equation*}
    for $S$ quasiregular-semiperfectoid and the associated results of the previous subsection.

    Regarding the odd homotopy groups of our invariants, remember that for $S$ quasiregular-semiperfectoid, these fit in the exact sequence:
    \begin{equation*}
        \begin{tikzcd}[row sep=tiny]
            0 \arrow[r] & \nr^{\geq i} \Delc_S \{ i\} \arrow[r] & \displaystyle{\prod_{1 \leq k \leq r} \n^{\geq i} \Delc_S \{ i\}} \arrow[r, "\alpha"] & \, \\
            \, \arrow[r, "\alpha"] & \displaystyle{\prod_{1 \leq k \leq r-1} \Delc_S \{ i\}} \arrow[r] & \pi_{2i-1} \TR^r (S;\Zp)^{hS^1} \arrow[r] & 0
        \end{tikzcd}
    \end{equation*}
    Pre-composing the map $\alpha$, with the diagonal map:
    \begin{equation*}
        \mathrm{diag}: \n^{\geq i} \Delc_S \{ i\} \longrightarrow \prod_{1 \leq k \leq r} \n^{\geq i} \Delc_S \{ i\}
    \end{equation*}
    we are able to use the vanishing theorem of Bhatt-Scholze \cite[Sec. 14]{BS}. In particular, by directly applying that result, there exists a suitable quasisyntomic cover $S' \to S$, for which $\alpha \circ \mathrm{diag}$ is surjective. Hence, the same is true for $\alpha$, as well, from which the local vanishing of the cokernel follows.

    Therefore, locally for the quasisyntomic topology, we have that:
    \begin{equation*}
        \grM^i \TR^r (-;\Zp)^{hS^1} \simeq \nr^{\geq i} \Delc_{(-)} \{ i\} [2i]
    \end{equation*}
    Hence, by letting $v_r \mapsto 0$, we also have the analogous result for $\TR^r$, locally in the quasisyntomic topology:
    \begin{equation*}
        \grM^i \TR^r (-;\Zp) \simeq \nr^i \Delc_{(-)} \{ i\} [2i]
    \end{equation*}
    Taking the limit over restriction maps, in analogy with the perfectoid case, we also obtain:
    \begin{equation*}
        \begin{cases}            
        \grM^i \TR (-;\Zp)^{hS^1} \simeq \rlimr \nr^{\geq i} \Delc_{(-)} \{ i\} [2i] =: \ninf^{\geq i} \Delc_{(-)} \{ i\} [2i] \\[5pt]
        \grM^i \TR (-;\Zp) \simeq \rlimr \nr^i \Delc_{(-)} \{ i\} [2i] =: \ninf^i \Delc_{(-)} \{ i\} [2i]
        \end{cases}
    \end{equation*}

    As in \cite[Sec. 6]{APC}, it is possible to left Kan extend the motivic filtrations, and thus all associated prismatic invariants, from the quasisyntomic to the animated case. However, the motivic filtrations need not be exhaustive in this case.
\end{proof}

Following this, the proof of Theorem \ref{thm4} is a direct consequence of the above:

\begin{proof}[Proof of Theorem \ref{thm4}]
    We consider the invariant
    \begin{equation*}
        \widetilde{\TC}^r (-;\Zp) := \fib \Big( \Res^{hS^1} - \F^{hS^1} : \TR^r (-;\Zp)^{hS^1} \longrightarrow \TR^{r-1} (-;\Zp)^{hS^1} \Big)
    \end{equation*}
    Since, locally for the quasisyntomic topology we have that:
    \begin{equation*}
        \grM^i \TR^r (-;\Zp)^{hS^1} \simeq \nr^{\geq i} \Delc_{(-)} \{ i\} [2i]
    \end{equation*}
    it follows that by taking fibers, $\widetilde{\TC}^r$ is equipped with a motivic filtration, whose description, locally for the quasisyntomic topology, is the following:
    \begin{equation*}
        \grM^i \widetilde{\TC}^r (-;\Zp) \simeq \Big( \Res^{hS^1} - \F^{hS^1} : \nr^{\geq i} \Delc_{(-)} \{ i\} [2i] \longrightarrow \n_{r-1} \Delc_{(-)} \{ i\} [2i] \Big)
    \end{equation*}
    We proceed similarly for $\TC^r$.
\end{proof}

\newpage

\section{The $r$-Nygaard filtered absolute prismatic cohomology}

In this section we reconsider the $r$-Nygaard filtration on non-completed prismatic cohomology, from a more algebraic perspective. We also discuss the the $r$-Hodge--Tate cohomology, as well as relative versions of these invariants. In particular, we mainly address the contents of Theorems \ref{thm1} and \ref{thm2}. Finally, we provide some insights regarding ways in which the theory of the de Rham--Witt complex fits in the prismatic setting.

\subsection{The absolute case}

\begin{defn}[The $r$-Nygaard filtration]
    Let $S$ be an animated ring. For $1 \leq r \leq \infty$, we consider the following iterated pullback construction on its absolute prismatic cohomology $\Del_S \{ i\}$, along the canonical inclusion $\iota : \n^{\geq i} \Del_S \{ i\} \hookrightarrow \Del_S \{ i\}$ and the divided prismatic Frobenius map $\varphi_i : \n^{\geq i} \Del_S \{ i\} \to \Del_S \{ i\}$:
    \begin{equation*}
        \nr^{\geq i} \Del_S \{ i\} := \n^{\geq i} \Del_S \{ i\} \times_{\Del_S \{ i\}} \dots \times_{\Del_S \{ i\}} \n^{\geq i} \Del_S \{ i\} 
    \end{equation*}
    This yields a decreasing filtration on absolute prismatic cohomology $\nr^{\geq i} \Del_S \{ i\}$, which we call the \emph{$r$-Nygaard filtration}.
\end{defn}

These invariants come equipped with a number of additional structure. Let us begin by first discussing the case $1 \leq r < \infty$. One observes that there are canonical Restriction, Frobenius, and Verschiebung maps:
\begin{equation*}
    \begin{cases}
        \Res : \n_{r+1}^{\geq i} \Del_S \{ i\} \longrightarrow \nr^{\geq i} \Del_S \{ i\} \\[5pt]
        \F : \n_{r+1}^{\geq i} \Del_S \{ i\} \longrightarrow \nr^{\geq i} \Del_S \{ i\} \\[5pt]
        \V : \nr^{\geq i} \Del_S \{ i\} \longrightarrow \n_{r+1}^{\geq i} \Del_S \{ i\}
    \end{cases}
\end{equation*}
These are obtained by forgetting the leftmost factor, by forgetting the right most factor, and by adding a factor to the left, respectively. In particular, notice that these also pass functorially to the graded pieces, inducing analogous maps, as well.

In analogy with the case $r=1$, there exist two additional maps, with target the non-filtered prismatic cohomology. The first map is the canonical inclusion:
\begin{equation*}
    \begin{tikzcd}[column sep=huge]
        \iota : \nr^{\geq i} \Del_S \{ i\} \arrow[r, hook] & \Del_S \{ i\}
    \end{tikzcd}
\end{equation*}
This is obtained by mapping from the iterated pullback construction, first to the rightmost factor $\n^{\geq i} \Del_S \{ i\}$ via Restriction and then to $\Del_S \{ i\}$ via the canonical inclusion.

The second map is a divided version of the $r$-th iteration of prismatic Frobenius, which we call the \emph{$r$-divided prismatic Frobenius}:
\begin{equation*}
    \begin{tikzcd}[column sep=huge]
        \varphi_{r, i} : \nr^{\geq i} \Del_S \{ i\} \arrow[r] & \Del_S \{ i\}
    \end{tikzcd}
\end{equation*}
This is obtained by mapping from the iterated pullback construction, first to the leftmost factor $\n^{\geq i} \Del_S \{ i\}$ via Frobenius (of this iterated construction) and then to $\Del_S \{ i\}$ via the divided prismatic Frobenius $\varphi_i$.

In order to understand the properties of the $r$-th iterated prismatic Frobenius on absolute prismatic cohomology, of its divided version and how they correspond to the homotopical machinery we used in previous sections, we plan to briefly go through some sheafy constructions on the prismatization stack $\Sigma$.

\begin{defn}[The invertible ideal sheaf $I_r$]
    Remember that, under the equivalence
    \begin{equation*}
        \begin{tikzcd}[column sep=huge]
            \mathcal{D} (\Sigma) \arrow[r, "\simeq"] & \lim_{(A,I)} \widehat{\mathcal{D}} (A)
        \end{tikzcd}
    \end{equation*}
    making constructions on $\mathcal{D} (\Sigma)$ is equivalent to making analogous cnstructions, in a compatible way, over the prism maps. Furthermore, following the discussion surrounding \cite[Prop. 2.4.1]{APC}, we may only perform these constructions on bounded, transversal prisms.

    Let $(A,I)$ be a bounded transversal prism and consider the invertible ideal $I_r \to A$, \cite[Not. 2.2.2]{APC}. It follows that this induces an invertible ideal sheaf $\mathcal{I}_r$ on $\Sigma$.

    Given an animated ring $S$, we denote by $I_r^i \Del_S$ the global sections of the sheaf $\mathcal{I}_r^i \otimes \mathcal{H}_{\Del} (S)$, for $r \geq 0$, $i \in \Z$.
\end{defn}

\begin{con}[Relative $r$-Nygaard filtration]
    Let $S$ be an animated $\overline{A}$-algebra, for a given bounded prism $(A,I)$. Then the relative prismatic cohomology of $S$ is equipped with a relative $r$-Nygaard filtration, which maps to the $I_r$-adic filtration, under the $r$-th iteration of the relative prismatic Frobenius:
    \begin{equation*}
        \varphi_A^r : \nr^{\geq i} (\varphi_A^r)^* \Del_{S/A} \{ i\} \longrightarrow I_r^i \Del_{S/A} \{ i\}
    \end{equation*}

    As this is done in a compatible way, for any bounded prism $(A,I)$, it gives rise to an analogous formula for sheaves on $\Sigma$, for a given animated ring $S$:
    \begin{equation*}
        \nr^{\geq i} (\Phi^r)^* \mathcal{H}_{\Del} (S) \{ i\} \longrightarrow \mathcal{I}_r^i \otimes \mathcal{H}_{\Del} (S) \{ i\}
    \end{equation*}
\end{con}

\begin{con}[$r$-th iteration of the absolute prismatic Frobenius]
    Globalizing the aforementioned formula, as in \cite[Sec. 5.7]{APC}, we obtain a filtered version of the $r$-th iteration of the prismatic Frobenius:
    \begin{equation*}
        \Fil^i (\varphi^r) : \nr^{\geq i} \Del_S \{ i\} \longrightarrow I_r^i \Del_S \{ i\}
    \end{equation*}
    Moreover, as in \cite[Rem 5.7.8]{APC}, one can enhance the $r$-divided prismatic Frobenius to a filtered version:
    \begin{equation*}
        \Fil^{\bullet} (\varphi_{r,i}) : \nr^{\bullet} \Del_S \{ i\} \longrightarrow I_r^{\bullet - i} \Del_S \{ i\}
    \end{equation*}
\end{con}

One is able provide an explicit description of the $r$-Nygaard filtration in the case of quasiregular-semiperfectoid rings:
\begin{prop}[The $r$-Nygaard filtration for quasiregular-semiperfectoid rings] \label{r-Nygaard filtration for quasiregular-semiperfectoid}
    Let $S$ be a quasiregular-semiperfectoid ring. Then its prismatic cohomology $\Del_S$ is a discrete commutative ring and one has the initial prism $(\Del_S, (d))$, \cite[Ex. 4.4.13]{APC}. Then, the computation of the pullback takes place in the discrete world of modules, and therefore one obtains the following identification:
    \begin{equation*}
        \nr^{\geq i} \Del_S \simeq \Big\{ x \in \Del_S \; \big| \; \varphi^r (x) \in d_r^i \Del_S \Big\} \simeq \big(\varphi^{-r} (d_r)\big)^i \Del_S
    \end{equation*}
\end{prop}

By applying quasisyntomic descent and taking sifted colimits from the finitely generated polynomial algebra case \cite[Rem. 5.7.10]{APC}, one arrives at the following generalization of the previous formula:
\begin{prop}[Recurring identification of the $r$-Nygaard filtration] \label{Recurring identification of the $r$-Nygaard filtration}
    Let $S$ be an animated ring. Then, under the action of the $r$-th iterated prismatic Frobenius, one has the following recursive identification of the $r$-Nygaard filtered, complete prismatic cohomology of $S$:
    \begin{equation*}
        (\varphi^r)_*\, \nr^{\geq i} \Delc_S \simeq \varphi_*\, \n^{\geq i} \Delc_S \otimes_{\Delc_S} (\varphi^2)_*\, \n^{\geq i} \Delc_S \otimes_{\Delc_S} \dots \otimes_{\Delc_S} (\varphi^r)_*\, \n^{\geq i} \Delc_S
    \end{equation*}
\end{prop}

The following statements are a consequence of the latter description for the $r$-Nygaard filtered prismatic cohomology:
\begin{lem}
    Let $S$ be an animated ring. Then, its absolute prismatic cohomology $\Del_S$ is complete with respect to the Nygaard filtration $\n^{\geq i} \Del_S$ if and only if it is complete with respect to the $r$-Nygaard filtration $\nr^{\geq i} \Del_S$.
\end{lem}

Since the functors $S \longmapsto \n^{\geq \bullet} \Delc_S$, $S \longmapsto \n^{\bullet} \Delc_S$ commute with sifted colimits \cite[Rem. 5.5.10, 5.7.10]{APC}, the following result follows:

\begin{prop}
    The functors $S \longmapsto \nr^{\geq \bullet} \Delc_S$, $S \longmapsto \nr^{\bullet} \Delc_S$ commute with sifted colimits.
\end{prop}

The above identification for $\nr^{\geq \bullet} \Del_S$, paired with the Propositions \cite[Prop 5.5.19, Rem. 5.7.10]{APC}, as well as the properties of the Beilinson t-structure/d\'ecalage functor outlined in \cite[Lem. 6.11]{BMS1}, \cite[Cor. 7.10, Rem. 7.11]{BMS2}, yield the following:
\begin{prop}[Connectivity estimates] \label{r-Nygaard Beilinson t-structure}
    Let $S$ be an animated ring. Then, the filtered complex $\nr^{\geq i} \Del_S$ is connective with respect to the Beilinson $t$-structure or, equivalently, the cohomology groups of the complex $\nr^i \Del_S$ are concentrated in degrees $\leq i$.

    If we allow $S$ to be a, finitely generated polynomial algebra over $\Z$, then the filtered complex $\nr^{\geq \bullet} \Delc_S$ gets identified under $\Fil^i (\varphi^r)$ with the connective cover of $I_r^{\bullet} \Del_S$, with respect to the Beilinson t-structure.
\end{prop}

Let us also introduce the $r$-Hodge--Tate cohomology, using some prismatization related constructions.
\begin{defn}[The $r$-Hodge--Tate divisor]
    We denote by $\Sigma^{\HT, r}$ the \emph{$r$-Hodge--Tate divisor}, which is defined to be the closed substack of $\Sigma$, associated with the invertible ideal sheaf $\mathcal{I}_r$. For $r=1$, this naturally identifies with the usual Hodge--Tate divisor $\Sigma^{\HT}$.
\end{defn}

\begin{cor}[Prism maps to the $r$-Hodge--Tate divisor]
    Let $(A,I)$ be a prism. Then there exists a map from the formal scheme associated with $A/I_r$ to the $r$-Hodge--Tate divisor, making the following square Cartesian:
    \begin{equation*}
        \begin{tikzcd}[column sep=huge, row sep=huge]
            \spf A/ I_r \arrow[r, "\rho_A^{\HT, r}"] \arrow[d, hook] & \Sigma^{\HT, r} \arrow[d, hook] \\
            \spf A \arrow[r, "\rho_A"] & \Sigma
        \end{tikzcd}
    \end{equation*}
\end{cor}

Note that a natural corollary of the above, as in the case of the prismatization stack $\Sigma$, is that one is able to relate sheaves on $\Sigma^{\HT,r}$ in terms of sheaves on $\spf A/I_r$. In particular, the following equivalence holds, for all $1 \leq r < \infty$:
\begin{equation*}
        \begin{tikzcd}
            \mathcal{D} (\Sigma^{\HT, r}) \arrow[r, "\simeq"] & \lim_{(A,I)} \widehat{\mathcal{D}} (A/I_r)
        \end{tikzcd}
    \end{equation*}

In general, the prism maps $\rho^{\HT,r}$, the geometry of the $r$-Hodge--Tate divisor and how it relates to the prismatization stack $\Sigma$ is better understood using the related stacks which classify degree $r$ distinguished elements, $\Sigma_r := [ \W_{\mathrm{prim}=r}/ \W^{\times}]$, which we hope to illuminate in future work. The last section contains a related discussion regarding this.

\begin{defn}[The $r$-Hodge--Tate cohomology]
    Let $S$ be an animated ring. In analogy with the usual Hodge--Tate cohomology, one now can use the canonical map $\iota: \Sigma^{\HT,r} \hookrightarrow \Sigma$, to define the $r$-Hodge--Tate cohomology of $S$, as:
    \begin{equation*}
        \mathcal{H}_{\Del^{\HT,r}} (S) \simeq 
        \mathcal{H}_{\Del} (S) \big|_{\Sigma^{\HT}}
    \end{equation*}
\end{defn}

Passing to global sections we have that the $r$-Hodge--Tate cohomology fits in the lower right corner of the commutative square. On the upper row we have the $r$-divided prismatic Frobenius from the $i$-th filtered piece of the $r$-Nygaard filtration, while on the lower row, we have the induced map from the $i$-th associated graded piece:
\begin{equation*}
    \begin{tikzcd}[column sep=huge]
        \nr^{\geq i} \Del_S \{ i\} \arrow[r, "\varphi_{r, i}"] \arrow[d] & \Del_S \{ i\} \arrow[d] \\
        \nr^i \Del_S \{ i\} \arrow[r] & \Del_S^{\HT, r} \{ i\}
    \end{tikzcd}
\end{equation*}

Note that under the viewpoint of \cite[Sec. 13]{BS}, \cite[Sec. 3]{riggenbach2022k} and via descent \cite[Theorem 5.6.2]{APC}, it follows that for $p$-complete animated rings, the algebraic way (via $\Sigma^{\HT, r}$ and the homotopy theoretic way (via $\THH (-;\Zp)^{tC_{p^r}}$) to introduce $\Del^{\HT, r}$ coincide.

The homotopy-theoretic machinery gives us a connection with the Witt vector functor:
\begin{equation*}
    \nr^0 \Del_S \simeq \W_r (S)
\end{equation*}

For $r=\infty$, we have to take the homotopy limit with respect to the Restriction maps, in order to obtain the $\infty$-Nygaard filtration:
\begin{gather*}
    \ninf^{\geq i} \Del_S \{ i\} \simeq \limr \nr^{\geq i} \Del_S \{ i\} \simeq\\
    \dots \times_{\Del_S \{ i\}} \n^{\geq i} \Del_S \{ i\} \times_{\Del_S \{ i\}} \dots \times_{\Del_S \{ i\}} \n^{\geq i} \Del_S \{ i\}
\end{gather*}

In this case we do not have a divided version of the Frobenius nor a square analogous to the one involving the $r$-Hodge--Tate cohomology. However, there exists a Frobenius endomorphism, preserving the $\infty$-Nygaard filtration:
\begin{equation*}
    \begin{tikzcd}[column sep=huge]
        \F : \ninf^{\geq i} \Del_S \{ i\} \arrow[r] & \ninf^{\geq i} \Del_S \{ i\}
    \end{tikzcd}
\end{equation*}
This passes to the associated graded pieces and, thus, it is also respected by the equivalence with the Witt vectors:
\begin{equation*}
    \ninf^0 \Del_S \simeq \W (S)
\end{equation*}

\subsection{The relative case}

There is an analogous picture regarding the relative theory. Let $(A,I)$ be a bounded prism and consider $S$ to be a $p$-complete animated $\overline{A}$-algebra. There exists a filtered relative prismatic Frobenius mapping the $r$-Nygaard filtration to the $I_r$-adic filtration:
\begin{equation*}
    \Fil^i (\varphi_A^r) : \nr^{\geq i} (\varphi^r) \Del_{S/A} \{ i\} \longrightarrow I_r^i \Del_{S/A} \{ i\}
\end{equation*}
Its divided version gives rise to the following commutative square:
\begin{equation*}
    \begin{tikzcd}[column sep=huge]
        \nr^{\geq i} (\varphi_A^r)^* \Del_{S/A} \{ i\} \arrow[r, "\varphi_{r, i}"] \arrow[d] & \Del_{S/A} \{ i\} \arrow[d] \\
        \nr^i (\varphi_A^r)^* \Del_{S/A} \{ i\} \arrow[r] & \Del_{S/A}^{\HT, r} \{ i\}
    \end{tikzcd}
\end{equation*}

Let us go back to the discussion in the introduction regarding possible connections between the $r$-Nygaard filtration and the de Rham--Witt complex. We would like to view the map of the lower row as factoring through the $i$-th filtered piece of a possible conjugate filtration on the $r$-Hodge--Tate cohomology $\Del_{S/A}^{\HT, r} \{ i\}$, whose associated graded term gives rise to a version of the $i$-th de Rham--Witt forms. Such a viewpoint can be made possible, at least over a perfect prism. Working over a general prism is more subtle, as then one would have to make guesses on how to vary this construction, so that an absolute version is possible.

Working towards a possible de Rham--Witt comparison requires building a complex for each stage $r \geq 1$, equipped with a suitable conjugate filtration whose associated graded terms gives rise to such a comparison. In the setting of $A\Omega$-cohomology, as developed in \cite{BMS1}, such a construction is possible. In particular, the authors build an $\F$-$\V$-procomplex, using the d\'ecalage functor $L\eta_{\widetilde{\xi}_r}$, for the element $\xi_r \in \ainf$, which gives rise to a comparison with the $p$-complete, relative de Rham--Witt forms of Langer--Zink \cite{LangerZink}. This argument has been extended in \cite{molokov2020prismatic} to extend such a comparison in the case of relative prismatic cohomology, over a perfect prism. By adapting this idea, we are able to produce a conjugate filtration on the $r$-Hodge--Tate cohomology, enhancing the $r$-divided Frobenius diagram, above.

\begin{con}[Building a complex for each $1 \leq r < \infty$] \label{Building a complex}
    Let us restrict to the case that $S$ is a $p$-completely smooth $\overline{A}$-algebra, for a given bounded prism $(A,I)$. After explaining the constructions, it is possible to extend to the case that $S$ is a $p$-complete $\overline{A}$-algebra, by left Kan extending. Following the ideas of \cite[Paragraphs 3 and 4]{APC}, it is possible to reduce even further to the case that $(A,I)$ is a transversal prism.
    
    The discussion in \cite[Con. 5.2.1]{APC} produces a complex, with differential coming from the Bockstein operator, out of the graded pieces for the Nygaard filtration (and the $I$-adic filtration) on $\Del_{S/A}$. This is essentially the outcome of the formalism of the Beilinson t-structure on filtered derived categories and the identification of its heart with the abelian category of honest cochain complexes. This produces the relative de Rham complex and thus gives rise to the Hodge--Tate and de Rham comparison theorems. The specialization morphism maps the Nygaard filtered relative prismatic cohomology to the Hodge filtered relative de Rham cohomology.

    One is able to verbatim reproduce these ideas, by replacing $I^{\bullet} \Del_{S/A}$ (resp. $\n^{\geq \bullet} \Del_{S/A} \{ \bullet\}$), with $I_r^{\bullet} \Del_{S/A}$ (resp. $\nr^{\geq \bullet} \Del_{S/A} \{ \bullet \}$). This is made possible via the formalism of the Beilinson t-structure in the relative setting, by adapting Propositions \ref{Recurring identification of the $r$-Nygaard filtration} and \ref{r-Nygaard Beilinson t-structure}, together with \cite[5.2.14]{APC}:
    \begin{cor} \label{r-Nygaard filtration via decalage}
        Let $S$ be a $p$-complete animated $\overline{A}$-algebra, for a given bounded prism $(A,I)$. The filtered complex $\nr^{\geq i} \Del_{S/A}$ is connective with respect to the Beilinson t-structure or, equivalently, the cohomology groups of $\nr^i \Del_{S/A}$ are concentrated in degrees $\leq i$.

        Restricting to the case that $S$ is a smooth $\overline{A}$-algebra, we have that the $r$-th iteration of the relative prismatic Frobenius factors isomorphically through the d\'ecalage with respect to $I_r$:
        \begin{equation*}
            \begin{tikzcd}
                (\varphi_A^r)^* \Del_{S/A} \{ i\} \arrow[r, "\simeq"] & L \eta_{I_r} \Del_{S/A} \{ i\} \arrow[r] & \Del_{S/A}
            \end{tikzcd}
        \end{equation*}
        Moreover, the $r$-Nygaard filtered relative prismatic cohomology $\nr^{\geq \bullet} (\varphi_A^r)^* \Del_{S/A} \{ \bullet\}$ can be identified with the connective cover, with respect to the Beilinson t-structure, of the $I_r$-adic filtration $I_r^{\bullet} \Del_{S/A} \{ \bullet\}$, under the action of the $r$-iterated relative prismatic Frobenius.
    \end{cor}

    The presence of the d\'ecalage filtration $L \eta_{I_r} \Del_{S/A} \{ i\}$ allows us to build a complex, for each $r \geq 1$, with differential coming from Bockstein. Putting these together, we have a pro-system of complexes equipped with obvious symmetries of Frobenius $\F$ and Verschiebung $\V$. The procedure is exactly the one that is spelled out in \cite[Sec. 11]{BMS1}, with the object $L \eta_{\widetilde{\xi}_r} (-)$ taking the place of $L \eta_{I_r} (-)$.
\end{con}

The following enhancement of the $r$-divided Frobenius diagram is a direct corollary of the construction:
\begin{cor}[Conjugate filtration]
    Let $S$ be a $p$-complete animated $\overline{A}$, for a given bounded prism $(A,I)$. The $r$-Hodge--Tate cohomology is equipped with an increasing, multiplicative, exhaustive conjugate filtration $\Fil_{\bullet}^{\conj} \Del_S^{\HT,r}$, such that the lower row of the $r$-divided relative prismatic Frobenius commutative square factors through its $i$-th filtered piece:
    \begin{equation*}
        \begin{tikzcd}[column sep=huge]
            \nr^{\geq i} (\varphi^r)^* \Del_{S/A} \{ i\} \arrow[d] \arrow[rr, "\varphi_{r,i}"] & & \Del_{S/A} \{ i\} \arrow[d] \\
            \nr^i (\varphi^r)^* \Del_{S/A} \arrow[rr] \arrow[dr, bend right=10, "\simeq" description] & & \Del_{S/A}^{\HT,r} \{ i\} \\[-10pt]
            & \Fil_i^{\conj} \Del_{S/A}^{\HT, r} \{ i\} \arrow[ur, bend right=10, hook]
        \end{tikzcd}
    \end{equation*}
    In particular, the following isomorphism holds:
    \begin{equation*}
        \begin{tikzcd}[column sep=huge]
            \nr^i (\varphi^r)^* \Del_{S/A} \{ i\} \arrow[r, "\simeq"] & \Fil_i^{\conj} \Del_{S/A}^{\HT, r} \{ i\}
        \end{tikzcd}
    \end{equation*}
    Furthermore, following \cite{APC},if $S$ is a $p$-complete, finitely generated polynomial $\overline{A}$-algebra (or a $p$-completely smooth $\overline{A}$-algebra), the conjugate filtration can be identified with the filtration induced from the Postnikov tower:
    \begin{equation*}
        \Fil_i^{\conj} \Del_{S/A}^{\HT, r} \{ i\} \simeq \tau^{\leq i} \Del_{S/A}^{\HT,r} \{ i\}
    \end{equation*}
\end{cor}

For $r=1$, the Hodge--Tate comparison \cite{BS, APC} ensures that the associated graded terms of the conjugate filtration can be identified with the $i$-th de Rham forms:
\begin{equation*}
    \gr_i^{\conj} \overline{\Del}_{S/A} \{ i\} \simeq \gr_{\hod}^i \dr_{S/ \overline{A}} \simeq L\widehat{\Omega}_{S/ \overline{A}}^i [-i]
\end{equation*}

As we already mentioned, in the case one works over a perfect prism, Construction \ref{Building a complex} gives rise to an $\F$-$\V$-procomplex, which in turn produces gives a desired comparison with de Rham--Witt forms:
\begin{thm}[de Rham--Witt comparison, \cite{BMS1,molokov2020prismatic}]
    Let $S$ be a $p$-completely smooth $\overline{A}$-algebra, for a given perfect prism $(A,I)$. Then the $i$-th associated graded piece of the conjugate filtration can be identified with the continuous version of the relative de Rham--Witt complex of Langer--Zink:
    \begin{equation*}
        \gr_i^{\conj} \overline{\Del}_{S/A} \{ i\} \simeq W_r\Omega_{S/\overline{A}}^{i, \mathrm{cont}}
    \end{equation*}
\end{thm}

In order to be able to generalize such a theorem, away from a perfectoid base, one should have a good candidate for an absolute comparison. This will have to be based on an absolute version of the conjugate filtration:

\begin{con}
    Let $S$ be an animated ring. Then the prism maps to the $r$-Hodge--Tate divisor, together with the previous construction of the relative conjugate filtration, induce a conjugate filtration on the corresponding $r$-Hodge--Tate sheaf associated with $S$:
    \begin{equation*}
        \Fil_{\bullet}^{\conj} \mathcal{H}_{\Del^{\HT, r}} (S)
    \end{equation*}
\end{con}
Understanding the associated graded terms of this object, would hopefully give rise to a suitable comparison with a de Rham--Witt complex. Note that one is able to identify the $0$-th graded piece, using the Witt vector functor:
\begin{equation*}
    \gr_0^{\conj} \mathcal{H}_{\Del^{\HT, r}} (S) \simeq \W_r (S) \otimes \mathcal{O}_{\Sigma^{\HT, r}}
\end{equation*}
In fact, for $r=1$, there exists a complete description of the associated graded terms, as a result of the Hodge--Tate comparison, \cite[]{APC}:
\begin{equation*}
    \gr_i^{\conj} \mathcal{H}_{\overline{\Del}} (S) \simeq L\widehat{\Omega}_S^i \otimes \mathcal{O}_{\Sigma^{\HT}}
\end{equation*}

\newpage

\section{The cases of mixed/positive characteristic}

We now restrict to the cases of mixed and positive characteristic. In this setting, one deals with the $A\Omega$-cohomology of \cite{BMS1} and the associated phenomena. Of particular importance is the discussion following \cite[Rem. 1.20]{BMS1}, regarding the lifts of the Cartier isomorphism to mixed characteristic, which is exactly the one obtained from our viewpoint. This is just a rewriting of these results, from the filtered viewpoint.

\subsection{The mixed characteristic case}
In this section, we discuss the contents of Theorem \ref{thm5}
\begin{prop}
    Let $S$ be a $p$-completely smooth ring over a fixed perfectoid base $R_0$. Then, prismatic cohomology $\Del_S \simeq \Del_{S/R_0}$ is identified with the $A\Omega_S$ cohomology of \cite{BMS1}. Moreover, the $r$-Nygaard filtration $\nr^{\geq \bullet} \Del_S$ is identified with the filtration induced from the d\'ecalage functor $L\eta_{\xi_r}^{\geq \bullet} A\Omega_S$, for the element $\xi_r \in \ainf (R_0)$.
\end{prop}

\begin{proof}
    The first part is explained, for example, in \cite[Prop. 9.10]{BMS2}. Identifying the $r$-Nygaard filtration, with the one induced from the d\'ecalage with respect to $\xi_r \in \ainf (R_0)$, follows from Corollary \ref{r-Nygaard filtration via decalage}. A related discussion can be found in \cite[Cor. 7.10, Rem. 7.11]{BMS2}.
\end{proof}

After this identification, it follows that the lift of the Cartier isomorphism to the mixed characteristic, introduced in \cite{BMS1} is nothing but essentially the lower row of the following commutative diagram:
\begin{equation*}
    \begin{tikzcd}[column sep=huge]
        \nr^{\geq i} A\Omega_S \simeq L \eta_{\xi_r}^{\geq i} A\Omega_S \{ i\} \arrow[d] \arrow[rr, "\varphi_{r,i}"] & & A\Omega_S \{ i\} \arrow[d] \\
        \nr^i A\Omega_S \arrow[rr] \arrow[dr, bend right=15, "\simeq" description] & & \widetilde{W_r\Omega}_S \{ i\} \\
        & \tau^{\leq i} \widetilde{W_r\Omega}_S \{ i\} \arrow[ur, bend right=15, hook] &
    \end{tikzcd}
\end{equation*}

From the homotopy theoretic viewpoint the above diagram arises by passing to the associated graded pieces of the even parts for the motivic filtrations of the invariants in the following diagram:
\begin{equation*}
    \begin{tikzcd}[column sep=huge]
        \TR^r (S;\Zp)^{hS^1} \arrow[r] \arrow[d] & \TC^- (S;\Zp) \arrow[r, "\varphi^{hS^1}"] \arrow[d] & \TP (S;\Zp) \arrow[d] \\
        \TR^r (S;\Zp) \arrow[r] & \THH (S;\Zp)^{hC_{p^{r-1}}} \arrow[r, "\varphi^{hC_{p^{r-1}}}"] & \THH (S;\Zp)^{tC_{p^r}}
    \end{tikzcd}
\end{equation*}

More specifically, one has the following isomorphism, for $1 \leq r < \infty$:
\begin{equation*}
    \begin{tikzcd}
        \grM^{i, \mathrm{even}} \TR^r (S;\Zp) \simeq \nr^i A\Omega_S [2i] \arrow[r, "\simeq"] & \tau^{\leq i} \widetilde{W_r\Omega}_S \{ i\} [2i]
    \end{tikzcd}
\end{equation*}
By taking the limit with respect to the Frobenius maps, we obtain the following identifiation, regarding topological Frobenius homology:
\begin{equation*}
    \grM^{i, \mathrm{even}} \TF (S;\Zp) \simeq \tau^{\leq i} A\Omega_S \{ i\} [2i]
\end{equation*}

\subsection{The positive characteristic case}

Let us, finally, treat the case of positive characteristic and the contents of Theorem \ref{thm6}. Given a quasisyntomic $\Fp$-algebra $S$, we know from \cite{BMS2} that its prismatic cohomology is identified with the Nygaard-completed derived de Rham complex:
\begin{equation*}
    \Delc_S \simeq \widehat{\mathrm{LW}\Omega}_S
\end{equation*}
In particular, if $S$ is a quasiregular-semiperfect $\Fp$-algebra, we have that $\Delc_S \simeq \acrys (S)$. The Nygaard filtration has the following simple description:
\begin{equation*}
    \n^{\geq i} \acrys (S) = \Bigg\{ x \in \acrys (S) \; \Big| \; \varphi^i (x) \in p^i \acrys (S) \Bigg\}
\end{equation*}
Using this, we can show the odd vanishing for the invariants of $\TR$:

\begin{proof}
    Consider the exact sequence coming from the iterated pullback description of $\TR^r (S;\Zp)^{hS^1}$ for a quasiregular - semiperfect $\Fp$-algebra $S$:
    \begin{equation*}
        \begin{tikzcd}[row sep=tiny]
            0 \arrow[r] & \nr^{\geq i} \acrys (S) \{ i\} \arrow[r] & \displaystyle{\prod_{1 \leq k \leq r} \n^{\geq i} \acrys (S) \{ i\}} \arrow[r, "\alpha"] & \, \\
            \, \arrow[r, "\alpha"] & \displaystyle{\prod_{1 \leq k \leq r-1} \acrys (S) \{ i\}} \arrow[r] & \pi_{2i-1} \TR^r (S;\Zp)^{hS^1} \arrow[r] & 0
        \end{tikzcd}
    \end{equation*}
    Following \cite[Sec. 8]{BMS2}, we know that the map
    \begin{equation*}
        \alpha \circ \mathrm{diag} : \n^{\geq i} \acrys (S) \{ i\} \to \prod_{1 \leq k \leq r} \acrys (S) \{ i\}
    \end{equation*}
    is surjective. Therefore the same is also true for $\alpha$ itself. The vanishing of $\pi_{2i-1} \TR^r (S;\Zp)^{hS^1}$ follows. Letting $v_r \mapsto 0$, we can also see that $\TR^r (S;\Zp)$ is also even.

    Taking the limit over restriction maps, since we are in the characteristic $p>0$ case, notice that the $\limr^1$ term vanishes, therefore $\TR (S;\Zp)^{hS^1}$ and $\TR (S;\Zp)$ are also concentrated on even degrees.

    Hence, the discussion we had in the perfectoid case, also applies here. In particular, because of the vanishing of odd homotopy groups, it follows that the $S^1$-homotopy fixed points spectral sequence degenerates and the $r$-Nygaard filtration on $\acrys (S)$ is indeed the filtration coming from the spectral sequence.  

\end{proof}

Applying quasisyntomic descent, the following is a direct corollary of what we just discussed:

\begin{cor}
    Let $S$ be a quasisyntomic algebra over $\Fp$. Then the spectra $\TR^r (S;\Zp)^{hS^1}$ and $\TR^r (S;\Zp)$ are equipped with motivic filtrations, whose graded pieces can be identified with:
    \begin{gather*}
        \begin{cases}            
         \grM^i \TR^r (S;\Zp)^{hS^1} \simeq \nr^{\geq i} \Delc_S \{ i\} [2i] \\[5pt]
         \grM^i \TR^r (S;\Zp) \simeq \nr^i \Delc_S \{ i\} [2i]
        \end{cases}
    \end{gather*}
     
    In an analogous manner, the graded pieces for the motivic filtrations of $\widetilde{\TC}^r \big( \TR (S;\Zp) \big)$ and $\TC^r \big( \TR (S;\Zp) \Big)$ can be shown to be equivalent to:
    \begin{gather*}
        \grM^i \widetilde{\TC}^r \Big( \TR (S;\Zp) \Big) \simeq \fib \Big( \Res^{hS^1} - \F^{hS^1} : \nr^{\geq i} \Delc_S \{ i\} [2i] \to \n_{r-1}^{\geq i} \Delc_S \{ i\} [2i] \Big) \\
         \grM^i \TC^r \Big( \TR (S;\Zp) \Big) \simeq \fib \Big( \can - \varphi^{hS^1} : \nr^{\geq i} \Delc_S \{ i\} [2i] \to \Delc_S \{ i\} [2i] \Big)
    \end{gather*}
\end{cor}

Finally, let $S$ be a smooth $k$-algebra, where $k$ is a perfect field of characteristic $p>0$. From the identification of the associated graded pieces for the even parts of the motivic filtration of $\TR^r$ in terms of the relative de Rham--Witt complex of Langer--Zink and the related discussion on the mixed characteristic version of the Cartier isomorphism, we have the following equivalence:
\begin{equation*}
    \begin{cases}
        \grM^i \TR^r (S;\Zp) \simeq \tau^{\leq i} W_r\Omega_{S/k}^{\bullet} [2i] \\[5pt]
        \grM^i \TR (S;\Zp) \simeq \grM^i \TF (S;\Zp) \simeq \tau^{\leq i} W\Omega_{S/k}^{\bullet} [2i]
    \end{cases}
\end{equation*}

\newpage

\section{Remarks and description of ongoing work}

In the theory of prismatic cohomology there are two sides, which behave as opposite ones for the same phenomenon. Let $X= \spf S$ be a $p$-adic formal scheme. On one hand, we can associate to it certain homotopy theoretic objects, such as its topological Hochschild homology $\THH (S;\Zp)$ and associated invariants obtained from suitably studying the $S^1$-action: $\TC^- (S;\Zp)$, $\TP (S;\Zp)$, $\THH (S;\Zp)^{tC_p}$, $\TC (S;\Zp)$, etc. On the other hand, we can associate to $X$ certain stacks which encode its prismatic cohomology and related structure, such as the prismatization stack $\Sigma_X$ and its refinements $\Sigma_X'$, $\Sigma_X''$, the Hodge--Tate divisor $\Sigma_X^{\HT}$, the diffracted Hodge stack $X^{\slashed{D}}$, etc. In what follows, we briefly discuss some ideas, both in the homotopy-theoretic and stacky perspectives, which build on the work of this article and which we hope to address, in further detail, in future work.

\subsection{The geometry of the $r$-Hodge--Tate divisor}

In the process of trying to understand the $r$-Hodge--Tate cohomology of a given (animated $p$-complete) ring $S$, we introduced the $r$-Hodge--Tate divisor $\Sigma^{\HT, r}$ of the prismatization stack $\Sigma$. One of its main properties is that it is the target of prismatic maps from $\spf A/I_r$:
\begin{equation*}
    \begin{tikzcd}[column sep=huge]
        \spf A/I_r \arrow[r, "\rho_A^{\HT, r}"] \arrow[d, hook] & \Sigma^{\HT, r} \arrow[d, hook] \\
        \spf A \arrow[r, "\rho_A"] & \Sigma
    \end{tikzcd}
\end{equation*}
The role of the prismatic maps becomes more apparent in the use of geometric objects which classify "distinguished elements of degree $r$". We consider such a geometric object to be $\Sigma_r$, which is defined as the quotient $[W_{\mathrm{prim}=r}/\W^{\times}]$. Then we have that $\Sigma_1 = \Sigma$ and, for general $r \geq 1$, there are natural transition maps $\Sigma_{r+1} \to \Sigma_r$ and $\Sigma_r \to \Sigma_{r+1}$. It follows that one can identify the central fibre of $\Sigma_r$ as the $r$-Hodge--Tate divisor:
\begin{equation*}
    \begin{tikzcd}[column sep=huge]
        \Sigma^{\HT, r} \arrow[r] \arrow[d, hook] & \big[ \{ 0\} / \mathbb{G}_m \big] \arrow[d, hook] & \Sigma^{\HT, r} \arrow[r] \arrow[d] & \Sigma^{\HT, r+1} \arrow[d] \\
        \Sigma_r \arrow[r] & \big[ \widehat{\mathbb{A}}^1 / \mathbb{G}_m \big] & \Sigma_r \arrow[r] & \Sigma_{r+1}
    \end{tikzcd}
\end{equation*}

The stack $\Sigma_r$ interacts well with the $r$-th iteration of the prismatic Frobenius $\Phi^r$, which in turn, results to a generalization of the Frobenius square related to $\Sigma$ \cite[Prop. 3.6.6]{APC}. Using the diffracted prism associated with the element $V(1)$, one can obtain an explicit presentation of $\Sigma^{\HT, r}$ in terms of $(\W_r^{\#})^{\times}$, generalizing the description for the Hodge--Tate divisor:
\begin{equation*}
    \Sigma^{\HT} \simeq \spf \Zp \times B \mathbb{G}_m^{\#}
\end{equation*}
culminating to a "Sen operator $\Theta_r$ of degree $r$" and an "$r$-truncated diffracted Hodge--Witt complex":
\begin{equation*}
    \Fil_i^{\conj} \Del_{S/\Zp [[\widetilde{p}]]}^{\HT, r} \simeq \Fil_i^{\conj} W_r\widehat{\Omega}_S^{\slashed{D}}
\end{equation*}
These are expected to fit in a fibre sequence, together with the associated graded terms of the $r$-Nygaard filtration:
\begin{equation*}
    \begin{tikzcd}[column sep=huge]
        \nr^i \Del_S \{ n\} \arrow[r] & \Fil_i^{\conj} W_r\widehat{\Omega}_S^{\slashed{D}} \arrow[r, "\Theta_r +i"] & \Fil_{i-1}^{\conj} W_r\widehat{\Omega}_S^{\slashed{D}}
    \end{tikzcd}
\end{equation*}

Ultimately, these phenomena are expected to pave the way towards a comparison with the absolute de Rham--Witt complex of Hesselholt--Madsen:
\begin{equation*}
    \begin{cases}
        \gr_i^{\conj} \mathcal{H}_{\Del^{\HT,r}} (S) \{ i\} \simeq LW_r\widehat{\Omega}_S^i \otimes \mathcal{O}_{\Sigma^{\HT,r}} [-i] \\[5pt]
        \gr_i^{\conj} \Del_{S/ \Zp [[\widetilde{p}]]}^{\HT,r} \{ i\} \simeq W_r \widehat{\Omega}_S^{\slashed{D}} [-i]
    \end{cases}
\end{equation*}
Such an expectation is also backed by Hesselholt's comparison between the Witt complex $\pi_* \TR^r (S)$ and the absolute de Rham--Witt complex $W_r\Omega_S^i$, for degrees $\leq 1$, \cite{hesselholttopological}. Moreover, the above phenomena should also extend to a comparison in the integral setting, involving the global prismatic complexes of Bhatt--Lurie and the big de Rham--Witt complex of Hesselholt--Madsen.

Going back to the diagram discussed in the introduction, the missing slots are occupied by the $r$-truncated Hodge--Witt complex, equipped with its conjugate filtration. The expectation is that the de Rham specialization lifts to a comparison from the $r$-Nygaard filtered prismatic cohomology to the Hodge-filtered $r$-truncated de Rham--Witt complex:
\begin{equation*}
    \begin{tikzcd}[column sep=huge, row sep=huge]
        \Fil_{\hod}^i LW_r\widehat{\Omega}_S \arrow[d] & \nr^{\geq i} \Del_S \{ i\} \arrow[l] \arrow[r, "\varphi_{r,i}"] \arrow[d] & \Del_S \{ i\} \arrow[d] \\
        \gr_{\hod}^i LW_r\widehat{\Omega}_S \arrow[d] & \nr^i \Del_S \{ i\} \arrow[l] \arrow[r] \arrow[d] & \Del_S^{\HT, r} \{ i\} \arrow[d] \\
        LW_r\widehat{\Omega}_S^i [-i] & \Fil_i^{\conj} W_r\widehat{\Omega}_S^{\slashed{D}} \arrow[l] \arrow[r] & W_r\widehat{\Omega}_S^{\slashed{D}}
    \end{tikzcd}
\end{equation*}

\subsection{On the prismatization stacks $\Sigma_r'$}

As we already noted in the introduction, one of the main ideas explained in Scholze's ICM address \cite{ScholzeICM} was that $p$-adic cohomology theories should have certain similarities to shtukas. Equivalently, a slightly different way of saying this is that there should exist an automorphic viewpoint for cohomology theories. Of course, such an aim aligns with the basic principles of the Langlands program.

Strong evidence for such a claim can be found in Fargues' theorem, which explains the equivalence between Breuil--Kisin--Fargues modules and certain shtukas with one leg, that arise in the world of $p$-adic geometry \cite[Sec. 14]{ScholzeWeinstein}. Prismatic theory can then be viewed as an attempt to geometrize and generalize the theory of Breuil--Kisin--Fargues modules. Therefore, it would be natural to ask how certain automorphic objects, such as shtukas with $r$-legs in the sense of $p$-adic geometry, would be related to the prismatic perspective.

We believe that such information is captured by the $r$-Nygaard filtration in prismatic cohomology. This should have a stacky counterpart. Indeed, there exist prismatization stacks $\Sigma_r'$, which encode the relevant data, for every $r \geq 1$. These come together with morphisms analogous to the $r$-divided prismatic Frobenius and inclusion maps:
\begin{equation*}
    r \text{-divided } \mathrm{Frob} : \Sigma \longrightarrow \Sigma_r' \qquad \mathrm{incl}: \Sigma \longrightarrow \Sigma_r'
\end{equation*}
Passing to the associated stable $\infty$-categories of quasi-coherent sheaves, one obtains a family of correspondences, which can be viewed as Hecke correspondences in the prismatic setting:
\begin{equation*}
    \begin{tikzcd}[column sep=huge, row sep=tiny]
        & \mathrm{QCoh} ( \Sigma_r') \arrow[dl, "\stackrel{\leftarrow}{h}"'] \arrow[dr, "\stackrel{\rightarrow}{h}"] & \\
        \mathrm{QCoh} (\Sigma) & & \mathrm{QCoh} (\Sigma)
    \end{tikzcd}
\end{equation*}
where the map on the left (resp. right) corresponds to the $r$-divided Frobenius $\varphi_{r,\bullet}$ (resp. to the canonical map). Taking the equalizer for the maps of $\infty$-categories of quasi-coherent sheaves (coequalizer on the level of stacks) should produce prismatic counterparts of shtukas with $r$-number of legs. These prismatic constructions should be thought of as analogous to the local Hecke stack and the Beilinson--Drinfeld affine Grassmannian for degree $r$ Cartier divisors discussed in \cite[Sec. VI.1]{FarguesScholze} and generalizing the constructions of Zhu in positive characteristic \cite{zhu2017affine, Zhu}.

The properties of geometric objects often encountered in the realm of the geometric Langlands program, such as the (Beilinson--Drinfeld) affine Grassmannian and the Hecke correspondences, are highlighted when viewed under the lens of factorization geometry. From the homotopy theoretic perspective, we believe this is intricately related to factorization properties related to $\TR^r$. Remember that topological Hochschild homology of a connective $\einfty$-algebra $A$ is obtained as its factorization homology over the circle:
\begin{equation*}
    \THH (A) \simeq \int_{S^1} A
\end{equation*}
In an analogous way, we expect to obtain the $p$-typical $\TR^r$, by constructing a modification of the Ran space related to $C_{p^{\infty}} \subset S^1$. Repeating over the different primes should also produce an integral version.

Furthermore, given a ring $A$, we hope to construct a spectral version of the affine Grassmannian, using tools from spectral algebraic geometry, which captures information regarding $\TR^r$. In analogy, passing to the local Hecke stack, should relate to $(\TR^r)^{hS^1}$:
\begin{equation*}
    \begin{tikzcd}[column sep=huge]
        \TR^r \arrow[r, no head, dotted, leftrightarrow] &[-10pt] \mathrm{Gr}_{\mathrm{GL}_{\infty}}^{(r)} = \Big[ L^{(r)}/L^{(r), +} \Big], &[-30pt]
        (\TR^r)^{hS^1} \arrow[r, no head, dotted, leftrightarrow] &[-10pt] \mathrm{Hecke}^{(r)}= \Big[ L^{(r),+}\textbackslash L^{(r)}/L^{(r), +} \Big]
    \end{tikzcd}
\end{equation*}
Note that a related construction of topological versions of toy shtukas and excursion operations has been outlined in \cite[Sec. 5]{ToyShtuka}.

The expectation is that studying the geometry of these objects and certain $\infty$-categories of sheaves on them, should yield motivic information related to $\TR^r$ and of its $S^1$-homotopy fixed points. The relation to algebraic $\K$-theory, and in general the theory of motives, comes from the passage to cyclic $\K$-theory/ $\K$-theory of endomorphisms and then to $\TR^r$. A recent viewpoint on the theory of the trace map is outlined in \cite{harpaz2024trace} and also relates to the story of curves in $\K$-theory via the study of the augmentation map $\Sigma_+^{\infty} =: \Sph [t] \to \Sph$, as pointed out in \cite{arone2020goodwillie}. Moreover, note that the relation of $\TC$ to Langlands flavoured phenomena is particularly highlighted in the work of Clausen \cite{clausen2017k}, in which a universal construction for the Artin reciprocity map is provided, using Selmer $\K$-theory.

\printbibliography

@article{riggenbach2022k,
  title={$ K $-Theory of Truncated Polynomials},
  author={Riggenbach, Noah},
  journal={arXiv:2211.11110},
  year={2022}
}

@incollection{arone2020goodwillie,
  title={Goodwillie calculus},
  author={Arone, Gregory and Ching, Michael},
  booktitle={Handbook of homotopy theory},
  pages={1--38},
  year={2020},
  publisher={Chapman and Hall/CRC}
}

@article{harpaz2024trace,
  title={Trace methods for stable categories I: The linear approximation of algebraic K-theory},
  author={Harpaz, Yonatan and Nikolaus, Thomas and Saunier, Victor},
  journal={arXiv:2411.04743},
  year={2024}
}

@phdthesis{bouis2024motivic,
  title={On the motivic cohomology of mixed characteristic schemes},
  author={Bouis, Tess},
  year={2024},
  school={Universit{\'e} Paris-Saclay}
}

@article{betley2005cyclotomic,
  title={The cyclotomic trace and curves on $\K$-theory},
  author={Betley, Stanislaw and Schlichtkrull, Christian},
  journal={Topology},
  volume={44},
  number={4},
  pages={845--874},
  year={2005},
  publisher={Elsevier}
}

@article{nikolaustopological,
  title={Topological Hochschild homology and cyclic $\K$-theory},
  author={Nikolaus, Thomas},
  journal={Author's website},
  year={2018}
}

@article{blumberg2016,
  title={$\K$-theory of endomorphisms via noncommutative motives},
  author={Blumberg, Andrew and Gepner, David and Tabuada, Gon{\c{c}}alo},
  journal={Transactions of the American Mathematical Society},
  volume={368},
  number={2},
  pages={1435--1465},
  year={2016}
}

@inproceedings{hesselholttopological,
  title={Topological Hochschild Homology and the de Rham-Witt Complex for $\Z_{(p)}$-Algebras.},
  author={Hesselholt, Lars},
  booktitle={Homotopy Theory: Relations with Algebraic Geometry, Group Cohomology, and Algebraic $\K$-theory, an International Conference on Algebraic Topology, March 24-28, 2002, Northwestern University},
  volume={346},
  pages={253},
  year={2004},
  organization={American Mathematical Soc.}
}

@phdthesis{andriopoulos2024motivic,
  title={On the Motivic Filtration of TR},
  author={Andriopoulos, Faidon},
  year={2024},
  school={The University of Chicago, \href{https://knowledge.uchicago.edu/record/12395}{https://knowledge.uchicago.edu/record/12395}}
}

@article{molokov2020prismatic,
  title={Prismatic cohomology and de Rham-Witt forms},
  author={Molokov, Semen},
  journal={arXiv:2008.04956},
  year={2020}
}

@article{elmanto2023motivic,
  title={Motivic cohomology of equicharacteristic schemes},
  author={Elmanto, Elden and Morrow, Matthew},
  journal={arXiv:2309.08463},
  year={2023}
}

@article{AMMN,
  title={On the Beilinson fiber square},
  author={Antieau, Benjamin and Mathew, Akhil and Morrow, Matthew and Nikolaus, Thomas},
  journal={Duke Mathematical Journal},
  volume={171},
  number={18},
  pages={3707--3806},
  year={2022},
  publisher={Duke University Press}
}

@article{AntieauDerived,
  title={Periodic cyclic homology and derived de Rham cohomology},
  author={Antieau, Benjamin},
  journal={Annals of K-theory},
  volume={4},
  number={3},
  pages={505--519},
  year={2019},
  publisher={Mathematical Sciences Publishers}
}

@article{AN,
  title={Cartier modules and cyclotomic spectra},
  author={Antieau, Benjamin and Nikolaus, Thomas},
  journal={Journal of the American Mathematical Society},
  volume={34},
  number={1},
  pages={1--78},
  year={2021}
}

@article{BMS1,
  title={Integral $p$-adic Hodge theory},
  author={Bhatt, Bhargav and Morrow, Matthew and Scholze, Peter},
  journal={Publications math{\'e}matiques de l'IH{\'E}S},
  volume={128},
  number={1},
  pages={219--397},
  year={2018},
  publisher={Springer}
}

@article{BMS2,
  title={Topological Hochschild homology and integral $p$-adic Hodge theory},
  author={Bhatt, Bhargav and Morrow, Matthew and Scholze, Peter},
  journal={Publications math{\'e}matiques de l'IH{\'E}S},
  volume={129},
  number={1},
  pages={199--310},
  year={2019},
  publisher={Springer}
}

@article{BS,
  title={Prisms and prismatic cohomology},
  author={Bhatt, Bhargav and Scholze, Peter},
  journal={Annals of Mathematics},
  volume={196},
  number={3},
  pages={1135--1275},
  year={2022},
  publisher={Department of Mathematics, Princeton University Princeton, New Jersey, USA}
}

@article{APC,
  title={Absolute prismatic cohomology},
  author={Bhatt, Bhargav and Lurie, Jacob},
  journal={arXiv:2201.06120},
  year={2022}
}

@article{APC2,
  title={The prismatization of $p$-adic formal schemes},
  author={Bhatt, Bhargav and Lurie, Jacob},
  journal={arXiv:2201.06124},
  year={2022}
}

@article{APC3,
  title={Prismatic $\F$-gauges},
  author={Bhatt, Bhargav},
  journal={Lecture notes available at https://www. math. ias. edu/\~{} bhatt/teaching/mat549f22/lectures. pdf},
  year={2022}
}

@article{clausen2017k,
  title={A K-theoretic approach to Artin maps},
  author={Clausen, Dustin},
  journal={arXiv:1703.07842},
  year={2017}
}

@article{CMM,
  title={K-theory and topological cyclic homology of henselian pairs},
  author={Clausen, Dustin and Mathew, Akhil and Morrow, Matthew},
  journal={Journal of the American Mathematical Society},
  volume={34},
  number={2},
  pages={411--473},
  year={2021}
}

@article{darrell2023mathrm,
  title={$\TR$ of quasiregular semiperfect rings is even},
  author={Darrell, Micah and Riggenbach, Noah},
  journal={arXiv:2308.13008},
  year={2023}
}

@article{devalapurkar2023p,
  title={$p$-typical curves on $p$-adic Tate twists and de Rham-Witt forms},
  author={Devalapurkar, Sanath K and Mondal, Shubhodip},
  journal={arXiv:2309.16623},
  year={2023}
}

@article{Prismatization,
  title={Prismatization},
  author={Drinfeld, Vladimir},
  journal={arXiv:2005.04746},
  year={2020}
}

@article{FarguesScholze,
  title={Geometrization of the local Langlands correspondence},
  author={Fargues, Laurent and Scholze, Peter},
  journal={arXiv:2102.13459},
  year={2021}
}

@article{GR,
  title={Crystals and D-modules},
  author={Gaitsgory, Dennis and Rozenblyum, Nick},
  journal={arXiv:1111.2087},
  year={2011}
}

@book{GL,
  title={Weil's Conjecture for Function Fields: Volume I (AMS-199)},
  author={Gaitsgory, Dennis and Lurie, Jacob},
  volume={199},
  year={2019},
  publisher={Princeton University Press}
}

@article{ToyShtuka,
  title={A toy model for the Drinfeld--Lafforgue shtuka construction},
  author={Gaitsgory, Dennis and Kazhdan, David and Rozenblyum, Nick and Varshavsky, Yakov},
  journal={Indagationes Mathematicae},
  volume={33},
  number={1},
  pages={39--189},
  year={2022},
  publisher={Elsevier}
}

@article{hesselholt2015big,
  title={The big de Rham--Witt complex},
  author={Hesselholt, Lars},
  journal={Acta Mathematica},
  volume={214},
  number={1},
  pages={135--207},
  year={2015},
  publisher={Springer}
}

@misc{hesselholt2005absolute,
  title={The absolute de Rham-Witt complex},
  author={Hesselholt, Lars},
  year={2005},
  publisher={MIT, Cambridge, Massachusetts, USA}
}

@inproceedings{hesselholt2004rham,
  title={On the de Rham-Witt complex in mixed characteristic},
  author={Hesselholt, Lars and Madsen, Ib},
  booktitle={Annales scientifiques de l'Ecole normale sup{\'e}rieure},
  volume={37},
  number={1},
  pages={1--43},
  year={2004}
}

@article{hesselholt2003k,
  title={On the $\K$-theory of local fields},
  author={Hesselholt, Lars and Madsen, Ib},
  journal={Annals of mathematics},
  volume={158},
  number={1},
  pages={1--113},
  year={2003},
  publisher={JSTOR}
}

@article{hesselholt1997k,
  title={On the K-theory of finite algebras over Witt vectors of perfect fields},
  author={Hesselholt, Lars and Madsen, Ib},
  journal={Topology},
  volume={36},
  number={1},
  pages={29--101},
  year={1997},
  publisher={Pergamon}
}

@article{hesselholt1996p,
  title={On the p-typical curves in Quillen's K-theory},
  author={Hesselholt, Lars},
  journal={Acta Mathematica},
  volume={177},
  pages={1--53},
  year={1996},
  publisher={Springer}
}

@article{hesselholt2019arbeitsgemeinschaft,
  title={Arbeitsgemeinschaft: Topological cyclic homology},
  author={Hesselholt, Lars and Scholze, Peter},
  journal={Oberwolfach Reports},
  volume={15},
  number={2},
  pages={805--940},
  year={2019}
}

@article{krause2018lectures,
  title={Lectures on topological Hochschild homology and cyclotomic spectra},
  author={Krause, Achim and Nikolaus, Thomas},
  journal={Author's website},
  year={2018}
}

@article{krause2023polygonic,
  title={Polygonic spectra and TR with coefficients},
  author={Krause, Achim and McCandless, Jonas and Nikolaus, Thomas},
  journal={arXiv:2302.07686},
  year={2023}
}

@article{LangerZink,
  title={De Rham--Witt cohomology for a proper and smooth morphism},
  author={Langer, Andreas and Zink, Thomas},
  journal={Journal of the Institute of Mathematics of Jussieu},
  volume={3},
  number={2},
  pages={231--314},
  year={2004},
  publisher={Cambridge University Press}
}

@article{lindenstrauss2012taylor,
  title={On the Taylor tower of relative K--theory},
  author={Lindenstrauss, Ayelet and McCarthy, Randy},
  journal={Geometry \& Topology},
  volume={16},
  number={2},
  pages={685--750},
  year={2012},
  publisher={Mathematical Sciences Publishers}
}

@book{HTT,
  title={Higher topos theory},
  author={Lurie, Jacob},
  year={2009},
  publisher={Princeton University Press}
}

@misc{HA,
  title={Higher algebra},
  author={Lurie, Jacob},
  year={2017}
}

@article{MK1,
  title={On $\K (1)$-local $\TR$},
  author={Mathew, Akhil},
  journal={Compositio Mathematica},
  volume={157},
  number={5},
  year={2021}
}

@article{mccandless2021curves,
  title={On curves in K-theory and TR},
  author={McCandless, Jonas},
  journal={arXiv: 2102.08281},
  year={2021}
}

@article{moulinos2022universal,
  title={A universal Hochschild--Kostant--Rosenberg theorem},
  author={Moulinos, Tasos and Robalo, Marco and To{\"e}n, Bertrand},
  journal={Geometry \& Topology},
  volume={26},
  number={2},
  pages={777--874},
  year={2022},
  publisher={Mathematical Sciences Publishers}
}

@article{NS,
  title={On topological cyclic homology},
  author={Nikolaus, Thomas and Scholze, Peter},
  journal={Acta Mathematica},
  volume={221},
  number={2},
  pages={203--409},
  year={2018},
  publisher={International Press}
}

@article{Nikolaustalk,
  title={Characteristic polynomials and $\TR$ with coefficients, video from a talk},
  author={Nikolaus, Thomas},
  journal={BIRS Workshop on Equivariant Stable Homotopy Theory and $p$-adic Hodge Theory},
  year={2020}
}

@article{Raksit,
  title={Hochschild homology and the derived de Rham complex revisited},
  author={Raksit, Arpon},
  journal={arXiv:2007.02576},
  year={2020}
}

@book{ScholzeWeinstein,
  title={Berkeley Lectures on p-adic Geometry:(AMS-207)},
  author={Scholze, Peter and Weinstein, Jared},
  year={2020},
  publisher={Princeton University Press}
}

@inproceedings{ScholzeICM,
  title={p-adic Geometry},
  author={Scholze, Peter},
  booktitle={Proceedings of the International Congress of Mathematicians: Rio de Janeiro 2018},
  pages={899--933},
  year={2018},
  organization={World Scientific}
}

@article{Zhu,
  title={An introduction to affine Grassmannians and the geometric Satake equivalence},
  author={Zhu, Xinwen},
  journal={arXiv:1603.05593},
  year={2016}
}

@article{zhu2017affine,
  title={Affine Grassmannians and the geometric Satake in mixed characteristic},
  author={Zhu, Xinwen},
  journal={Annals of Mathematics},
  volume={185},
  number={2},
  pages={403--492},
  year={2017},
  publisher={Department of Mathematics, Princeton University Princeton, New Jersey, USA}
}
\end{document}